\documentclass[10pt,reqno]{amsart}
\usepackage{amsmath,amssymb,cite,mathrsfs}
\usepackage[left=1.9cm,right=1.9cm,top=2.1cm,bottom=2.0cm]{geometry}
\usepackage{enumitem,microtype}

\usepackage[hidelinks,
linktocpage,pdfpagelabels, bookmarksnumbered,bookmarksopen]{hyperref}
\usepackage[english]{babel}

\numberwithin{equation}{section}

\newtheorem{thm}{Theorem}[section]
\newtheorem{lem}[thm]{Lemma}
\newtheorem{cor}[thm]{Corollary}
\newtheorem{Prop}[thm]{Proposition}

\newtheorem{Rem}[thm]{Remark}

\newcommand{\N}{\mathbb{N}}
\newcommand{\R}{\mathbb{R}}

\newcommand{\ths}{{\widetilde{H}^s(\Omega)}}
\newcommand{\Ds}{{(-\Delta)^{s}}}
\newcommand{\ds}{{(-\Delta)^{s/2}}}
\newcommand{\twostar}{{2^{*}_{\mu,s}}}
\newcommand{\SH}{{\mathcal S_{H,s}}}

\begin{document}
	\baselineskip=14pt

	\title[Critical fractional Hartree problems]{Critical functions and  asymptotics for fractional Hartree Brezis--Nirenberg problems in low dimensions}
	 \author[G. Vaira]{Giusi Vaira$^{*}$}
     \author[Y. Wu]{Yue Wu}
	\author[M. Yang]{Minbo Yang$^{**}$}

	\address{Giusi Vaira
		 \newline\indent Dipartimento di Matematica, Università degli Studi di Bari Aldo Moro
		\newline\indent
		Bari 70125, Italy}
	    \email{\href{mailto:giusi.vaira@uniba.it}{giusi.vaira@uniba.it}}
	\address{Yue Wu
		 \newline\indent School of Mathematical Sciences, Zhejiang Normal University,
		\newline\indent
		Jinhua 321004, People's Republic of China}
	    \email{\href{mailto:yuewu1@zjnu.edu.cn}{yuewu1@zjnu.edu.cn}}

	\address{Minbo Yang
		 \newline\indent School of Mathematical Sciences, Zhejiang Normal University,
		\newline\indent
		Jinhua 321004, People's Republic of China}
	    \email{\href{mailto:mbyang@zjnu.edu.cn}{mbyang@zjnu.edu.cn}}


\thanks{$^*$ Giusi Vaira is partially supported by the INdAM-GNAMPA
project ``Singolarità, nonlinearità e interazioni non locali: nuovi
approcci analitici per equazioni differenziali e sistemi''
(CUP E53C25002010001).
   }
	\thanks{$^{**}$Minbo Yang is partially supported by National Natural Science Foundation of China (12471114) and the Natural Science Foundation of Zhejiang Province (LZ26A010002).}

\begin{abstract}
We study the Brezis--Nirenberg problem for the fractional Laplacian with an upper critical Hartree nonlinearity on a bounded domain $\Omega\subset\R^N$, in the low-dimensional range $2s<N<4s$. For a continuous potential $a$, we prove that the following properties are equivalent: the Robin function of $(-\Delta)^s+a$ is negative somewhere in $\Omega$; the associated Hartree--Sobolev quotient $\SH(a)$ lies strictly below the best constant $\SH$; and $\SH(a)$ is attained. This leads to the notion of a critical potential, for which $\SH(a)=\SH$ and every nontrivial decrease of $a$ lowers the infimum. For a critical potential $a$ and $V\in L^\infty(\Omega)$, we determine the leading term in the asymptotics of $\SH(a+\varepsilon V)$ as $\varepsilon\to0^+$ and describe the blow-up of almost-minimizers, including the concentration speed and the location of the concentration point. Compared with the Sobolev case, the nonlocal nonlinearity produces several new interaction terms; a key ingredient is a proof of non-attainment at the threshold based on the sharp Hardy--Littlewood--Sobolev inequality on the sphere.
\end{abstract}

\maketitle

\medskip

\section{Introduction and main results}

\subsection{Background}
Let $\Omega\subset\R^N$ be a bounded domain and let $a\in C(\overline\Omega)$. The classical Brezis--Nirenberg problem \cite{BrNi} asks whether the quotient
\begin{equation*}
S(a):=\inf_{u\in H^1_0(\Omega)\setminus\{0\}}\frac{\int_\Omega|\nabla u|^2dx+\int_\Omega au^2dx}{\|u\|_{L^{2^*}(\Omega)}^2},
\qquad 2^*=\frac{2N}{N-2},
\end{equation*}
lies strictly below the Sobolev constant $S$. Since $S(a)\leq S$ always holds, the strict inequality is the condition that rules out loss of compactness through concentration, and it yields a minimizer, hence a positive solution of $-\Delta u+au=u^{2^*-1}$ in $\Omega$. Brezis and Nirenberg discovered that the answer depends on the dimension in an essential way. If $N\geq4$, the condition is local: $S(a)<S$ as soon as $a<0$ at some point. If $N=3$, local information on $a$ is not enough; for instance, when $\Omega$ is a ball and $a\equiv-\lambda$ is constant, one has $S(a)<S$ if and only if $\lambda>\lambda_1/4$, where $\lambda_1$ is the first Dirichlet eigenvalue. The reason is that the Aubin--Talenti bubble is not square integrable in dimension three, so that the gain coming from a negative potential is of the same order as the interaction of the bubble with the boundary. This interaction is described by the regular part of the Green function of $-\Delta+a$.

This picture was made precise by Druet \cite{Druet2002}, who proved, as conjectured by Brezis, that for $N=3$ the infimum $S(a)$ is attained if and only if the Robin function of $-\Delta+a$ (the regular part of the Green function evaluated on the diagonal) is negative somewhere in $\Omega$. Following Hebey and Vaugon \cite{HeVa2001}, a potential $a$ is called \emph{critical} if $S(a)=S$ and $S(\widetilde a)<S$ for every $\widetilde a\leq a$ with $\widetilde a\not\equiv a$. Critical potentials sit exactly at the threshold between compactness and concentration, and it is natural to ask how $S(a+\varepsilon V)$ behaves when a critical potential $a$ is perturbed. For $N=3$, this question was answered by Frank, K\"onig and Kova\v{r}\'{\i}k \cite{FKK1}, who computed the leading term of $S-S(a+\varepsilon V)$ as $\varepsilon\to0^+$ and described the concentration of almost-minimizers, including the concentration speed and the location of the concentration points. The analysis was extended to solutions in \cite{FKK3} and to higher dimensions in \cite{FKK2}. The fact that the Robin function governs the location of blow-up points of critical problems goes back to the works of Rey \cite{Rey1989,Rey1990}.

For the fractional Laplacian $\Ds$, $s\in(0,1)$, Brezis--Nirenberg type existence results were obtained by Servadei and Valdinoci \cite{SeVa2013,SeVa2015}. In the fractional setting, the role of dimension three is played by the range $2s<N<4s$: the fractional bubble belongs to $L^2(\R^N)$ if and only if $N>4s$, and for $2s<N<4s$ the energy gain produced by the potential is again of the same order as the boundary interaction. In this low-dimensional range, De Nitti and K\"onig \cite{DeNittiKonig2023} proved the fractional counterparts of the results of Druet and of Frank, K\"onig and Kova\v{r}\'{\i}k: the characterization of $S(a)<S$ through the Robin function of $\Ds+a$, the energy asymptotics for perturbations of critical potentials, and the blow-up profile of almost-minimizers.

A different source of nonlocality comes from the nonlinearity. Problems with a Hartree (or Choquard) term $\bigl(|x|^{-\mu}*|u|^p\bigr)|u|^{p-2}u$ arise in quantum mechanics and in the theory of self-gravitating matter, and the Hardy--Littlewood--Sobolev inequality singles out an upper critical exponent for which the associated energy is invariant under translations and dilations. Brezis--Nirenberg type problems with the upper critical Hartree nonlinearity were first studied by Gao and Yang \cite{GY}. More recently, energy asymptotics and blow-up of concentrating solutions for the critical Hartree problem in the local case and in dimension three were investigated in \cite{GGYZ,WuYangHartree,Gao2026Existence}. For the fractional Laplacian, the blow-up of solutions of slightly subcritical Hartree problems, and of the critical Hartree Brezis--Nirenberg problem in the high-dimensional range $N>4s$, was recently studied in \cite{BorgiaCingolaniYangZhao2026}. The purpose of the present paper is to carry out the program of \cite{Druet2002,FKK1,DeNittiKonig2023} for a problem which is nonlocal both in the operator and in the nonlinearity, namely the fractional Hartree Brezis--Nirenberg problem in the low-dimensional range $2s<N<4s$.

\subsection{Setting}

For a given $a\in C(\overline{\Omega})$, we study the variational problem of minimizing the quotient
\begin{equation*}
\SH(a)[u]
:=
\frac{
\displaystyle\int_{\R^N}|\ds u|^2dx
+\int_\Omega a(x)u^2(x)dx
}{
\displaystyle\left(
\int_\Omega\int_\Omega
\frac{|u(x)|^{\twostar}|u(y)|^{\twostar}}
{|x-y|^\mu}
dxdy
\right)^{\frac{1}{\twostar}}
}
\end{equation*}
for $u\in\ths\setminus\{0\}$ and
$
\ths
:=
\left\{
u\in H^s(\R^N):
u\equiv0\ \text{on }\R^N\setminus\Omega
\right\},
$
where  $N\in\N$, $s\in(0,1)$,  $2s<N<4s$, $0<\mu<N$, $\Omega\subset\R^N$ is a bounded  domain with $C^2$ boundary, the exponent
$\twostar=(2N-\mu)/(N-2s)$ is the upper critical Hardy--Littlewood--Sobolev exponent. We recall the Hardy--Littlewood--Sobolev (HLS for short) inequality.
\begin{Prop}[Hardy--Littlewood--Sobolev inequality, see \cite{LL}]\label{HLS1}
Let $t,r>1$ and $0<\mu<N$ with $1/t+\mu/N+1/r=2$, $f\in L^{t}(\mathbb{R}^N)$ and $h\in L^{r}(\mathbb{R}^N)$. There exists a sharp constant $C(t,N,\mu,r)$, independent of $f,h$, such that
\begin{equation*}
\int_{\mathbb{R}^{N}}\int_{\mathbb{R}^{N}}\frac{f(x)h(y)}{|x-y|^{\mu}}dxdy\leq C(t,N,\mu,r) \left\|f\right\|_{t}\left\|h\right\|_{r}.
\end{equation*}
If $t=r=2N/(2N-\mu)$, then
\begin{equation*}
 C(t,N,\mu,r)=C(N,\mu)=\pi^{\frac{\mu}{2}}\frac{\Gamma(\frac{N}{2}-\frac{\mu}{2})}{\Gamma(N-\frac{\mu}{2})}\left\{\frac{\Gamma(\frac{N}{2})}{\Gamma(N)}\right\}^{-1+\frac{\mu}{N}}.
\end{equation*}
\end{Prop}
\noindent
By Proposition~\ref{HLS1} and the fractional Sobolev embedding, the functional
\begin{equation*}
\int_{\mathbb{R}^{N}}\int_{\mathbb{R}^{N}}\frac{\left|u(x)\right|^p\left|v(y)\right|^p}{|x-y|^{\mu}}dxdy
\end{equation*}
is well defined on $H^s(\R^N) \times H^s(\R^N)$ for $\frac{2N-\mu}{N}\leq p \leq \frac{2N-\mu}{N-2s}$. Combining Proposition~\ref{HLS1} with the fractional Sobolev inequality
\begin{equation}
    \label{sob_ineq}
            \int_{\R^N} |\ds u|^2 dx \geq \mathcal S \| u\|_{\frac{2N}{N-2s}}^2 ,
\end{equation}
where $\mathcal S:=2^{2s}\pi^s
\frac{\Gamma((N+2s)/2)}{\Gamma((N-2s)/2)}
\left(\frac{\Gamma(N/2)}{\Gamma(N)}\right)^{2s/N}$ is the sharp fractional Sobolev constant, we are led to the following nonlocal Sobolev inequality
\begin{equation}
\label{hls_ineq}
\int_{\R^N}|\ds u|^2dx
\geq
\SH
\left(
\int_{\R^N}\int_{\R^N}
\frac{|u(x)|^{\twostar}|u(y)|^{\twostar}}
{|x-y|^\mu}
dxdy
\right)^{\frac{1}{\twostar}}.
\end{equation}
The sharp constant is given by
\begin{equation*}
\SH
=
\mathcal S\,C(N,\mu)^{-\frac{1}{\twostar}}.
\end{equation*}

In this paper, the fractional Laplacian operator $\Ds u$ is defined for any $u \in H^s(\R^N)$ by the Fourier transform
\begin{equation*}
(-\Delta)^s u = \mathcal F^{-1} (|\xi|^{2s} \mathcal F u).
\end{equation*}
Moreover, the fractional Laplacian also admits the singular integral representation (see  \cite{BuVa2016,MR3916700}):
\begin{equation*}
(-\Delta)^s u(x) := C_{N,s} P.V. \int_{\R^N} \frac{u(x) - u(y)}{|x-y|^{N+2s}}  dy,
\end{equation*}
where
\begin{equation*}
C_{N,s} :=   \frac{s 2^{2s} \Gamma(\frac{N+2s}{2})}{\pi^{N/2} \Gamma (1-s)}.
\end{equation*}

For the quotient $\SH(a)[u]$ introduced at the beginning of this section, we denote the corresponding infimum by
\begin{equation*}
\SH (a)  :=  \inf \left \{ \SH (a) [u] \, :u \in \ths,u\neq 0 \right \},
\end{equation*}
and compare it with $\SH$  introduced in \eqref{hls_ineq}.

Our goal is to answer, for the fractional Hartree quotient $\SH(a)[u]$ in the range $2s<N<4s$, the three questions considered in \cite{DeNittiKonig2023} for the fractional Sobolev quotient: the characterization of the strict inequality $\SH(a)<\SH$, the energy asymptotics under perturbations of a critical potential, and the blow-up behavior of almost-minimizers. As in the Sobolev case, the regular part of the Green function of $\Ds+a$ enters the leading energy correction. However, the Hartree structure creates several new difficulties, which we now describe.

First, the denominator of $\SH(a)[u]$ is a double integral. When it is evaluated at a projected bubble corrected by the regular part of the Green function, all powers $\phi_a(z)^k\lambda^{-k(N-2s)}$ with $k\leq K=\lfloor\frac{2s}{N-2s}\rfloor$ contribute at an order not smaller than $\lambda^{-2s}$, and they appear both through single integrals and through mixed double integrals; see Theorem~\ref{lem:de-2.1}. Second, the second variation of the Hartree term contains two genuinely nonlocal quadratic forms, and both of them must be kept when the remainder of a concentrating sequence is estimated. Their control on the orthogonal complement of the tangent space of the manifold of bubbles relies on the nondegeneracy and spectral estimates of Lu, Yang and Zhao \cite{LuYangZhao2025}. Third, the proof that $\SH(a)$ is not attained when $\SH(a)=\SH$ (Proposition~\ref{lem:de-3.1}), which is the starting point of the whole analysis, cannot be reduced to the Sobolev case. Following the strategy of Druet \cite{Druet2002}, we transplant a putative minimizer to the sphere and use the sharp inequalities of Frank and Lieb \cite{FrLi2012} for the Hardy--Littlewood--Sobolev functional on $\mathbb S^N$. Finally, the refined lower bound for almost-minimizers requires a reparametrization of both the concentration scale and the amplitude of the bubble, which produces the factor $\beta_\varepsilon$ in Theorem~\ref{lem:de-1.6}.

We now introduce the bubbles used to describe concentration:
\begin{equation*}
	U_{z, \lambda}(x)=\left(\frac{\lambda}{1+\lambda^2|x-z|^2} \right)^{\frac{N-2s}{2}}, \lambda \in \R^+,z\in \R^N.
\end{equation*}
A direct computation gives
\begin{equation*}
  \Ds U_{z, \lambda}(y)= A_{N,\mu,s} \left(\int_{\R^N}\frac{U_{z, \lambda}^{\twostar}(x)}{\left|x-y\right|^{\mu}}dx\right)U^{\twostar-1}_{z, \lambda}(y),
\end{equation*}
where
\begin{equation*}
A_{N,\mu,s}:
=\frac{2^{2s}
\Gamma\left(\frac{N+2s}{2}\right)
\Gamma\left(N-\frac{\mu}{2}\right)
}{\pi^{N/2}
\Gamma\left(\frac{N-2s}{2}\right)
\Gamma\left(\frac{N-\mu}{2}\right)}.
\end{equation*}

Moreover, by \cite{DeNittiKonig2023}, $U_{z, \lambda}$ satisfies
\begin{equation*}
  \Ds U_{z, \lambda}= a_{N,s} U_{z, \lambda}^{\frac{N+2s}{N-2s}},
\end{equation*}
where
\begin{equation*}
  a_{N,s}:= 2^{2s} \frac{\Gamma(\frac{N+2s}{2})}{\Gamma(\frac{N-2s}{2})}.
\end{equation*}
Combining these identities, we obtain
\begin{equation*}
\int_{\R^N}
    \frac{U_{z,\lambda}^{\twostar}(x)}{|x-y|^\mu}dx
    =
    \frac{a_{N,s}}{A_{N,\mu,s}}
    U_{z,\lambda}^{\frac{\mu}{N-2s}}(y),
    \qquad y\in\R^N.
\end{equation*}

After normalization of the amplitude, the Euler--Lagrange equation associated with \eqref{hls_ineq} is
\begin{equation}\label{equation_s_hls}
	\Ds u(z)=  \left(\int_{\R^N}\frac{\left|u(x)\right|^{\twostar}}{\left|x-z\right|^{\mu}}dx\right)\left|u(z)\right|^{\twostar-2}u(z),
\qquad z\in\R^N.
\end{equation}
Le \cite{Le2019Symmetry} proved that all positive solutions of \eqref{equation_s_hls} are functions of the form
\begin{equation*}
\bar{U}_{z, \lambda}(x)=A_{N,\mu,s}^{\frac{N-2s}{2(N+2s-\mu)}}\left(\frac{\lambda}{1+\lambda^2|x-z|^2} \right)^{\frac{N-2s}{2}}= \kappa_{N,\mu,s} U_{z, \lambda}, \lambda \in \R^+,z\in \R^N,
\end{equation*}
where $\kappa_{N,\mu,s}:=A_{N,\mu,s}^{\frac{N-2s}{2(N+2s-\mu)}}$. Then we have
\begin{equation*}
(-\Delta)^s\bar{U}_{z,\lambda}
=
a_{N,s}
A_{N,\mu,s}^{-\frac{2s}{N+2s-\mu}}
\bar{U}_{z,\lambda}^{\frac{N+2s}{N-2s}}
\end{equation*}
and
\begin{equation}
\label{easybarU}
\int_{\R^N}
\frac{\bar{U}_{z,\lambda}^{\twostar}(x)}
{|x-y|^\mu}dx
=
a_{N,s}
A_{N,\mu,s}^{-\frac{2s}{N+2s-\mu}}
\bar{U}_{z,\lambda}^{\frac{\mu}{N-2s}}(y),
\qquad y\in\R^N.
\end{equation}

\subsection{Main results}
\label{subsec:mainresults}

We shall use the Green function of $\Ds + a$. Assume that $a \in C(\overline{\Omega})$ and that $(-\Delta)^s + a$ is coercive, i.e., $\int_{\R^{N}} |\ds v|^2 dx + \int_\Omega a v^2 dx\geq c \int_{\R^N} |\ds v|^2 dx$ for some $c > 0$. We define $G_a: \Omega \times \R^N \to \R$ as the unique function such that for every fixed $x \in \Omega$
\begin{equation*}
\begin{cases}
(- \Delta_{z})^s G_a(x, z) + a(z) G_a(x, z) = \gamma_{N,s} \delta_x(z) & z \in \Omega,  \\
G_a(x, z) = 0  & z \in \R^N \setminus  \Omega.
\end{cases}
\end{equation*}
Here, we set $\gamma_{N,s} = \frac{2^{2s} \pi^{N/2} \Gamma(s)}{\Gamma(\frac{N-2s}{2})}$, so that $\Ds |z|^{-N+2s} = \gamma_{N,s} \delta_0$ on $\R^N$. With this choice of $\gamma_{N,s}$, we write $G_a$ in terms of its singular part and regular part $H_a(x,z)$, namely, $H_a(x,z)=|x-z|^{-(N-2s)}-G_a(x,z)$. The function $H_a$ is continuous up to the diagonal; see Remark~\ref{rem:continuous-a}. We then define the \emph{Robin function}
\begin{equation*}
\phi_a(x) := H_a(x,x), \qquad  x \in \Omega.
\end{equation*}
The local expansion of $H_a$ used below is recalled in Lemma~\ref{lemma Ha expansion}.

Our first result  is the fractional Hartree counterpart of the criterion of Druet \cite{Druet2002}.

\begin{thm}[Characterization of criticality]
	\label{main}
Let $2s < N < 4s$ and let $a \in C(\overline{\Omega})$ be such that $(-\Delta)^s + a$ is coercive. The following properties are equivalent.
	\begin{enumerate}
		\item[(i)] There is $x \in \Omega$ such that $\phi_a(x) < 0$.
		\item[(ii)] $\SH (a) < \SH $.
		\item[(iii)] $\SH (a)$ is achieved by some function $u \in \ths$.
		\end{enumerate}
	\end{thm}

Following the terminology of Hebey and Vaugon \cite{HeVa2001}, we call $a\in C(\overline\Omega)$ a \emph{critical potential} (or \emph{critical function}) if $\SH(a)=\SH$ and $\SH(\widetilde a)<\SH$ for every $\widetilde a\in C(\overline\Omega)$ such that $\widetilde a\leq a$ and $\widetilde a\not\equiv a$. Thus, criticality requires that every nontrivial decrease of the potential lowers the best constant below the threshold. Theorem~\ref{main} gives the following property.

\begin{cor}
\label{lem:de-1.3}
Let $a$ be a critical potential. Then $\inf_{x\in\Omega}\phi_a(x)=0$.
\end{cor}

We next consider bounded perturbations of a critical potential. Let $\mathcal N_a:=\{x\in\Omega:\phi_a(x)=0\}$. For $V\in L^\infty(\Omega)$, define $Q_V(x):=\int_\Omega V(y)G_a(x,y)^2dy$ and $\mathcal N_a(V):=\{x\in\mathcal N_a:Q_V(x)<0\}$. The integral defining $Q_V$ is finite because $N<4s$. The set $\mathcal N_a$ describes the possible limiting concentration points, while $Q_V$ measures the first contribution of the perturbation at these points. For the following asymptotic results, we assume in addition that
\begin{equation}
\label{eq:assumption-a-on-Na}
a\in C(\overline\Omega)\cap C^1(\Omega),
\qquad
a(x)<0
\quad\text{for all }x\in\mathcal N_a.
\end{equation}

If $\mathcal N_a(V) \neq  \emptyset$, the perturbation lowers the energy below $\SH$. More precisely, we have
\begin{thm}[Energy asymptotics]
\label{lem:de-1.4}
Let $\frac52s<N<4s$, let $a$ be a critical potential satisfying \eqref{eq:assumption-a-on-Na}, let $V\in L^\infty(\Omega)$, and suppose that $\mathcal N_a(V)\neq\emptyset$. Then $\SH(a+\varepsilon V)<\SH$ for every $\varepsilon>0$. Moreover,
\begin{equation*}
\lim_{\varepsilon\to0^+}
\frac{\SH-\SH(a+\varepsilon V)}
{\varepsilon^{\frac{2s}{4s-N}}}
=
\frac{4s-N}{2s}
\left(\frac{N-2s}{2s}\right)^{\frac{N-2s}{4s-N}}
\frac{\kappa_{N,\mu,s}^{2}\,\SH^{-\frac{N-2s}{N+2s-\mu}}}
{\left(a_{N,s}d_{N,s}b_{N,s}+\alpha_{N,s}\right)^{\frac{N-2s}{4s-N}}}
\sup_{x\in\mathcal N_a(V)}
\frac{|Q_V(x)|^{\frac{2s}{4s-N}}}
{|a(x)|^{\frac{N-2s}{4s-N}}}.
\end{equation*}
Here $a_{N,s}$, $b_{N,s}$, $\alpha_{N,s}$ and $d_{N,s}$ are the constants of Lemmas~\ref{lem:explicit-constants} and \ref{lemma Ha expansion}.
\end{thm}
On the other hand, when $\mathcal N_a(V)=\emptyset$, the next theorem
shows that $\SH(a+\varepsilon V)=\SH$ for all sufficiently small
$\varepsilon>0$, provided that $Q_V>0$ on $\mathcal N_a$. In the
 case $\min_{\mathcal N_a}Q_V=0$, we only obtain an
$o\bigl(\varepsilon^{\frac{2s}{4s-N}}\bigr)$ estimate and do not determine the precise leading
term of $\SH(a+\varepsilon V)-\SH$.

\begin{thm}[Energy asymptotics in the degenerate case]
\label{lem:de-1.5}
Let $\frac52s<N<4s$, let $a$ be a critical potential satisfying \eqref{eq:assumption-a-on-Na}, let $V\in L^\infty(\Omega)$, and suppose that $\mathcal N_a(V)=\emptyset$. Then $\SH(a+\varepsilon V)=\SH+o\bigl(\varepsilon^{\frac{2s}{4s-N}}\bigr)$ as $\varepsilon\to0^+$. If, in addition, $Q_V(x)>0$ for every $x\in\mathcal N_a$, then $\SH(a+\varepsilon V)=\SH$ for all sufficiently small $\varepsilon>0$.
\end{thm}
To impose the boundary condition on $\Omega$, we project $	\bar{U}_{z, \lambda}$ to $\ths$. More precisely, $P	\bar{U}_{z, \lambda} \in \ths$ is the unique solution of
\begin{equation*}
\begin{cases}
(- \Delta)^s P	\bar{U}_{z, \lambda} = (- \Delta)^s \bar{U}_{z, \lambda}  & \quad \text{ in } \Omega, \\
P  \bar{U}_{z, \lambda}  = 0  & \quad \text{ on } \R^N \setminus  \Omega, \end{cases}
\end{equation*}
in the weak sense, that is,
\begin{equation*}
  \int_{\R^N} \ds P\bar{U}_{z, \lambda}  \ds \eta dy = \int_{\R^N} \Ds\bar{U}_{z, \lambda}  \eta d y = a_{N,s}
A_{N,\mu,s}^{-\frac{2s}{N+2s-\mu}} \int_\Omega \bar{U}_{z, \lambda} ^\frac{N+2s}{N-2s} \eta d y
\end{equation*}
for every $\eta \in \ths$.

We also set
\begin{equation*}
  T_{z, \lambda} = \text{span} \left\{ P\bar{U}_{z, \lambda}, \partial_\lambda P\bar{U}_{z, \lambda}, \{\partial_{z_i} P\bar{U}_{z, \lambda}\}_{i=1}^N \right\} \subset \ths ,
\end{equation*}
and denote by $T_{z, \lambda}^\perp \subset \ths$ its orthogonal complement in $\ths$ with respect to the scalar product $(u,v) := \int_{\R^N} \ds u \ds v dy$. We denote by $\Pi_{z, \lambda}$ and $\Pi_{z, \lambda}^\bot$ the projections onto $T_{z, \lambda}$ and $T_{z, \lambda}^\bot$ respectively.

If $\mathcal N_a(V)\neq\emptyset$, Theorem~\ref{lem:de-1.4} yields $\SH(a+\varepsilon V)<\SH$ for all sufficiently small $\varepsilon>0$. Since $(-\Delta)^s+a+\varepsilon V$ remains coercive for such $\varepsilon$, the infimum $\SH(a+\varepsilon V)$ is attained. Indeed, the proof of the implication (ii)~$\Rightarrow$~(iii) of Theorem~\ref{main} only uses that the potential is bounded and that the associated operator is coercive, and therefore applies to $a+\varepsilon V\in L^\infty(\Omega)$.

We now study the asymptotic behavior of normalized almost-minimizers of
$\SH(a+\varepsilon V)$. The following result gives their decomposition
into a projected bubble, a Green function correction, and an orthogonal
remainder.

\begin{thm}[Concentration of almost-minimizers]
\label{lem:de-1.6}
Let $\frac52s<N<4s$, let $a$ be a critical potential satisfying \eqref{eq:assumption-a-on-Na}, let $V\in L^\infty(\Omega)$, and suppose that $\mathcal N_a(V)\neq\emptyset$. Let $(u_\varepsilon)\subset\ths$ satisfy
\begin{equation*}
\frac{
\SH(a+\varepsilon V)[u_\varepsilon]-\SH(a+\varepsilon V)
}{
\SH-\SH(a+\varepsilon V)
}
\longrightarrow0
\end{equation*}
and
\begin{equation*}
\int_\Omega\int_\Omega
\frac{|u_\varepsilon(x)|^{\twostar}|u_\varepsilon(y)|^{\twostar}}
{|x-y|^\mu}dxdy
=
\int_{\R^N}\int_{\R^N}
\frac{\bar U_{0,1}^{\twostar}(x)\bar U_{0,1}^{\twostar}(y)}
{|x-y|^\mu}dxdy.
\end{equation*}
Then there exist sequences $(z_\varepsilon)\subset\Omega$, $(\lambda_\varepsilon)\subset(0,\infty)$, $(\alpha_\varepsilon)\subset\R$, $(\beta_\varepsilon)\subset(0,\infty)$, and $(r_\varepsilon)\subset T_{z_\varepsilon,\lambda_\varepsilon}^\bot$ such that, up to a subsequence,
\begin{equation*}
\begin{split}
u_\varepsilon
=
\alpha_\varepsilon\Big(&P\bar U_{z_\varepsilon,\lambda_\varepsilon}
+\kappa_{N,\mu,s}\beta_\varepsilon\lambda_\varepsilon^{-\frac{N-2s}{2}}
\Pi_{z_\varepsilon,\lambda_\varepsilon}^\bot
\left(H_0(z_\varepsilon,\cdot)-H_a(z_\varepsilon,\cdot)\right)
+r_\varepsilon\Big).
\end{split}
\end{equation*}
Moreover, $z_\varepsilon\to z_0\in\mathcal N_a(V)$, where
\begin{equation*}
\frac{|Q_V(z_0)|^{\frac{2s}{4s-N}}}
{|a(z_0)|^{\frac{N-2s}{4s-N}}}
=
\sup_{z\in\mathcal N_a(V)}
\frac{|Q_V(z)|^{\frac{2s}{4s-N}}}
{|a(z)|^{\frac{N-2s}{4s-N}}}.
\end{equation*}
Furthermore, $\phi_a(z_\varepsilon)=o(\varepsilon)$ and
\begin{equation*}
\lim_{\varepsilon\to0^+}\varepsilon^{\frac{1}{4s-N}}\lambda_\varepsilon
=
\left(
\frac{2s\left(a_{N,s}d_{N,s}b_{N,s}+\alpha_{N,s}\right)|a(z_0)|}
{(N-2s)|Q_V(z_0)|}
\right)^{\frac{1}{4s-N}}.
\end{equation*}
For some $\varsigma\in\{-1,1\}$, we also have $\alpha_\varepsilon=\varsigma+O\bigl(\varepsilon^{\frac{N-2s}{4s-N}}\bigr)$, $\beta_\varepsilon=1+O\bigl(\varepsilon^{\frac{N-2s}{4s-N}}\bigr)$ and $\|\ds r_\varepsilon\|_2^2=o\bigl(\varepsilon^{\frac{2s}{4s-N}}\bigr)$.
\end{thm}

\subsection{Comments on the results}

Some comments on the results are in order.

\begin{enumerate}[label=(\alph*),leftmargin=*]
\item Since $s\in(0,1)$, the range $2s<N<4s$ is nonempty only for $N\in\{1,2,3\}$: it corresponds to $s\in(\frac14,\frac12)$ if $N=1$, to $s\in(\frac12,1)$ if $N=2$, and to $s\in(\frac34,1)$ if $N=3$. Formally, the case $N=3$, $s\to1$ corresponds to the local three-dimensional problem studied in \cite{Druet2002,FKK1} and, for the Hartree nonlinearity, in \cite{GGYZ,WuYangHartree}.
\item Theorem~\ref{main} holds in the whole range $2s<N<4s$ and for every $0<\mu<N$, and it only requires the continuity of $a$. The additional restriction $\frac52s<N<4s$ in Theorems~\ref{lem:de-1.4}--\ref{lem:de-1.6} is used only in the refined lower bound of Proposition~\ref{prop:refined-lower-bound}\,(b). There, the error produced by the reparametrization must be $o(\lambda^{-2s})$ after squaring. In terms of $(N,s)$, this means $s\in(\frac14,\frac25)$ if $N=1$, $s\in(\frac12,\frac45)$ if $N=2$, while for $N=3$ the restriction is void.
\item The rate $\varepsilon^{\frac{2s}{4s-N}}$ in Theorem~\ref{lem:de-1.4} and the concentration scale $\lambda_\varepsilon\sim\varepsilon^{-\frac{1}{4s-N}}$ in Theorem~\ref{lem:de-1.6} come from the competition between the term $-a(z)\lambda^{-2s}$, generated by the expansion of $H_a$ near the diagonal, and the term $\varepsilon Q_V(z)\lambda^{-(N-2s)}$, generated by the perturbation. For $N=3$ and $s=1$ they reduce to $\varepsilon^2$ and $\varepsilon^{-1}$, in agreement with \cite{FKK1}. The exponent $\mu$ only enters through the constants $\kappa_{N,\mu,s}$ and $\SH$.
\end{enumerate}

\subsection{Organization of the paper}
In Section~\ref{sec:UB} we expand the quotient $\SH(a+\varepsilon V)$ at a suitable test function, namely a projected bubble corrected by the regular part of the Green function, and deduce the upper bounds and some necessary conditions for critical potentials. Section~\ref{sec:lower-bound-first-expansion} contains the non-existence of minimizers at the threshold level, the profile decomposition of almost-minimizing sequences, the coercivity of the second variation on the orthogonal complement of the tangent space, and a first expansion of the energy. In Section~\ref{sec:lower-bound-refined} we refine this expansion and obtain a lower bound which matches the upper bound of Section~\ref{sec:UB}. Theorem~\ref{main} and Corollary~\ref{lem:de-1.3} are proved in Section~\ref{sec:proof-characterization}, and Theorems~\ref{lem:de-1.4}, \ref{lem:de-1.5} and \ref{lem:de-1.6} in Section~\ref{sec:proof-asymptotics}. The Appendix collects the explicit constants and the estimates on bubbles and on the regular part of the Green function used throughout the paper.

\subsection{Notation}

Unless stated otherwise, we assume throughout that $s\in(0,1)$,
$2s<N<4s$, and $0<\mu<N$. We denote the fractional Sobolev critical
exponent and the upper critical Hardy--Littlewood--Sobolev exponent by
$$
2_s^*:=\frac{2N}{N-2s},
\qquad
\twostar:=\frac{2N-\mu}{N-2s},
$$
respectively.
We also set
$K:=\left\lfloor\frac{2s}{N-2s}\right\rfloor$, where $K$ is the largest
integer less than or equal to $\frac{2s}{N-2s}$.
 For $q\geq1$, we write
$\|u\|_q:=\|u\|_{L^q(\R^N)}$, and use the shorthand
$\|\ds u\|_2^2:=\int_{\R^N}|\ds u|^2dx$. For $z\in\Omega$, we set
$d(z):=\operatorname{dist}(z,\partial\Omega)$.

\section{Proof of the upper bound}
\label{sec:UB}

We derive an expansion of $\SH (a + \varepsilon V) [\psi_{z, \lambda}]$ using the test function
\begin{equation*}
\psi_{z, \lambda}(x) := P\bar{U}_{z, \lambda}(x) - \lambda^{-\frac{N-2s}{2}} \kappa_{N,\mu,s}(H_a(z,x) - H_0(z,x)),
\end{equation*}
as $\lambda \to \infty$. This expansion does not require the criticality of $a$.
\begin{thm}[Expansion of $\SH (a + \varepsilon V) {[\psi_{z, \lambda}]}$]\label{lem:de-2.1}
Let $a\in C(\overline\Omega)\cap C^1(\Omega)$ be such that $(-\Delta)^s+a$ is coercive, and let $V\in L^\infty(\Omega)$.
As $\lambda \to \infty$, uniformly for $z$ in compact subsets of $\Omega$ and for $\varepsilon \geq 0$, we have
\begin{equation}
\label{eq:numerator-expansion}
\begin{split}
\int_{\R^N} \left|\ds \psi_{z,\lambda}\right|^2dx
+ \int_\Omega \bigl(a(x)+\varepsilon V(x)\bigr)
\psi_{z,\lambda}^2(x)dx
&=\SH^{\frac{2N-\mu}{N+2s-\mu}}
-\kappa_{N,\mu,s}^{2}
a_{N,s}\mathfrak a_{N,s}^{(1)}
\phi_a(z)\lambda^{-(N-2s)}
\\
&+\kappa_{N,\mu,s}^{2}
\left(a_{N,s}d_{N,s}b_{N,s}-\alpha_{N,s}\right)
a(z)\lambda^{-2s}
\\
&+\kappa_{N,\mu,s}^{2}
\varepsilon Q_V(z)\lambda^{-(N-2s)}
+o\left(\lambda^{-2s}\right)
+o\left(\varepsilon\lambda^{-(N-2s)}\right),
\end{split}
\end{equation}
and
\begin{equation}
\label{eq:denominator-expansion}
\begin{split}
&\int_\Omega\int_\Omega
\frac{
\left|\psi_{z,\lambda}(x)\right|^{\twostar}
\left|\psi_{z,\lambda}(y)\right|^{\twostar}}
{|x-y|^\mu}dxdy
\\
=&\SH^{\frac{2N-\mu}{N+2s-\mu}}
-2\times\twostar a_{N,s}\kappa_{N,\mu,s}^{2}
\mathfrak a_{N,s}^{(1)}\phi_a(z)\lambda^{-(N-2s)}
+2\times\twostar a_{N,s}\kappa_{N,\mu,s}^{2}
d_{N,s}b_{N,s}a(z)\lambda^{-2s}
\\
&+\sum_{k=2}^{K}(-1)^k
\Bigg[
2a_{N,s}\kappa_{N,\mu,s}^{2}
\binom{\twostar}{k}\mathfrak a_{N,s}^{(k)}
+\kappa_{N,\mu,s}^{k}
\sum_{i=1}^{k-1}
\binom{\twostar}{i}
\binom{\twostar}{k-i}
C_{i,k-i}
\Bigg]
\phi_a^k(z)\lambda^{-k(N-2s)}
+o\left(\lambda^{-2s}\right).
\end{split}
\end{equation}
In particular,
\begin{equation}
\label{eq:quotient-expansion}
\begin{split}
\SH(a+\varepsilon V)[\psi_{z,\lambda}]
=&\SH
+\kappa_{N,\mu,s}^{2}
a_{N,s}\mathfrak a_{N,s}^{(1)}
\SH^{-\frac{N-2s}{N+2s-\mu}}
\phi_a(z)\lambda^{-(N-2s)}
\\
&+\mathcal T_1\left(\phi_a(z),\lambda\right)
-\kappa_{N,\mu,s}^{2}
\left(a_{N,s}d_{N,s}b_{N,s}+\alpha_{N,s}\right)
\SH^{-\frac{N-2s}{N+2s-\mu}}
a(z)\lambda^{-2s}
\\
&+\kappa_{N,\mu,s}^{2}
\SH^{-\frac{N-2s}{N+2s-\mu}}
\varepsilon Q_V(z)\lambda^{-(N-2s)}
+o\left(\lambda^{-2s}\right)
+o\left(\varepsilon\lambda^{-(N-2s)}\right).
\end{split}
\end{equation}
Here, for $i=1,2$, $\mathcal T_i(\phi,\lambda)$ are (possibly empty) sums of the form
\begin{equation}
\label{eq:T-i-definition}
\mathcal T_i(\phi,\lambda)
:=\sum_{k=2}^{K}\gamma_i(k)\phi^k\lambda^{-k(N-2s)}
\end{equation}
for suitable constants $\gamma_i(k)\in\R$.
\end{thm}
\begin{proof}
The numerator is treated as in \cite[Theorem~2.1]{DeNittiKonig2023}, which gives \eqref{eq:numerator-expansion}. We therefore focus on the denominator.

Recall that $\psi_{z,\lambda}(x)
=\bar U_{z,\lambda}(x)
-\kappa_{N,\mu,s}
\lambda^{-\frac{N-2s}{2}}H_a(z,x)
-f_{z,\lambda}(x)$. By H\"older's inequality, the HLS inequality and Lemma \ref{lem:barU-Lq-estimates}, we have
\begin{equation*}
\begin{split}
&\Bigg|
\int_\Omega\int_\Omega
\frac{|\psi_{z,\lambda}(x)|^{\twostar}
|\psi_{z,\lambda}(y)|^{\twostar}}{|x-y|^\mu}dxdy
\\
&\quad-
\int_\Omega\int_\Omega
\frac{
\Big|\bar U_{z,\lambda}(x)
-\kappa_{N,\mu,s}
\lambda^{-\frac{N-2s}{2}}H_a(z,x)\Big|^{\twostar}\Big|\bar U_{z,\lambda}(y)
-\kappa_{N,\mu,s}
\lambda^{-\frac{N-2s}{2}}H_a(z,y)\Big|^{\twostar}}
{|x-y|^\mu}dxdy
\Bigg|
\\
&\leq
\int_\Omega\int_\Omega
\frac{
\Big|
|\psi_{z,\lambda}(x)|^{\twostar}
-\big|\bar U_{z,\lambda}(x)
-\kappa_{N,\mu,s}
\lambda^{-\frac{N-2s}{2}}H_a(z,x)\big|^{\twostar}
\Big|
\Big|\psi_{z,\lambda}(y)\Big|^{\twostar}}
{|x-y|^\mu}dxdy
\\
&\quad+
\int_\Omega\int_\Omega
\frac{
\Big|\bar U_{z,\lambda}(x)
-\kappa_{N,\mu,s}
\lambda^{-\frac{N-2s}{2}}H_a(z,x)\Big|^{\twostar}
\Big|
|\psi_{z,\lambda}(y)|^{\twostar}
-\big|\bar U_{z,\lambda}(y)
-\kappa_{N,\mu,s}
\lambda^{-\frac{N-2s}{2}}H_a(z,y)\big|^{\twostar}
\Big|}
{|x-y|^\mu}dxdy
\\&
\lesssim
\left\|
|\psi_{z,\lambda}|^{\twostar}
-\left|\bar U_{z,\lambda}
-\kappa_{N,\mu,s}
\lambda^{-\frac{N-2s}{2}}H_a(z,\cdot)\right|^{\twostar}
\right\|_{\frac{2N}{2N-\mu}}
\left(
\left\|\psi_{z,\lambda}\right\|^{\twostar}_{2_s^*}
+\left\|
\bar U_{z,\lambda}
-\kappa_{N,\mu,s}
\lambda^{-\frac{N-2s}{2}}H_a(z,\cdot)
\right\|^{\twostar}_{2_s^*}
\right)
\\
&\lesssim
\left\|
\left|\bar U_{z,\lambda}
-\kappa_{N,\mu,s}
\lambda^{-\frac{N-2s}{2}}H_a(z,\cdot)\right|^{\twostar-1}
f_{z,\lambda}
\right\|_{\frac{2N}{2N-\mu}}
+
\left\||f_{z,\lambda}|^{\twostar}\right\|_{\frac{2N}{2N-\mu}}
\\
&\lesssim
\left\|
\bar U_{z,\lambda}
-\kappa_{N,\mu,s}
\lambda^{-\frac{N-2s}{2}}H_a(z,\cdot)
\right\|_{2_s^*}^{\twostar-1}
\left\|f_{z,\lambda}\right\|_{\infty}
+
\left\|f_{z,\lambda}\right\|_{\infty}^{\twostar}
\\&=o(\lambda^{-2s}),
\end{split}
\end{equation*}
where the last estimate follows from $\left\|f_{z,\lambda}\right\|_{\infty}\lesssim \lambda^{-\frac{N+4-2s}{2}}$. For the main term, Taylor's formula gives
\begin{equation*}
\begin{split}
\left|
\bar U_{z,\lambda}(x)
-\kappa_{N,\mu,s}\lambda^{-\frac{N-2s}{2}}H_a(z,x)
\right|^{\twostar}
=&
\sum_{k=0}^{K}
(-1)^k
\binom{\twostar}{k}
\bar U_{z,\lambda}^{\twostar-k}(x)
\kappa_{N,\mu,s}^{k}
\lambda^{-\frac{k(N-2s)}{2}}
H_a^k(z,x)+\rho_{K,z,\lambda}(x),
\end{split}
\end{equation*}
where
\begin{equation*}
\rho_{K,z,\lambda}(x)=O\left(
\bar U_{z,\lambda}^{\twostar-K-1}(x)
\lambda^{-\frac{(K+1)(N-2s)}{2}}
\left|H_a(z,x)\right|^{K+1}
+\lambda^{-\frac{2N-\mu}{2}}
\left|H_a(z,x)\right|^{\twostar}\right),
\end{equation*}
and $\binom{\twostar}{k}:=
\frac{\Gamma(\twostar+1)}
{\Gamma(\twostar-k+1)\Gamma(k+1)}$ is the generalized binomial coefficient. Then we have
\begin{equation}
\label{eq:denominator-taylor-expansion}
\begin{split}
&\int_\Omega\int_\Omega
\frac{
\left|\bar U_{z,\lambda}(x)
-\kappa_{N,\mu,s}\lambda^{-\frac{N-2s}{2}}H_a(z,x)\right|^{\twostar}
\left|\bar U_{z,\lambda}(y)
-\kappa_{N,\mu,s}\lambda^{-\frac{N-2s}{2}}H_a(z,y)\right|^{\twostar}}
{|x-y|^\mu}dxdy
\\
=&
\int_\Omega\int_\Omega
\frac{
\bar U_{z,\lambda}^{\twostar}(x)
\bar U_{z,\lambda}^{\twostar}(y)}
{|x-y|^\mu}dxdy
-2\times\twostar\kappa_{N,\mu,s}
\lambda^{-\frac{N-2s}{2}}
\int_\Omega\int_\Omega
\frac{
\bar U_{z,\lambda}^{\twostar}(x)
\bar U_{z,\lambda}^{\twostar-1}(y)H_a(z,y)}
{|x-y|^\mu}dxdy
\\
&+\sum_{\substack{i,j\geq0\\2\leq i+j\leq K}}
(-1)^{i+j}\kappa_{N,\mu,s}^{i+j}
\lambda^{-\frac{(i+j)(N-2s)}{2}}
\binom{\twostar}{i}
\binom{\twostar}{j}
\int_\Omega\int_\Omega
\frac{
\bar U_{z,\lambda}^{\twostar-i}(x)H_a^i(z,x)
\bar U_{z,\lambda}^{\twostar-j}(y)H_a^j(z,y)}
{|x-y|^\mu}dxdy
+\mathcal R_{K,z,\lambda},
\end{split}
\end{equation}
where
\begin{equation*}
\begin{split}
\mathcal R_{K,z,\lambda}
&=O\Biggl(
\sum_{\substack{0\leq i,j\leq K\\ i+j\geq K+1}}
\lambda^{-(i+j)(N-2s)/2}
\int_\Omega\int_\Omega
\frac{\bar U_{z,\lambda}^{\twostar-i}(x)|H_a(z,x)|^i
\bar U_{z,\lambda}^{\twostar-j}(y)|H_a(z,y)|^j}{|x-y|^\mu}dxdy
\Biggr)
\\
&+O\Biggl(
\int_\Omega\int_\Omega
\frac{\left|\bar U_{z,\lambda}(x)-\kappa_{N,\mu,s}\lambda^{-(N-2s)/2}H_a(z,x)\right|^{\twostar}
\left|\rho_{K,z,\lambda}(y)\right|}{|x-y|^\mu}dxdy+
\int_\Omega\int_\Omega
\frac{\left|\rho_{K,z,\lambda}(x)\right|
\left|\rho_{K,z,\lambda}(y)\right|}{|x-y|^\mu}dxdy
\Biggr).
\end{split}
\end{equation*}

Using \eqref{easybarU}, the HLS inequality, and Lemmas \ref{lem:barU-Lq-estimates} and \ref{lem:barU-global-exterior-Lq}, we obtain
\begin{equation*}
\begin{split}
&\int_\Omega\int_\Omega
\frac{\bar U_{z,\lambda}^{\twostar}(y)
\bar U_{z,\lambda}^{\twostar}(x)}{|x-y|^\mu}dxdy
\\
=&\int_{\R^N}\int_\Omega
\frac{\bar U_{z,\lambda}^{\twostar}(y)
\bar U_{z,\lambda}^{\twostar}(x)}{|x-y|^\mu}dxdy
-\int_{\R^N\setminus\Omega}\int_\Omega
\frac{\bar U_{z,\lambda}^{\twostar}(y)
\bar U_{z,\lambda}^{\twostar}(x)}{|x-y|^\mu}dxdy
\\
=&a_{N,s}A_{N,\mu,s}^{-\frac{2s}{N+2s-\mu}}
\int_{\R^N}\bar U_{z,\lambda}^{2_s^*}(x)dx
+O\left(\int_{\R^N\setminus\Omega}
\bar U_{z,\lambda}^{2_s^*}(x)dx\right)
\\
=&a_{N,s}\kappa_{N,\mu,s}^{2}\mathfrak a_{N,s}^{(0)}
+O\left(\lambda^{-N}\right)
=\SH^{\frac{2N-\mu}{N+2s-\mu}}
+o\left(\lambda^{-2s}\right).
\end{split}
\end{equation*}
The last equality follows from the equations for $\bar U_{0,1}$, Lemma \ref{lem:explicit-constants}, and the fact that $\bar U_{0,1}$ minimizes the quotient in \eqref{hls_ineq}: indeed, testing \eqref{equation_s_hls} with $\bar U_{0,1}$ gives $\|\ds\bar U_{0,1}\|_2^2=D:=\int_{\R^N}\int_{\R^N}\frac{\bar U_{0,1}^{\twostar}(x)\bar U_{0,1}^{\twostar}(y)}{|x-y|^\mu}dxdy$, hence $\SH=D^{1-\frac{1}{\twostar}}$, that is, $D=\SH^{\frac{2N-\mu}{N+2s-\mu}}$.

For the second term in \eqref{eq:denominator-taylor-expansion}, we first estimate the integral over $\R^N\setminus\Omega$. We claim that
\begin{equation}
\label{eq:linear-Ha-exterior-error}
\lambda^{-\frac{N-2s}{2}}
\left|\int_\Omega\int_{\R^N\setminus\Omega}
\frac{\bar U_{z,\lambda}^{\twostar}(x)
\bar U_{z,\lambda}^{\twostar-1}(y)H_a(z,y)}{|x-y|^\mu}dxdy\right|
=o\left(\lambda^{-2s}\right).
\end{equation}
To this end, choose $r>1$ such that
\begin{equation*}
1<r<
\begin{cases}
\displaystyle
\min\left\{\frac{N}{N-\mu},\frac{N}{4s-N}\right\},
&\text{if}\quad0<\mu\leq 2s,
\\[2mm]
\displaystyle
\frac{N}{N+2s-\mu},
&\text{if}\quad2s<\mu<N,
\end{cases}
\end{equation*}
and let $t>1$ be determined by $1/t+\mu/N+1/r=2$, so that $(t,r)$ is an HLS-admissible pair. Since $\twostar t>\frac{N}{N-2s}$ and $\R^N\setminus\Omega\subset\R^N\setminus B_{d(z)}(z)$, Lemma \ref{lem:barU-global-exterior-Lq} (with $\tau=d(z)$) gives
\begin{equation*}
\left\|\bar U_{z,\lambda}\right\|_{L^{\twostar t}(\R^N\setminus\Omega)}^{\twostar}
\lesssim
\lambda^{-\frac{2N-\mu}{2}}.
\end{equation*}
Moreover, Lemma \ref{lem:barU-Lq-estimates} yields
\begin{equation*}
\left\|\bar U_{z,\lambda}\right\|_{L^{(\twostar-1)r}(\Omega)}^{\twostar-1}
\lesssim
\begin{cases}
\displaystyle
\lambda^{\frac{N+2s-\mu}{2}-\frac{N}{r}}
=o\left(\lambda^{\frac{3N-6s-\mu}{2}}\right),
&\text{if}\quad 0<\mu\leq 2s,
\\[2mm]
\displaystyle
\lambda^{-\frac{N+2s-\mu}{2}},
&\text{if}\quad 2s<\mu<N.
\end{cases}
\end{equation*}
Since $H_a(z,\cdot)$ is uniformly bounded for $z$ in compact subsets of $\Omega$, the preceding bounds and the HLS inequality give
\begin{equation*}
\lambda^{-\frac{N-2s}{2}}
\left|
\int_\Omega\int_{\R^N\setminus\Omega}
\frac{
\bar U_{z,\lambda}^{\twostar}(x)
\bar U_{z,\lambda}^{\twostar-1}(y)H_a(z,y)}
{|x-y|^\mu}dxdy
\right|
\lesssim
\lambda^{-\frac{N-2s}{2}}
\left\|\bar U_{z,\lambda}\right\|_{L^{\twostar t}(\R^N\setminus\Omega)}^{\twostar}
\left\|\bar U_{z,\lambda}\right\|_{L^{(\twostar-1)r}(\Omega)}^{\twostar-1}
=o\left(\lambda^{-2s}\right).
\end{equation*}
Combining this estimate with \eqref{easybarU} and Lemma \ref{lem:barU-Ha-power-asymptotics}, we obtain
\begin{equation*}
\begin{split}
&-2\times\twostar\kappa_{N,\mu,s}
\lambda^{-\frac{N-2s}{2}}
\int_\Omega\int_\Omega
\frac{\bar U_{z,\lambda}^{\twostar}(x)
\bar U_{z,\lambda}^{\twostar-1}(y)H_a(z,y)}{|x-y|^\mu}dxdy
\\
=&-2\times\twostar\kappa_{N,\mu,s}
\lambda^{-\frac{N-2s}{2}}
\int_\Omega\int_{\R^N}
\frac{\bar U_{z,\lambda}^{\twostar}(x)
\bar U_{z,\lambda}^{\twostar-1}(y)H_a(z,y)}{|x-y|^\mu}dxdy
+o\left(\lambda^{-2s}\right)
\\
=&-2\times\twostar\kappa_{N,\mu,s}a_{N,s}
A_{N,\mu,s}^{-\frac{2s}{N+2s-\mu}}
\lambda^{-\frac{N-2s}{2}}
\int_\Omega\bar U_{z,\lambda}^{2_s^*-1}(y)H_a(z,y)dy
+o\left(\lambda^{-2s}\right)
\\
=&-2\times\twostar a_{N,s}\kappa_{N,\mu,s}^{2}
\mathfrak a_{N,s}^{(1)}\phi_a(z)\lambda^{-(N-2s)}
+2\times\twostar a_{N,s}\kappa_{N,\mu,s}^{2}
d_{N,s}b_{N,s}a(z)\lambda^{-2s}
+o\left(\lambda^{-2s}\right).
\end{split}
\end{equation*}

For the third term in \eqref{eq:denominator-taylor-expansion}, Lemma \ref{lem:barU-Ha-double-endpoint-asymptotics} gives
\begin{equation*}
\begin{split}
&\sum_{\substack{i,j\geq0\\2\leq i+j\leq K}}
(-1)^{i+j}\kappa_{N,\mu,s}^{i+j}
\lambda^{-\frac{(i+j)(N-2s)}{2}}
\binom{\twostar}{i}
\binom{\twostar}{j}
\int_\Omega\int_\Omega
\frac{
\bar U_{z,\lambda}^{\twostar-i}(x)H_a^i(z,x)
\bar U_{z,\lambda}^{\twostar-j}(y)H_a^j(z,y)}
{|x-y|^\mu}dxdy
\\
=&\sum_{k=2}^{K}(-1)^k
\Bigg[
2a_{N,s}\kappa_{N,\mu,s}^{2}
\binom{\twostar}{k}
\mathfrak a_{N,s}^{(k)}
+\kappa_{N,\mu,s}^{k}
\sum_{i=1}^{k-1}
\binom{\twostar}{i}
\binom{\twostar}{k-i}
C_{i,k-i}
\Bigg]
\phi_a^k(z)\lambda^{-k(N-2s)}
+o\left(\lambda^{-2s}\right).
\end{split}
\end{equation*}

It remains to bound $\mathcal R_{K,z,\lambda}$. Lemma \ref{lem:barU-Ha-double-endpoint-asymptotics} gives
\begin{equation*}
  \sum_{\substack{0\leq i,j\leq K\\ i+j\geq K+1}}
\lambda^{-(i+j)(N-2s)/2}
\int_\Omega\int_\Omega
\frac{\bar U_{z,\lambda}^{\twostar-i}(x)|H_a(z,x)|^i
\bar U_{z,\lambda}^{\twostar-j}(y)|H_a(z,y)|^j}{|x-y|^\mu}dxdy=o\left(\lambda^{-2s}\right).
\end{equation*}
Using \eqref{easybarU} and Lemmas \ref{lem:barU-Lq-estimates} and \ref{lem:barU-Ha-power-asymptotics}, a direct computation shows
\begin{equation*}
  \begin{split}
   & \int_\Omega\int_\Omega
\frac{\left|\bar U_{z,\lambda}(x)-\kappa_{N,\mu,s}\lambda^{-(N-2s)/2}H_a(z,x)\right|^{\twostar}
\left|\rho_{K,z,\lambda}(y)\right|}{|x-y|^\mu}dxdy\\
\lesssim &
\int_\Omega\int_\Omega
\frac{\bar U_{z,\lambda}^{\twostar}(x)
\left(
\bar U_{z,\lambda}^{\twostar-K-1}(y)
\lambda^{-\frac{(K+1)(N-2s)}{2}}
\left|H_a(z,y)\right|^{K+1}
+\lambda^{-\frac{2N-\mu}{2}}
\left|H_a(z,y)\right|^{\twostar}\right)}{|x-y|^\mu}dxdy+o\left(\lambda^{-2s}\right)
\\
\lesssim & \int_\Omega
\left( \lambda^{-\frac{(K+1)(N-2s)}{2}}
\bar U_{z,\lambda}^{2_s^*-K-1}(y)\left|H_a(z,y)\right|^{K+1}+\lambda^{-\frac{2N-\mu}{2}}
\bar U_{z,\lambda}^{\frac{\mu}{N-2s}}(y)\right)dy+o\left(\lambda^{-2s}\right)=o\left(\lambda^{-2s}\right).
  \end{split}
\end{equation*}
Similarly, we obtain
\begin{equation*}
  \int_\Omega\int_\Omega
\frac{\left|\rho_{K,z,\lambda}(x)\right|
\left|\rho_{K,z,\lambda}(y)\right|}{|x-y|^\mu}dxdy=o\left(\lambda^{-2s}\right).
\end{equation*}
Collecting these estimates gives \eqref{eq:denominator-expansion}.

We now apply the Taylor expansion
\begin{equation*}
(A+B)^{-\frac{1}{\twostar}}
=A^{-\frac{1}{\twostar}}
-\frac{1}{\twostar}A^{-\frac{1}{\twostar}-1}B
+\sum_{k=2}^{K}
\binom{-\frac{1}{\twostar}}{k}
A^{-\frac{1}{\twostar}-k}B^k
+O\left(B^{K+1}\right)
\end{equation*}
to the right-hand side of \eqref{eq:denominator-expansion}, with $
A=\SH^{\frac{2N-\mu}{N+2s-\mu}}
$ and
\begin{equation*}
\begin{split}
B=&-2\times\twostar a_{N,s}\kappa_{N,\mu,s}^{2}
\mathfrak a_{N,s}^{(1)}\phi_a(z)\lambda^{-(N-2s)}+2\times\twostar a_{N,s}\kappa_{N,\mu,s}^{2}
d_{N,s}b_{N,s}a(z)\lambda^{-2s}
\\
&+\sum_{k=2}^{K}(-1)^k
\Bigg[
2a_{N,s}\kappa_{N,\mu,s}^{2}
\binom{\twostar}{k}\mathfrak a_{N,s}^{(k)}
+\kappa_{N,\mu,s}^{k}
\sum_{i=1}^{k-1}
\binom{\twostar}{i}
\binom{\twostar}{k-i}
C_{i,k-i}
\Bigg]
\phi_a^k(z)\lambda^{-k(N-2s)}
+o\left(\lambda^{-2s}\right).
\end{split}
\end{equation*}
Since $B=O\left(\lambda^{-(N-2s)}\right)$ and $K+1>\frac{2s}{N-2s}$, we have $
O\left(B^{K+1}\right)=o\left(\lambda^{-2s}\right)$. Therefore,
\begin{equation*}
\begin{split}
\left(
\int_\Omega\int_\Omega
\frac{
\left|\psi_{z,\lambda}(x)\right|^{\twostar}
\left|\psi_{z,\lambda}(y)\right|^{\twostar}}
{|x-y|^\mu}dxdy
\right)^{-\frac{1}{\twostar}}
=&\SH^{-\frac{N-2s}{N+2s-\mu}}
+2a_{N,s}\kappa_{N,\mu,s}^{2}
\mathfrak a_{N,s}^{(1)}
\SH^{-\frac{3N-2s-\mu}{N+2s-\mu}}
\phi_a(z)\lambda^{-(N-2s)}
\\
+\mathcal T_2\left(\phi_a(z),\lambda\right)&
-2a_{N,s}\kappa_{N,\mu,s}^{2}
d_{N,s}b_{N,s}
\SH^{-\frac{3N-2s-\mu}{N+2s-\mu}}
a(z)\lambda^{-2s}
+o\left(\lambda^{-2s}\right),
\end{split}
\end{equation*}
where $\mathcal T_2$ is given by \eqref{eq:T-i-definition}. Multiplying this expansion by \eqref{eq:numerator-expansion}, we obtain \eqref{eq:quotient-expansion}.
\end{proof}

The expansion in Theorem \ref{lem:de-2.1} gives two useful properties of the potential.
\begin{cor}[Properties of critical potentials]
\label{lem:de-2.2}
Let $a\in C(\overline\Omega)\cap C^1(\Omega)$ be such that $(-\Delta)^s+a$ is coercive.
\begin{enumerate}
\item[(i)] If $\SH(a)=\SH$, then $\phi_a(x)\geq0$ for every $x\in\Omega$.
\item[(ii)] If $\SH(a)=\SH$ and $\phi_a(x_0)=0$ for some $x_0\in\Omega$, then $a(x_0)\leq0$.
\end{enumerate}
\end{cor}
\begin{proof}
We set $\varepsilon=0$ in \eqref{eq:quotient-expansion}. Note that $\kappa_{N,\mu,s}^2a_{N,s}\mathfrak a_{N,s}^{(1)}>0$ and the explicit formulas in Lemmas~\ref{lem:explicit-constants} and \ref{lemma Ha expansion} give
\begin{equation*}
a_{N,s}d_{N,s}b_{N,s}+\alpha_{N,s}
=-\frac{\pi^{N/2}\Gamma\left(\frac N2-2s\right)}{\Gamma(N-2s)}>0,
\end{equation*}
since $-1<\frac N2-2s<0$ and $N-2s>0$, so that the Gamma factors in the numerator and denominator are negative and positive, respectively. If $\phi_a(x)<0$ at some point $x\in\Omega$, the term of order $\lambda^{-(N-2s)}$ dominates, since $N-2s<2s$ and $\mathcal T_1(\phi_a(x),\lambda)=O(\lambda^{-2(N-2s)})$. If $\phi_a(x)=0<a(x)$, then $\mathcal T_1(0,\lambda)=0$ and the term of order $\lambda^{-2s}$ is negative. In both cases, for $\lambda$ sufficiently large, the expansion yields
\begin{equation*}
\SH(a)\leq\SH(a)[\psi_{x,\lambda}]<\SH,
\end{equation*}
contrary to the assumption $\SH(a)=\SH$.
\end{proof}
\begin{Rem}[{Expansion of $\SH(a+\varepsilon V)[\psi_{z,\lambda}]$ for merely continuous potentials}]
\label{rem:continuous-a}
Let $a\in C(\overline\Omega)$ be such that $(-\Delta)^s+a$ is coercive.
The difference $H_a(z,\cdot)-H_0(z,\cdot)=G_0(z,\cdot)-G_a(z,\cdot)$ is the weak solution in $\ths$ of
\begin{equation*}
(-\Delta)^s\bigl(H_a(z,\cdot)-H_0(z,\cdot)\bigr)=aG_a(z,\cdot)\quad\text{in }\Omega,
\qquad H_a(z,\cdot)-H_0(z,\cdot)=0\quad\text{in }\R^N\setminus\Omega.
\end{equation*}
Since $2(N-2s)<N$, the Green function bound gives $\|aG_a(z,\cdot)\|_{L^2(\Omega)}\leq C$, uniformly for $z$ in compact subsets of $\Omega$. The energy estimate and \cite[Theorem~1.3]{BiccariWarmaZuazua2017} yield locally uniform $W^{2s,2}$ bounds for $H_a(z,\cdot)-H_0(z,\cdot)$. By the Sobolev embedding, this difference is locally $C^\alpha$ for $0<\alpha<\min\{1,2s-N/2\}$, with bounds uniform for $z$ in compact sets. Its symmetry and the continuity of $H_0$ therefore imply that $H_a$ is jointly continuous in $\Omega\times\Omega$, and in particular $\phi_a\in C(\Omega)$.
Consequently, $H_a(z,x)=\phi_a(z)+o(1)$ as $x\to z$, uniformly for $z$ in compact sets.
Splitting the integral near and away from $z$, the uniform continuity at the diagonal, Lemma~\ref{lem:barU-Lq-estimates} and the bubble decay give
\begin{equation*}
\lambda^{-\frac{N-2s}{2}}\int_\Omega\bar U_{z,\lambda}^{2_s^*-1}(x)|H_a(z,x)-\phi_a(z)|\,dx
=o\left(\lambda^{-(N-2s)}\right).
\end{equation*}
Using this estimate in the proof of Theorem~\ref{lem:de-2.1}, together with $N<4s$ and $\varepsilon\int_\Omega V\psi_{z,\lambda}^2\,dx=O(\varepsilon\lambda^{-(N-2s)})$, we obtain, as $\lambda\to\infty$ and $\varepsilon\to0^+$, uniformly for $z$ in compact subsets of $\Omega$,
\begin{equation*}
\int_{\R^N}|\ds\psi_{z,\lambda}|^2\,dx+\int_\Omega(a+\varepsilon V)\psi_{z,\lambda}^2\,dx
=\SH^{\frac{2N-\mu}{N+2s-\mu}}-\kappa_{N,\mu,s}^2a_{N,s}\mathfrak a_{N,s}^{(1)}\phi_a(z)\lambda^{-(N-2s)}+o\left(\lambda^{-(N-2s)}\right),
\end{equation*}
and
\begin{equation*}
\int_\Omega\int_\Omega\frac{|\psi_{z,\lambda}(x)|^{\twostar}|\psi_{z,\lambda}(y)|^{\twostar}}{|x-y|^\mu}\,dxdy
=\SH^{\frac{2N-\mu}{N+2s-\mu}}-2\twostar\kappa_{N,\mu,s}^2a_{N,s}\mathfrak a_{N,s}^{(1)}\phi_a(z)\lambda^{-(N-2s)}+o\left(\lambda^{-(N-2s)}\right).
\end{equation*}
Taking the quotient yields
\begin{equation}
\label{eq:quotient-expansion-continuous}
\SH(a+\varepsilon V)[\psi_{z,\lambda}]
=\SH+c_1\phi_a(z)\lambda^{-(N-2s)}+o\left(\lambda^{-(N-2s)}\right),
\qquad
c_1:=\kappa_{N,\mu,s}^{2}a_{N,s}\mathfrak a_{N,s}^{(1)}\SH^{-\frac{N-2s}{N+2s-\mu}}>0.
\end{equation}
\end{Rem}

\section{Proof of the lower bound I: a first expansion}
\label{sec:lower-bound-first-expansion}

\subsection{Non-existence of a minimizer for $\SH(a)$}
\label{subsec:non-existence-minimizer}

We first show that, at the threshold level $\SH$, the infimum $\SH(a)$ is never attained. This will be used both in the profile decomposition below and in the proof of Theorem~\ref{main}.

\begin{Prop}[Non-existence of a minimizer for $\SH(a)$]
\label{lem:de-3.1}
Suppose that $a\in C(\overline{\Omega})$ satisfies $\SH(a)=\SH$; for instance, let $a$ be a critical potential. Then
\begin{equation*}
\SH(a)
=
\inf_{u\in\ths,\,u\not\equiv0}
\frac{
\displaystyle\int_{\R^N}|\ds u|^2dx
+\displaystyle\int_\Omega a(x)u^2(x)dx
}{
\displaystyle\left(
\int_\Omega\int_\Omega
\frac{|u(x)|^{\twostar}|u(y)|^{\twostar}}{|x-y|^\mu}dxdy
\right)^{\frac{1}{\twostar}}
}
\end{equation*}
is not achieved.
\end{Prop}

\begin{proof}
Suppose, by contradiction, that $\SH(a)=\SH$ and $\SH(a)$ is attained by some function $u$. Replacing $u$ by $|u|$, which is still a minimizer because $\|\ds|u|\|_2\leq\|\ds u\|_2$, we may assume that $u\geq0$. We normalize $u$ so that
\begin{equation}
\label{eq:de-3.2}
\int_\Omega\int_\Omega
\frac{u^{\twostar}(x)u^{\twostar}(y)}{|x-y|^\mu}dxdy
=1.
\end{equation}
Then $u$ satisfies the Euler--Lagrange equation
\begin{equation}
\label{eq:de-3.3}
\Ds u(x)+a(x)u(x)
=
\SH
\left(
\int_\Omega
\frac{u^{\twostar}(y)}{|x-y|^\mu}dy
\right)
u^{\twostar-1}(x)
\qquad\text{in }\Omega.
\end{equation}

Since $\SH(a)=\SH$, for every $\varphi\in C^\infty(\R^N)$ which is bounded together with all its derivatives and every sufficiently small $\varepsilon>0$ (so that $1+\varepsilon\varphi>0$), we have
\begin{equation}
\label{eq:de-3.4}
\begin{split}
&
\left(
\int_\Omega\int_\Omega
\frac{
|u(x)(1+\varepsilon\varphi(x))|^{\twostar}
|u(y)(1+\varepsilon\varphi(y))|^{\twostar}
}{|x-y|^\mu}dxdy
\right)^{\frac{1}{\twostar}}
\\
&\quad\leq
\SH^{-1}
\left[
\int_{\R^N}
\left|
\ds\left(u(1+\varepsilon\varphi)\right)
\right|^2dx
+
\int_\Omega
a(x)u^2(x)(1+\varepsilon\varphi(x))^2dx
\right].
\end{split}
\end{equation}

By Taylor's formula and \eqref{eq:de-3.2}, the left-hand side of \eqref{eq:de-3.4} has the expansion
\begin{equation}
\label{eq:de-3.5}
\begin{split}
&
\left(
\int_\Omega\int_\Omega
\frac{
|u(x)(1+\varepsilon\varphi(x))|^{\twostar}
|u(y)(1+\varepsilon\varphi(y))|^{\twostar}
}{|x-y|^\mu}dxdy
\right)^{\frac{1}{\twostar}}
\\
&\quad \quad \quad \quad=
1+2A_1\varepsilon
+
\left[
\twostar A_2
+
(\twostar-1)(A_3-2A_1^2)
\right]\varepsilon^2
+o(\varepsilon^2),
\end{split}
\end{equation}
where
\begin{equation*}
\begin{split}
A_1
&:=
\int_\Omega\int_\Omega
\frac{
u^{\twostar}(x)u^{\twostar}(y)\varphi(y)
}{|x-y|^\mu}dxdy,
\\
A_2
&:=
\int_\Omega\int_\Omega
\frac{
u^{\twostar}(x)\varphi(x)
u^{\twostar}(y)\varphi(y)
}{|x-y|^\mu}dxdy,
\\
A_3
&:=
\int_\Omega\int_\Omega
\frac{
u^{\twostar}(x)u^{\twostar}(y)\varphi^2(y)
}{|x-y|^\mu}dxdy.
\end{split}
\end{equation*}

For the right-hand side of \eqref{eq:de-3.4}, we use the fractional product rule
\begin{equation}
\label{eq:de-3.1}
\Ds(uv)
=
u\Ds v+v\Ds u-I_s(u,v),
\end{equation}
where
\begin{equation*}
I_s(u,v)(x)
:=
C_{N,s}\operatorname{P.V.}
\int_{\R^N}
\frac{(u(x)-u(y))(v(x)-v(y))}{|x-y|^{N+2s}}dy.
\end{equation*}
Using \eqref{eq:de-3.1} and integration by parts, we obtain
\begin{equation*}
\begin{split}
&\int_{\R^N}
\left|
\ds\left(u(1+\varepsilon\varphi)\right)
\right|^2dx
\\=&
\int_\Omega
u(1+\varepsilon\varphi)
\Ds\left(u(1+\varepsilon\varphi)\right)dx
\\
=&
\int_\Omega
u\Ds u(1+\varepsilon\varphi)^2dx
+
\varepsilon
\int_\Omega
u^2(1+\varepsilon\varphi)\Ds\varphi dx
-
\varepsilon
\int_\Omega
u(1+\varepsilon\varphi)I_s(u,\varphi)dx
\\
=&
\int_\Omega
u\Ds u(1+\varepsilon\varphi)^2dx
+
C_{N,s}
\int_{\R^N}\int_{\R^N}
u(x)u(y)(1+\varepsilon\varphi(x))
\frac{\varepsilon(\varphi(x)-\varphi(y))}{|x-y|^{N+2s}}dxdy
\\
=&
\int_\Omega
u\Ds u(1+\varepsilon\varphi)^2dx
+
\frac{C_{N,s}\varepsilon^2}{2}
\int_{\R^N}\int_{\R^N}
u(x)u(y)
\frac{|\varphi(x)-\varphi(y)|^2}{|x-y|^{N+2s}}dxdy.
\end{split}
\end{equation*}
In the last equality, we have symmetrized in $x$ and $y$. Using \eqref{eq:de-3.2}, \eqref{eq:de-3.3}, and the definitions of $A_1$ and $A_3$, we rewrite the right-hand side of \eqref{eq:de-3.4} as
\begin{equation}
\label{eq:right-hand-side-expansion}
\begin{split}
&
\SH^{-1}
\left[
\int_{\R^N}
\left|
\ds\left(u(1+\varepsilon\varphi)\right)
\right|^2dx
+
\int_\Omega
a u^2(1+\varepsilon\varphi)^2dx
\right]
\\
&\quad=
\int_\Omega
\left(
\int_\Omega
\frac{u^{\twostar}(y)}{|x-y|^\mu}dy
\right)
u^{\twostar}(x)(1+\varepsilon\varphi(x))^2dx
+
\SH^{-1}\varepsilon^2\mathcal I(\varphi)
\\
&\quad=
1+2A_1\varepsilon
+A_3\varepsilon^2
+\SH^{-1}\varepsilon^2\mathcal I(\varphi),
\end{split}
\end{equation}
where
\begin{equation*}
\mathcal I(\varphi)
:=
\frac{C_{N,s}}{2}
\int_{\R^N}\int_{\R^N}
u(x)u(y)
\frac{|\varphi(x)-\varphi(y)|^2}{|x-y|^{N+2s}}dxdy.
\end{equation*}

Combining \eqref{eq:de-3.4} with the expansions \eqref{eq:de-3.5} and \eqref{eq:right-hand-side-expansion}, we obtain
\begin{equation}
\label{eq:de-3.8}
\twostar A_2+(\twostar-2)A_3
\leq
2(\twostar-1)A_1^2
+\SH^{-1}\mathcal I(\varphi).
\end{equation}

We next choose $(\zeta,t)\in\R^N\times(0,\infty)$ such that
\begin{equation}
\label{eq:FG=0}
\begin{split}
F(\zeta,t)
&:=
\int_\Omega\int_\Omega
\frac{u^{\twostar}(x)u^{\twostar}(y)}{|x-y|^\mu}
\frac{2t(y-\zeta)}{1+t^2|y-\zeta|^2}dxdy
=0,
\\
G(\zeta,t)
&:=
\int_\Omega\int_\Omega
\frac{u^{\twostar}(x)u^{\twostar}(y)}{|x-y|^\mu}
\frac{1-t^2|y-\zeta|^2}{1+t^2|y-\zeta|^2}dxdy
=0.
\end{split}
\end{equation}
Indeed, following \cite{Druet2002}, we consider the map $H:\R^N\times\R\longrightarrow\R^N\times\R$ defined by
\begin{equation*}
H(\zeta,\tau)
:=
\left(
F\left(\zeta,\frac{\tau+\sqrt{\tau^2+4}}{2}\right)+\zeta,
G\left(\zeta,\frac{\tau+\sqrt{\tau^2+4}}{2}\right)+\tau
\right).
\end{equation*}
A direct asymptotic computation using \eqref{eq:de-3.2} and the boundedness of $\Omega$ yields $
\left|H(\zeta,\tau)\right|^2
\leq
|\zeta|^2+\tau^2
$ whenever $|\zeta|^2+\tau^2$ is sufficiently large. Moreover, by \eqref{eq:de-3.2}, $|(F,G)|\leq1$, so that $|H(\zeta,\tau)|\leq|(\zeta,\tau)|+1$ for all $(\zeta,\tau)$. Thus, for sufficiently large $R>0$, $H$ maps $\overline{B_R(0)}$ into itself, and Brouwer's fixed point theorem implies that $H$ admits a fixed point $(\zeta,\tau)$. Consequently, $
\left(\zeta,\frac{\tau+\sqrt{\tau^2+4}}{2}\right)
$ satisfies \eqref{eq:FG=0}.

With the pair $(\zeta,t)$ obtained above, let $\mathscr S:\R^N\to\mathbb S^N$ be the inverse stereographic projection given by
\begin{equation*}
  \mathscr S(x)
:=
\left(
\frac{2t(x-\zeta)}{1+t^2|x-\zeta|^2},
\frac{1-t^2|x-\zeta|^2}{1+t^2|x-\zeta|^2}
\right).
\end{equation*}
Set
\begin{equation}
\label{eq:de-3.9}
\varphi_i(x)
:=
\bigl(\mathscr S(x)\bigr)_i,
\qquad i=1,\ldots,N+1.
\end{equation}
Then \eqref{eq:FG=0} gives
\begin{equation}
\label{eq:de-3.10}
A_1(\varphi_i)=0,
\qquad i=1,\ldots,N+1.
\end{equation}
Moreover, since $
\sum_{i=1}^{N+1}\varphi_i^2(x)=1$, it follows from \eqref{eq:de-3.2}, \eqref{eq:de-3.8} and \eqref{eq:de-3.10} that
\begin{equation}
\label{eq:de-3.11}
\twostar\sum_{i=1}^{N+1}A_2(\varphi_i)
+(\twostar-2)
\leq
\SH^{-1}\sum_{i=1}^{N+1}\mathcal I(\varphi_i).
\end{equation}
We first estimate $A_2$. Set $
J_{\mathscr S^{-1}}(\eta)
:=
\det D\mathscr S^{-1}(\eta)
$ and define
\begin{equation*}
U(\eta)
:=
u\bigl(\mathscr S^{-1}(\eta)\bigr)
J_{\mathscr S^{-1}}(\eta)^{\frac{N-2s}{2N}},
\qquad \eta\in\mathbb S^N.
\end{equation*}
The distance transformation is
\begin{equation*}
\left|\mathscr S^{-1}(\eta)-\mathscr S^{-1}(\xi)\right|
=
J_{\mathscr S^{-1}}(\eta)^{\frac{1}{2N}}
|\eta-\xi|
J_{\mathscr S^{-1}}(\xi)^{\frac{1}{2N}},
\end{equation*}
so a change of variables in \eqref{eq:de-3.2} gives
\begin{equation}\label{eq:twostar=1}
\int_{\mathbb S^N}\int_{\mathbb S^N}
\frac{
U^{\twostar}(\eta)U^{\twostar}(\xi)}
{|\eta-\xi|^\mu}
d\eta d\xi
=\int_\Omega\int_\Omega
\frac{u^{\twostar}(x)u^{\twostar}(y)}{|x-y|^\mu}dxdy=1.
\end{equation}
The same change of variables yields
\begin{equation*}
\begin{split}
\sum_{i=1}^{N+1}A_2(\varphi_i)
&=
\int_\Omega\int_\Omega
\frac{u^{\twostar}(x)u^{\twostar}(y)}{|x-y|^\mu}
\sum_{i=1}^{N+1}\varphi_i(x)\varphi_i(y)dxdy
\\
&=
\int_\Omega\int_\Omega
\frac{
u^{\twostar}(x)
\bigl(\mathscr S(x)\cdot\mathscr S(y)\bigr)
u^{\twostar}(y)}
{|x-y|^\mu}dxdy
\\
&=
\int_{\mathbb S^N}\int_{\mathbb S^N}
\frac{
U^{\twostar}(\eta)
(\eta\cdot\xi)
U^{\twostar}(\xi)}
{|\eta-\xi|^\mu}
d\eta d\xi
\\
&\geq
\frac{\mu}{2N-\mu}
\int_{\mathbb S^N}\int_{\mathbb S^N}
\frac{
U^{\twostar}(\eta)U^{\twostar}(\xi)}
{|\eta-\xi|^\mu}
d\eta d\xi
=
\frac{\mu}{2N-\mu},
\end{split}
\end{equation*}
where the inequality follows from \cite[Proposition~4.1]{FrLi2012}. The inequality is strict because $U^{\twostar}$ vanishes on a nonempty open subset of $\mathbb S^N$ and therefore cannot be constant. Consequently, \eqref{eq:de-3.11} yields
\begin{equation}\label{de:eq-i}
\sum_{i=1}^{N+1}\mathcal I(\varphi_i)
>
\frac{4s}{N-2s}\SH.
\end{equation}

We now estimate $\sum_{i=1}^{N+1}\mathcal I(\varphi_i)$. By the definition of $\mathcal I$, \eqref{eq:de-3.9},  and the Cauchy--Schwarz inequality, we obtain
\begin{equation*}
\begin{split}
\sum_{i=1}^{N+1}\mathcal I(\varphi_i)
&=
\frac{C_{N,s}}{2}
\int_{\R^N}\int_{\R^N}
u(x)u(y)
\frac{
\sum_{i=1}^{N+1}
|\varphi_i(x)-\varphi_i(y)|^2
}
{|x-y|^{N+2s}}
dxdy
\\
&=
\frac{C_{N,s}}{2}
\int_{\R^N}\int_{\R^N}
u(x)u(y)
\frac{
|\mathscr S(x)-\mathscr S(y)|^2
}
{|x-y|^{N+2s}}
dxdy
\\
&=
\frac{C_{N,s}}{2}
\int_{\mathbb S^N}\int_{\mathbb S^N}
\frac{U(\eta)U(\xi)}
{|\eta-\xi|^{N+2s-2}}
d\eta d\xi
\\
&\leq
\frac{C_{N,s}}{2}\delta_{N,s}
\int_{\mathbb S^N}U^2(\eta)d\eta,
\end{split}
\end{equation*}
where
\begin{equation*}
\delta_{N,s}
:=
\int_{\mathbb S^N}
\frac{d\xi}{|\eta-\xi|^{N+2s-2}}
=
2^{2-2s}\pi^{N/2}
\frac{\Gamma(1-s)}
{\Gamma\left(\frac N2+1-s\right)}.
\end{equation*}
It remains to bound $\int_{\mathbb S^N}U^2(\eta)d\eta$. Let $
\mathfrak m
:=
\frac{1}{|\mathbb S^N|}
\int_{\mathbb S^N}U^{\twostar}(\eta)d\eta.
$ By \eqref{eq:twostar=1} and the decomposition $U^{\twostar}=\mathfrak m+(U^{\twostar}-\mathfrak m)$, we obtain
\begin{equation*}
\begin{split}
1
&=
\int_{\mathbb S^N}\int_{\mathbb S^N}
\frac{
\left[
\mathfrak m
+
\left(U^{\twostar}(\eta)-\mathfrak m\right)
\right]
\left[
\mathfrak m
+
\left(U^{\twostar}(\xi)-\mathfrak m\right)
\right]
}
{|\eta-\xi|^\mu}
d\eta d\xi
\\
&=
\int_{\mathbb S^N}\int_{\mathbb S^N}
\frac{
\mathfrak m^2
+
\mathfrak m\left(U^{\twostar}(\xi)-\mathfrak m\right)
+
\left(U^{\twostar}(\eta)-\mathfrak m\right)\mathfrak m
+
\left(U^{\twostar}(\eta)-\mathfrak m\right)
\left(U^{\twostar}(\xi)-\mathfrak m\right)
}
{|\eta-\xi|^\mu}
d\eta d\xi
\\
&=
2^{N-\mu}\pi^{N/2}
\frac{\Gamma\left(\frac{N-\mu}{2}\right)}
{\Gamma\left(N-\frac{\mu}{2}\right)}
\frac{1}{|\mathbb S^N|}
\left(\int_{\mathbb S^N}U^{\twostar}(\eta)d\eta\right)
^2
+
\int_{\mathbb S^N}\int_{\mathbb S^N}
\frac{
\left(U^{\twostar}(\eta)-\mathfrak m\right)
\left(U^{\twostar}(\xi)-\mathfrak m\right)
}
{|\eta-\xi|^\mu}
d\eta d\xi.
\end{split}
\end{equation*}
Here the cross term vanishes because
\begin{equation*}
\int_{\mathbb S^N}\int_{\mathbb S^N}
\frac{
\mathfrak m\left(U^{\twostar}(\xi)-\mathfrak m\right)
}
{|\eta-\xi|^\mu}
d\eta d\xi
=
2^{N-\mu}\pi^{N/2}
\frac{\Gamma\left(\frac{N-\mu}{2}\right)}
{\Gamma\left(N-\frac{\mu}{2}\right)}
\mathfrak m
\int_{\mathbb S^N}
\left(
U^{\twostar}(\xi)-\mathfrak m
\right)d\xi
=0.
\end{equation*}
By \cite[Corollary~4.3]{FrLi2012} and the fact that $U^{\twostar}-\mathfrak m\not\equiv0$, we have
\begin{equation*}
\int_{\mathbb S^N}\int_{\mathbb S^N}
\frac{
\left(U^{\twostar}(\eta)-\mathfrak m\right)
\left(U^{\twostar}(\xi)-\mathfrak m\right)
}
{|\eta-\xi|^\mu}
d\eta d\xi
>0.
\end{equation*}
H\"older's inequality then gives
\begin{equation*}
\begin{split}
2^{N-\mu}\pi^{N/2}
\frac{\Gamma\left(\frac{N-\mu}{2}\right)}
{\Gamma\left(N-\frac{\mu}{2}\right)}
|\mathbb S^N|^{1-\twostar}
\left(
\int_{\mathbb S^N}U^2(\eta)d\eta
\right)^{\twostar}
\leq
2^{N-\mu}\pi^{N/2}
\frac{\Gamma\left(\frac{N-\mu}{2}\right)}
{\Gamma\left(N-\frac{\mu}{2}\right)}
\frac{1}{|\mathbb S^N|}
\left(
\int_{\mathbb S^N}U^{\twostar}(\eta)d\eta
\right)^2
<1.
\end{split}
\end{equation*}
Combining the preceding estimates and using $\SH=\mathcal S C(N,\mu)^{-1/\twostar}$, together with the identities
\begin{equation*}
\mathcal S=\frac{\Gamma\left(\frac{N+2s}{2}\right)}{\Gamma\left(\frac{N-2s}{2}\right)}|\mathbb S^N|^{\frac{2s}{N}},
\qquad
C(N,\mu)=2^{N-\mu}\pi^{N/2}\frac{\Gamma\left(\frac{N-\mu}{2}\right)}{\Gamma\left(N-\frac{\mu}{2}\right)}|\mathbb S^N|^{-\frac{N-\mu}{N}},
\qquad
\frac{C_{N,s}}{2}\delta_{N,s}=\frac{4s}{N-2s}\frac{\Gamma\left(\frac{N+2s}{2}\right)}{\Gamma\left(\frac{N-2s}{2}\right)},
\end{equation*}
a direct computation shows
\begin{equation*}
\begin{split}
\sum_{i=1}^{N+1}\mathcal I(\varphi_i)
&\leq
\frac{C_{N,s}}{2}\delta_{N,s}
\int_{\mathbb S^N}U^2(\eta)d\eta
<
\frac{C_{N,s}}{2}\delta_{N,s}
|\mathbb S^N|^{\frac{2s}{N}}
C(N,\mu)^{-\frac{1}{\twostar}}
=
\frac{4s}{N-2s}\SH.
\end{split}
\end{equation*}
This contradicts \eqref{de:eq-i} and proves the proposition.
\end{proof}

\subsection{Profile decomposition}
\label{subsec:profile-decomposition}

We now study the asymptotic decomposition of  normalized almost-minimizers of $\SH(a+\varepsilon V)$.
\begin{Prop}[Profile decomposition]
\label{prop:profile-decomposition}
Let $a\in C(\overline\Omega)$ satisfy $\SH(a)=\SH$ (for instance, let $a$ be critical) and let $V\in L^\infty(\Omega)$. Suppose that $\SH(a+\varepsilon V)<\SH$ for a sequence $\varepsilon\to0^+$  and let $(u_\varepsilon)\subset\ths$  be a sequence satisfying
\begin{equation}
\label{eq:profile-almost-minimizing}
\frac{\SH(a+\varepsilon V)[u_\varepsilon]-\SH(a+\varepsilon V)}{\SH-\SH(a+\varepsilon V)}\longrightarrow0
\end{equation}
and
\begin{equation}
\label{eq:profile-normalization}
\int_\Omega\int_\Omega\frac{|u_\varepsilon(x)|^{\twostar}|u_\varepsilon(y)|^{\twostar}}{|x-y|^\mu}\,dx\,dy
=\int_{\R^N}\int_{\R^N}\frac{\bar U_{0,1}^{\twostar}(x)\bar U_{0,1}^{\twostar}(y)}{|x-y|^\mu}\,dx\,dy.
\end{equation}
Then, up to a subsequence, there exist $z_\varepsilon\in\Omega$, $\lambda_\varepsilon>0$, $\alpha_\varepsilon\in\R$, and $w_\varepsilon\in T_{z_\varepsilon,\lambda_\varepsilon}^{\perp}$ such that
\begin{equation}
\label{eq:profile-decomposition}
u_\varepsilon=\alpha_\varepsilon\bigl(P\bar U_{z_\varepsilon,\lambda_\varepsilon}+w_\varepsilon\bigr).
\end{equation}
Moreover, as $\varepsilon\to0$,
\begin{equation}
\label{eq:profile-parameter-limits}
\|\ds w_\varepsilon\|_2\to0,\qquad
\lambda_\varepsilon d(z_\varepsilon)\to\infty,\qquad
z_\varepsilon\to z_0\in\overline\Omega,\qquad
\alpha_\varepsilon\to\varsigma\in\{-1,1\}.
\end{equation}
\end{Prop}

\begin{proof}
Since $\SH(a+\varepsilon V)\to\SH(a)=\SH$,
it follows from \eqref{eq:profile-almost-minimizing} and
\eqref{eq:profile-normalization} that
\begin{equation*}
\int_{\R^N}|\ds u_\varepsilon|^2\,dx+\int_\Omega(a+\varepsilon V)u_\varepsilon^2\,dx
=\SH^{\frac{2N-\mu}{N+2s-\mu}}+o(1).
\end{equation*}
Since $\SH(a)=\SH>0$, the operator $(-\Delta)^s+a$ is coercive (if $\|\ds v_n\|_2=1$ and $\|\ds v_n\|_2^2+\int_\Omega av_n^2\,dx\to0$, a weak limit $v$ of $(v_n)$ would satisfy $\int_\Omega av^2\,dx=-1$, so $v\neq0$, and $\|\ds v\|_2^2+\int_\Omega av^2\,dx\leq0$, contradicting $\SH(a)>0$). Hence $(u_\varepsilon)$ is bounded in $\ths$. Then, up to a subsequence, we may assume that $u_\varepsilon\rightharpoonup u_0$  in $\ths$. Moreover, by Rellich's lemma, we can assume that  $u_\varepsilon\to u_0$ in $L^2(\Omega)$ and almost everywhere. Using the nonlocal Brezis--Lieb lemma \cite[Lemma~3.3]{MukherjeeSreenadh2016} and adapting the splitting argument in \cite[proof of Proposition~2.2, Step~2]{WuYangHartree}, we obtain
\begin{equation*}
\int_{\R^N}|\ds u_0|^2\,dx+\int_\Omega au_0^2\,dx
\leq\SH\left(\int_\Omega\int_\Omega\frac{|u_0(x)|^{\twostar}|u_0(y)|^{\twostar}}{|x-y|^\mu}\,dx\,dy\right)^{\frac{1}{\twostar}}.
\end{equation*}
Thus, if $u_0\neq0$, the definition of $\SH(a)$ implies that $u_0$ is a minimizer for $\SH(a)$. Proposition~\ref{lem:de-3.1} rules out this possibility, and hence $u_0\equiv0$. It follows that
\begin{equation}
\label{eq:profile-vanishing-potential}
\int_\Omega(a+\varepsilon V)u_\varepsilon^2\,dx=o(1),\qquad
\|\ds u_\varepsilon\|_2^2=\SH^{\frac{2N-\mu}{N+2s-\mu}}+o(1).
\end{equation}
Taking the HLS inequality and  the fractional Sobolev embedding into consideration, we have
\begin{equation*}
\mathcal S\leq\frac{\|\ds u_\varepsilon\|_2^2}{\|u_\varepsilon\|_{2_s^*}^2}
\leq C(N,\mu)^{\frac{1}{\twostar}}\frac{\|\ds u_\varepsilon\|_2^2}{\displaystyle\left(\int_\Omega\int_\Omega\frac{|u_\varepsilon(x)|^{\twostar}|u_\varepsilon(y)|^{\twostar}}{|x-y|^\mu}\,dx\,dy\right)^{\frac{1}{\twostar}}}\longrightarrow C(N,\mu)^{\frac{1}{\twostar}}\SH=\mathcal S.
\end{equation*}
Hence $(u_\varepsilon)$ is a minimizing sequence for the fractional Sobolev quotient.  After normalization in $\dot H^s(\R^N)$, the hypotheses of
\cite[Theorem~3]{PaPi2014} are satisfied. By translations and a change of variables, there exist
$\widetilde z_\varepsilon\in\R^N$, $\widetilde\lambda_\varepsilon>0$, and
$\sigma_\varepsilon\in\dot H^s(\R^N)$ such that
\begin{equation}
\label{eq:profile-unprojected-bubble}
u_\varepsilon
=
\varsigma\bar U_{\widetilde z_\varepsilon,\widetilde\lambda_\varepsilon}
+\sigma_\varepsilon,
\qquad
\|\ds\sigma_\varepsilon\|_2\to0.
\end{equation}
Here $\varsigma\in\{-1,1\}$, since \eqref{eq:profile-vanishing-potential} gives $\|\ds u_\varepsilon\|_2^2\to\|\ds\bar U_{0,1}\|_2^2$.
Moreover, since $u_\varepsilon\equiv0$ outside $\Omega$ and $\|u_\varepsilon-\varsigma\bar U_{\widetilde z_\varepsilon,\widetilde\lambda_\varepsilon}\|_{2_s^*}\to0$, the smoothness of $\Omega$, together with
\begin{equation*}
\int_{\widetilde\lambda_\varepsilon(\Omega-\widetilde z_\varepsilon)}
\bar U_{0,1}^{2_s^*}(y)\,dy
=
\int_\Omega
\bar U_{\widetilde z_\varepsilon,\widetilde\lambda_\varepsilon}^{2_s^*}(y)\,dy
=
\int_{\R^N}\bar U_{0,1}^{2_s^*}(y)\,dy+o(1)
\end{equation*}
implies that
$\widetilde\lambda_\varepsilon\operatorname{dist}
(\widetilde z_\varepsilon,\R^N\setminus\Omega)\to\infty$.

Finally,  applying the argument in
\cite[proof of Proposition~3.3, Step~2]{DeNittiKonig2023}, we complete
 the proof.
\end{proof}

\subsection{Coercivity}
\label{subsec:coercivity}

In the following, we shall improve the estimate for the remainder obtained in
Proposition~\ref{prop:profile-decomposition}. The coercivity inequality below,
together with its version including the potential in
Proposition~\ref{prop:coercivity-with-potential}, will be essential for refining
the expansion of $u_\varepsilon$.
\begin{Prop}[Coercivity inequality]
\label{prop:coercivity-without-potential} (See \cite[Lemma~2.5]{LuYangZhao2025}.)
There exists a constant $\rho_0>0$, depending only on $N$, $\mu$, and $s$, such that for every $z\in\Omega$, $\lambda>0$, and $v\in T_{z,\lambda}^{\perp}$,
\begin{equation}
\label{eq:coercivity-without-potential}
\begin{aligned}
\|\ds v\|_2^2
&-(\twostar-1)\int_\Omega\int_\Omega
\frac{\bar U_{z,\lambda}^{\twostar}(x)\bar U_{z,\lambda}^{\twostar-2}(y)v(y)^2}{|x-y|^\mu}\,dx\,dy\\
&-\twostar\int_\Omega\int_\Omega
\frac{\bar U_{z,\lambda}^{\twostar-1}(x)v(x)\bar U_{z,\lambda}^{\twostar-1}(y)v(y)}{|x-y|^\mu}\,dx\,dy
\geq\rho_0\|\ds v\|_2^2.
\end{aligned}
\end{equation}
\end{Prop}

Although \cite[Lemma~2.5]{LuYangZhao2025} is formulated on $\R^N$, it applies in the present setting. Indeed, for $v\in\ths$ the definition of $P\bar U_{z,\lambda}$ gives $(v,P\bar U_{z,\lambda})=(v,\bar U_{z,\lambda})$, and similarly $(v,\partial_\lambda P\bar U_{z,\lambda})=(v,\partial_\lambda\bar U_{z,\lambda})$, $(v,\partial_{z_i}P\bar U_{z,\lambda})=(v,\partial_{z_i}\bar U_{z,\lambda})$, where $(\cdot,\cdot)$ is the scalar product of $\dot H^s(\R^N)$. Hence every $v\in T_{z,\lambda}^\perp$ is orthogonal in $\dot H^s(\R^N)$ to the tangent space of the manifold of bubbles at $\bar U_{z,\lambda}$. Moreover, since $v\equiv0$ outside $\Omega$, the last double integral in \eqref{eq:coercivity-without-potential} coincides with the corresponding integral over $\R^N\times\R^N$, while the first one is bounded from above by it.

We next establish the corresponding coercivity inequality with the potential $a$.
\begin{Prop}[Coercivity with a potential $a$]
\label{prop:coercivity-with-potential}
Let $a\in C(\overline\Omega)$ and assume that $(-\Delta)^s+a$ is coercive. There exist constants $T^*>0$ and $\rho>0$ such that, whenever $z\in\Omega$, $\lambda d(z)\geq T^*$, and $v\in T_{z,\lambda}^{\perp}$,
\begin{equation}
\label{eq:coercivity-with-potential}
\begin{aligned}
\|\ds v\|_2^2+\int_\Omega av^2\,dx
&-(\twostar-1)\int_\Omega\int_\Omega
\frac{\bar U_{z,\lambda}^{\twostar}(x)\bar U_{z,\lambda}^{\twostar-2}(y)v(y)^2}{|x-y|^\mu}\,dx\,dy\\
&-\twostar\int_\Omega\int_\Omega
\frac{\bar U_{z,\lambda}^{\twostar-1}(x)v(x)\bar U_{z,\lambda}^{\twostar-1}(y)v(y)}{|x-y|^\mu}\,dx\,dy
\geq\rho\|\ds v\|_2^2.
\end{aligned}
\end{equation}
\end{Prop}

\begin{proof}
We use the method of  \cite{Esposito2004}. Denote
\begin{equation*}
\begin{aligned}
\Psi_{z,\lambda}(v):=(\twostar-1)\int_\Omega\int_\Omega
\frac{\bar U_{z,\lambda}^{\twostar}(x)\bar U_{z,\lambda}^{\twostar-2}(y)v(y)^2}{|x-y|^\mu}\,dx\,dy+\twostar\int_\Omega\int_\Omega
\frac{\bar U_{z,\lambda}^{\twostar-1}(x)v(x)\bar U_{z,\lambda}^{\twostar-1}(y)v(y)}{|x-y|^\mu}\,dx\,dy.
\end{aligned}
\end{equation*}
Let $z_n\in\Omega$ and $\lambda_n>0$ satisfy $\lambda_n d(z_n)\to+\infty$, and define
\begin{equation*}
C_n:=\inf_{\substack{v\in T_{z_n,\lambda_n}^{\perp}\\\|\ds v\|_2=1}}
\left\{1+\int_\Omega av^2\,dx-\Psi_{z_n,\lambda_n}(v)\right\}.
\end{equation*}
The HLS inequality and the Sobolev embedding imply that
$C_n$ is bounded from below, uniformly in $n$. We first
claim that $C_n$ is achieved whenever $C_n<1$.
Indeed, assume that $C_n<1$ and let
$(\tilde v_k)\subset T_{z_n,\lambda_n}^{\perp}$ be a
minimizing sequence. Up to a subsequence, we may assume that
$\tilde v_k\rightharpoonup v_n$ in $\ths$ as $k\to\infty$, where
$v_n\in T_{z_n,\lambda_n}^{\perp}$ and
$\|\ds v_n\|_2^2\leq1$. Since $z_n$ and $\lambda_n$ are fixed here, the weights $\bar U_{z_n,\lambda_n}$ are bounded on $\Omega$; hence, by the HLS inequality and the compact embedding $\ths\hookrightarrow L^2(\Omega)$, both $\int_\Omega a\tilde v_k^2\,dx$ and $\Psi_{z_n,\lambda_n}(\tilde v_k)$ converge to the corresponding quantities for $v_n$. We obtain
\begin{equation*}
C_n
=
1+\int_\Omega av_n^2\,dx
-\Psi_{z_n,\lambda_n}(v_n).
\end{equation*}

Since $C_n<1$, we have $v_n\not\equiv0$. Moreover, the
definition of $C_n$ gives
\begin{equation*}
\begin{aligned}
0
\leq
(1-C_n)\|\ds v_n\|_2^2
+\int_\Omega av_n^2\,dx
-\Psi_{z_n,\lambda_n}(v_n)=
(1-C_n)
\bigl(\|\ds v_n\|_2^2-1\bigr)
\leq0.
\end{aligned}
\end{equation*}
Therefore, $\|\ds v_n\|_2=1$, and the infimum
$C_n$ is achieved by $v_n$.

Let $c\in(0,1]$ be the coercivity constant of $(-\Delta)^s+a$, that is, $\|\ds v\|_2^2+\int_\Omega av^2\,dx\geq c\|\ds v\|_2^2$ for all $v\in\ths$. We now show that $\liminf_{n\to\infty}C_n\geq\min\{c,\rho_0\}$. Otherwise, by the uniform lower bound for $C_n$, we may choose a subsequence such that $C_n\to L<\min\{c,\rho_0\}\leq1$; in particular, $C_n$ is achieved by some $v_n$ for $n$ large, and we may assume $v_n\rightharpoonup v_0$ in $\ths$. The minimizer $v_n$ satisfies
\begin{equation*}
(1-C_n)(v_n,\eta)+\int_\Omega av_n\eta\,dx-\frac12 D\Psi_{z_n,\lambda_n}(v_n)[\eta]=0
\qquad\text{for all }\eta\in T_{z_n,\lambda_n}^\perp.
\end{equation*}
For fixed $\eta\in C_c^\infty(\Omega)$ we apply this identity to $\Pi_{z_n,\lambda_n}^\perp\eta$. Since $\lambda_n d(z_n)\to\infty$, the bubbles concentrate, so that $\|\ds\Pi_{z_n,\lambda_n}\eta\|_2\to0$ and $D\Psi_{z_n,\lambda_n}(v_n)[\eta]\to0$. Letting $n\to\infty$, we find that $v_0$ satisfies
\begin{equation*}
(1-L)(-\Delta)^s v_0+av_0=0
\qquad\text{in }(\ths)'.
\end{equation*}
Testing with $v_0$ gives $0=(1-L)\|\ds v_0\|_2^2+\int_\Omega av_0^2\,dx\geq(c-L)\|\ds v_0\|_2^2$, and since $L<c$ we conclude $v_0=0$. The compact embedding into $L^2(\Omega)$ then gives $\int_\Omega av_n^2\,dx=o(1)$. By Proposition~\ref{prop:coercivity-without-potential} and $\|\ds v_n\|_2=1$,
\begin{equation*}
C_n=\|\ds v_n\|_2^2-\Psi_{z_n,\lambda_n}(v_n)+o(1)
\geq\rho_0+o(1),
\end{equation*}
contradicting $L<\rho_0$. Since the sequence $(z_n,\lambda_n)$ with $\lambda_n d(z_n)\to\infty$ was arbitrary, a contradiction argument shows that the conclusion holds with $\rho=\frac12\min\{c,\rho_0\}$ and
$T^*$ sufficiently large. This completes the proof.
\end{proof}

\subsection{Refining the decomposition}
\label{subsec:Refining the decomposition}

We next refine the decomposition in Proposition~\ref{prop:profile-decomposition}. The following proposition gives a quantitative estimate on $w_\varepsilon$ and shows that the  concentration point $z_0$ lies in the interior of $\Omega$.
\begin{Prop}
\label{prop:remainder-estimates}
As $\varepsilon\to0$, we have
\begin{equation}
\label{eq:remainder-and-distance-estimates}
\|\ds w_\varepsilon\|_2=O\left(\lambda_\varepsilon^{-\frac{N-2s}{2}}\right),
\qquad d(z_\varepsilon)^{-1}=O(1).
\end{equation}
In particular, $z_0\in\Omega$.
\end{Prop}
The proof of Proposition~\ref{prop:remainder-estimates} relies on the following expansion of $\SH(a+\varepsilon V)[u_\varepsilon]$ under the decomposition $u_\varepsilon=\alpha_\varepsilon(P\bar U_{z_\varepsilon,\lambda_\varepsilon}+w_\varepsilon)$ obtained in Proposition~\ref{prop:profile-decomposition}.
\begin{lem}
\label{lem:first-energy-expansion}
As $\varepsilon\to0$, we have
\begin{equation}
\label{eq:first-energy-expansion}
\begin{aligned}
&\SH(a+\varepsilon V)[u_\varepsilon]=\SH+\kappa_{N,\mu,s}^2 a_{N,s}\mathfrak a_{N,s}^{(1)}
\SH^{-\frac{N-2s}{N+2s-\mu}}
\phi_0(z_\varepsilon)\lambda_\varepsilon^{-(N-2s)}\\
&\quad+\SH^{-\frac{N-2s}{N+2s-\mu}}
\Bigg\{\|\ds w_\varepsilon\|_2^2+\int_\Omega aw_\varepsilon^2\,dx-(\twostar-1)\int_\Omega\int_\Omega
\frac{\bar U_{z_\varepsilon,\lambda_\varepsilon}^{\twostar}(x)
\bar U_{z_\varepsilon,\lambda_\varepsilon}^{\twostar-2}(y)w_\varepsilon(y)^2}
{|x-y|^\mu}\,dx\,dy\\
&\quad-\twostar\int_\Omega\int_\Omega
\frac{\bar U_{z_\varepsilon,\lambda_\varepsilon}^{\twostar-1}(x)w_\varepsilon(x)
\bar U_{z_\varepsilon,\lambda_\varepsilon}^{\twostar-1}(y)w_\varepsilon(y)}
{|x-y|^\mu}\,dx\,dy\Bigg\}+O\left(\lambda_\varepsilon^{-(N-2s)}\right)
+O\left(\lambda_\varepsilon^{-\frac{N-2s}{2}}\|\ds w_\varepsilon\|_2\right)\\
&\quad+o\left((\lambda_\varepsilon d(z_\varepsilon))^{-(N-2s)}\right)
+o\left(\|\ds w_\varepsilon\|_2^2\right).
\end{aligned}
\end{equation}
\end{lem}
\begin{proof}
The numerator is treated as in \cite[proof of Lemma~3.7, Step~1]{DeNittiKonig2023}. We state the resulting expansion directly:
\begin{equation}
\label{eq:first-expansion-numerator}
\begin{aligned}
&\alpha_\varepsilon^{-2}\left(\|\ds u_\varepsilon\|_2^2+\int_\Omega(a+\varepsilon V)u_\varepsilon^2\,dx\right)=\SH^{\frac{2N-\mu}{N+2s-\mu}}
-\kappa_{N,\mu,s}^2 a_{N,s}\mathfrak a_{N,s}^{(1)}\phi_0(z_\varepsilon)\lambda_\varepsilon^{-(N-2s)}\\
&\qquad\quad+\|\ds w_\varepsilon\|_2^2+\int_\Omega aw_\varepsilon^2\,dx
+O\left(\lambda_\varepsilon^{-(N-2s)}\right)
+O\left(\lambda_\varepsilon^{-\frac{N-2s}{2}}\|\ds w_\varepsilon\|_2\right)\\
&\qquad\quad+o\left((\lambda_\varepsilon d(z_\varepsilon))^{-(N-2s)}\right)
+o\left(\|\ds w_\varepsilon\|_2^2\right).
\end{aligned}
\end{equation}
For the denominator, we argue as in the proof of \cite[Lemma~5.5]{GGYZ}, using  the fractional Sobolev embedding, to obtain
\begin{equation}
\label{eq:first-expansion-denominator}
\begin{aligned}
&|\alpha_\varepsilon|^{-2 \cdot \twostar}
\int_\Omega\int_\Omega
\frac{|u_\varepsilon(x)|^{\twostar}|u_\varepsilon(y)|^{\twostar}}
{|x-y|^\mu}\,dx\,dy=\SH^{\frac{2N-\mu}{N+2s-\mu}}
-2\cdot\twostar\kappa_{N,\mu,s}^2 a_{N,s}\mathfrak a_{N,s}^{(1)}
\phi_0(z_\varepsilon)\lambda_\varepsilon^{-(N-2s)}\\
&\quad+\twostar(\twostar-1)\int_\Omega\int_\Omega
\frac{\bar U_{z_\varepsilon,\lambda_\varepsilon}^{\twostar}(x)
\bar U_{z_\varepsilon,\lambda_\varepsilon}^{\twostar-2}(y)
w_\varepsilon(y)^2}
{|x-y|^\mu}\,dx\,dy+(\twostar)^2\int_\Omega\int_\Omega
\frac{\bar U_{z_\varepsilon,\lambda_\varepsilon}^{\twostar-1}(x)
w_\varepsilon(x)
\bar U_{z_\varepsilon,\lambda_\varepsilon}^{\twostar-1}(y)
w_\varepsilon(y)}
{|x-y|^\mu}\,dx\,dy\\
&\quad+o\left((\lambda_\varepsilon d(z_\varepsilon))^{-(N-2s)}\right)
+o\left(\|\ds w_\varepsilon\|_2^2\right).
\end{aligned}
\end{equation}
Since $\twostar>2$, the Taylor remainder in $w_\varepsilon$ is $o(\|\ds w_\varepsilon\|_2^2)$. Applying Taylor's formula to the power $-1/\twostar$ of \eqref{eq:first-expansion-denominator} and multiplying by \eqref{eq:first-expansion-numerator}, we obtain \eqref{eq:first-energy-expansion}.
\end{proof}

\begin{proof}[Proof of Proposition~\ref{prop:remainder-estimates}]
By \eqref{eq:profile-almost-minimizing} and $\SH(a+\varepsilon V)<\SH$, we have $\SH(a+\varepsilon V)[u_\varepsilon]<\SH$ for $\varepsilon$ small. Since $w_\varepsilon\in T_{z_\varepsilon,\lambda_\varepsilon}^\perp$ and $\lambda_\varepsilon d(z_\varepsilon)\to\infty$, Proposition~\ref{prop:coercivity-with-potential} bounds the term in braces in \eqref{eq:first-energy-expansion} from below by $\rho\|\ds w_\varepsilon\|_2^2$. Moreover, $\phi_0(z)\geq c\,d(z)^{-(N-2s)}$ for some $c>0$ (see \cite[Lemma~A.1, (A.1)]{DeNittiKonig2023}). Hence \eqref{eq:first-energy-expansion} gives
\begin{equation*}
c\,(\lambda_\varepsilon d(z_\varepsilon))^{-(N-2s)}+\rho\|\ds w_\varepsilon\|_2^2
\lesssim
\lambda_\varepsilon^{-(N-2s)}+\lambda_\varepsilon^{-\frac{N-2s}{2}}\|\ds w_\varepsilon\|_2
+o\left((\lambda_\varepsilon d(z_\varepsilon))^{-(N-2s)}\right)+o\left(\|\ds w_\varepsilon\|_2^2\right).
\end{equation*}
Absorbing the $o(\cdot)$ terms and using the inequality
\begin{equation*}
\lambda_\varepsilon^{-\frac{N-2s}{2}}\|\ds w_\varepsilon\|_2\leq\frac{\rho}{2C}\|\ds w_\varepsilon\|_2^2+C'\lambda_\varepsilon^{-(N-2s)},
\end{equation*}
we obtain $(\lambda_\varepsilon d(z_\varepsilon))^{-(N-2s)}+\|\ds w_\varepsilon\|_2^2\lesssim\lambda_\varepsilon^{-(N-2s)}$, which is \eqref{eq:remainder-and-distance-estimates}.
\end{proof}

\section{Proof of the lower bound II: a refined expansion}
\label{sec:lower-bound-refined}

Throughout this section we write
\begin{equation*}
\theta:=\frac{N-2s}{N+2s-\mu},\qquad
D_0:=\SH^{\frac{2N-\mu}{N+2s-\mu}},\qquad
c_0:=a_{N,s}A_{N,\mu,s}^{-\frac{2s}{N+2s-\mu}},
\end{equation*}
so that $(-\Delta)^s\bar U_{z,\lambda}=c_0\bar U_{z,\lambda}^{2_s^*-1}$ and $D_0=\|\ds\bar U_{0,1}\|_2^2=\int_{\R^N}\int_{\R^N}\frac{\bar U_{0,1}^{\twostar}(x)\bar U_{0,1}^{\twostar}(y)}{|x-y|^\mu}dxdy$. We also set
\begin{equation*}
c_1:=\kappa_{N,\mu,s}^{2}a_{N,s}\mathfrak a_{N,s}^{(1)}\SH^{-\theta},\qquad
c_2:=\kappa_{N,\mu,s}^{2}\left(a_{N,s}d_{N,s}b_{N,s}+\alpha_{N,s}\right)\SH^{-\theta},\qquad
c_3:=\kappa_{N,\mu,s}^{2}\SH^{-\theta},
\end{equation*}
which are positive constants, so that \eqref{eq:quotient-expansion} reads
\begin{equation}
\label{eq:quotient-expansion-short}
\begin{split}
\SH(a+\varepsilon V)[\psi_{z,\lambda}]
&=\SH+c_1\phi_a(z)\lambda^{-(N-2s)}+\mathcal T_1(\phi_a(z),\lambda)-c_2a(z)\lambda^{-2s}+c_3\varepsilon Q_V(z)\lambda^{-(N-2s)}
+o\left(\lambda^{-2s}\right)\\&\quad+o\left(\varepsilon\lambda^{-(N-2s)}\right).
\end{split}
\end{equation}
For $v\in\ths$ we abbreviate
\begin{equation*}
\mathcal Q_\varepsilon[v]:=\SH(a+\varepsilon V)[v],\qquad
\mathcal N_\varepsilon(v):=\|\ds v\|_2^2+\int_\Omega(a+\varepsilon V)v^2\,dx,\qquad
\mathcal D(v):=\int_\Omega\int_\Omega\frac{|v(x)|^{\twostar}|v(y)|^{\twostar}}{|x-y|^\mu}dxdy,
\end{equation*}
and, with $\Psi_{z,\lambda}$ as in the proof of Proposition~\ref{prop:coercivity-with-potential},
\begin{equation*}
\mathcal E_{z,\lambda}(v):=\|\ds v\|_2^2+\int_\Omega av^2\,dx-\Psi_{z,\lambda}(v).
\end{equation*}
Finally, we put
\begin{equation*}
\sigma:=\min\{1,4s-N\}\in(0,1],\qquad
\gamma_*:=\min\left\{\frac{N+2s}{2},\frac{2N-\mu}{2}\right\},
\end{equation*}
\begin{equation*}
\omega(\lambda):=\lambda^{-(N-2s)-\sigma}+\lambda^{-2(N-2s)}+\lambda^{-\gamma_*}\log\lambda .
\end{equation*}
Note that $\gamma_*>s$ and $N+\sigma>3s$ (recall $N>2s\geq3s-1$ and $N+(4s-N)=4s$). Hence, if $N>\frac52s$, then $\omega(\lambda)^2=o(\lambda^{-2s})$.

\subsection{Derivatives of the quotient at the test function}

Recall the test function $\psi_{z,\lambda}=P\bar U_{z,\lambda}+\kappa_{N,\mu,s}\lambda^{-\frac{N-2s}{2}}h_z$, where
\begin{equation*}
h_z:=H_0(z,\cdot)-H_a(z,\cdot)=G_a(z,\cdot)-G_0(z,\cdot)\quad\text{in }\Omega,\qquad h_z:=0\quad\text{in }\R^N\setminus\Omega.
\end{equation*}

\begin{lem}[Properties of $h_z$]
\label{lem:hz}
Let $a\in C(\overline\Omega)$ be such that $(-\Delta)^s+a$ is coercive. Then, for every $z\in\Omega$, $h_z\in\ths\cap L^\infty(\Omega)$ and
\begin{equation*}
(-\Delta)^sh_z=-a\,G_a(z,\cdot)\quad\text{in }\Omega,
\end{equation*}
with $\|\ds h_z\|_2+\|h_z\|_{L^\infty(\Omega)}$ bounded uniformly for $z$ in compact subsets of $\Omega$. Moreover, $h_z(x)-h_z(z)=o(1)$ as $x\to z$, and if in addition $a\in C^1(\Omega)$, then $|h_z(x)-h_z(z)|\lesssim|x-z|^\sigma$ for $x\in\Omega$, uniformly for $z$ in compact subsets of $\Omega$.
\end{lem}

\begin{proof}
The equation follows from $(-\Delta)^s_xG_a(z,\cdot)+aG_a(z,\cdot)=\gamma_{N,s}\delta_z=(-\Delta)^s_xG_0(z,\cdot)$ in $\Omega$ and the symmetry of $G_a$ and $G_0$. Since $0\leq G_a(z,x)\lesssim|x-z|^{-(N-2s)}$ and $(N-2s)\frac{2N}{N+2s}<N$ (because $N<6s$), the right-hand side belongs to $L^{\frac{2N}{N+2s}}(\Omega)$, uniformly for $z$ in compact sets; hence $h_z\in\ths$ with uniformly bounded norm. Boundedness and the modulus of continuity at $x=z$ follow from Remark~\ref{rem:continuous-a}, respectively from Lemma~\ref{lemma Ha expansion} and the smoothness of $H_0(z,\cdot)$ near $z$; away from $z$, $h_z$ is bounded because $|h_z(x)|\leq|x-z|^{-(N-2s)}+|H_a(z,x)|$.
\end{proof}

\begin{lem}[First derivative]
\label{lem:first-derivative}
Let $a\in C(\overline\Omega)\cap C^1(\Omega)$ be such that $(-\Delta)^s+a$ is coercive, and let $V\in L^\infty(\Omega)$. As $\lambda\to\infty$, uniformly for $z$ in compact subsets of $\Omega$, $\varepsilon\in[0,1]$ and $\eta\in\ths$,
\begin{equation}
\label{eq:first-derivative}
\left|D\mathcal Q_\varepsilon[\psi_{z,\lambda}]\eta\right|
\lesssim
\left(|\phi_a(z)|\lambda^{-(N-2s)}+\omega(\lambda)+\varepsilon\lambda^{-\frac{N-2s}{2}}\right)\|\ds\eta\|_2 .
\end{equation}
If $a\in C(\overline\Omega)$ only, the same holds with $\omega(\lambda)$ replaced by a function $\omega_0(\lambda)=o(\lambda^{-(N-2s)})$.
\end{lem}

\begin{proof}
We write $\psi=\psi_{z,\lambda}$, $\bar U=\bar U_{z,\lambda}$ and $\|\cdot\|:=\|\ds\cdot\|_2$.

\emph{Step 1: the equation for $\psi$.} Since $(-\Delta)^sP\bar U=(-\Delta)^s\bar U=c_0\bar U^{2_s^*-1}$ in $\Omega$, Lemma~\ref{lem:hz} gives
\begin{equation*}
(-\Delta)^s\psi+a\psi=c_0\bar U^{2_s^*-1}+a\left(\psi-\kappa_{N,\mu,s}\lambda^{-\frac{N-2s}{2}}G_a(z,\cdot)\right)
\quad\text{in }\Omega .
\end{equation*}
By Lemma~\ref{lem:projection-barU}\,(ii),
\begin{equation*}
g:=\psi-\kappa_{N,\mu,s}\lambda^{-\frac{N-2s}{2}}G_a(z,\cdot)
=\bar U-\kappa_{N,\mu,s}\lambda^{-\frac{N-2s}{2}}|\cdot-z|^{-(N-2s)}-f_{z,\lambda}
\quad\text{in }\Omega .
\end{equation*}
Since $0\leq\kappa_{N,\mu,s}\lambda^{-\frac{N-2s}{2}}|x-z|^{-(N-2s)}-\bar U(x)\lesssim\lambda^{-\frac{N-2s}{2}}|x-z|^{-(N-2s)}\min\{1,(\lambda|x-z|)^{-2}\}$, a direct computation (splitting $\{|x-z|<\lambda^{-1}\}$ and its complement) gives $\|g\|_{L^{\frac{2N}{N+2s}}(\Omega)}\lesssim\lambda^{-2s}$; here we used $\|f_{z,\lambda}\|_{L^\infty(\Omega)}\lesssim\lambda^{-\frac{N+4-2s}{2}}=o(\lambda^{-2s})$. Hence
\begin{equation}
\label{eq:equation-psi}
\left|\int_\Omega\left((-\Delta)^s\psi+a\psi\right)\eta\,dx-c_0\int_\Omega\bar U^{2_s^*-1}\eta\,dx\right|\lesssim\lambda^{-2s}\|\eta\|.
\end{equation}

\emph{Step 2: the formula for $D\mathcal Q_\varepsilon$.} We have
\begin{equation*}
D\mathcal Q_\varepsilon[\psi]\eta
=2\mathcal D(\psi)^{-\frac{1}{\twostar}}
\left[\int_\Omega\left((-\Delta)^s\psi+a\psi\right)\eta\,dx+\varepsilon\int_\Omega V\psi\eta\,dx-\frac{\mathcal N_\varepsilon(\psi)}{\mathcal D(\psi)}\mathcal J(\eta)\right],
\end{equation*}
where
\begin{equation*}
\mathcal J(\eta):=\int_\Omega\int_\Omega\frac{|\psi(x)|^{\twostar}|\psi(y)|^{\twostar-2}\psi(y)\eta(y)}{|x-y|^\mu}dxdy .
\end{equation*}
The projection estimates, Lemma~\ref{lem:hz} and the HLS inequality give $\|\psi\|\lesssim1$ and $\mathcal D(\psi)=D_0+o(1)$.
Moreover, by Lemma~\ref{lem:barU-Lq-estimates} with $q=2<\frac{N}{N-2s}$, $\|\psi\|_{L^2(\Omega)}\lesssim\lambda^{-\frac{N-2s}{2}}$, so that $\varepsilon|\int_\Omega V\psi\eta|\lesssim\varepsilon\lambda^{-\frac{N-2s}{2}}\|\eta\|$.

\emph{Step 3: expansion of $\mathcal J(\eta)$.} Write $\psi=\bar U-b-f_{z,\lambda}$ with $b:=\kappa_{N,\mu,s}\lambda^{-\frac{N-2s}{2}}H_a(z,\cdot)$, so that $\|b\|_{L^\infty(\Omega)}\lesssim\lambda^{-\frac{N-2s}{2}}$. We claim that
\begin{equation*}
\begin{split}
\mathcal J(\eta)
=&\int_\Omega\int_\Omega\frac{\bar U^{\twostar}(x)\bar U^{\twostar-1}(y)\eta(y)}{|x-y|^\mu}dxdy
-\twostar\int_\Omega\int_\Omega\frac{\bar U^{\twostar-1}(x)b(x)\bar U^{\twostar-1}(y)\eta(y)}{|x-y|^\mu}dxdy\\
&-(\twostar-1)\int_\Omega\int_\Omega\frac{\bar U^{\twostar}(x)\bar U^{\twostar-2}(y)b(y)\eta(y)}{|x-y|^\mu}dxdy
+O\left(\left(\lambda^{-2(N-2s)}+\lambda^{-\frac{N+2s}{2}}\log\lambda\right)\|\eta\|\right).
\end{split}
\end{equation*}
Here the remainder contains the terms linear in $f_{z,\lambda}$ and the terms at least quadratic in $b+f_{z,\lambda}$. For the terms linear in $f_{z,\lambda}$, \eqref{easybarU}, the HLS and H\"older inequalities and Lemma~\ref{lem:barU-Lq-estimates} give
\begin{equation*}
\begin{split}
&\twostar\left|\int_\Omega\int_\Omega\frac{\bar U^{\twostar-1}(x)f_{z,\lambda}(x)\bar U^{\twostar-1}(y)\eta(y)}{|x-y|^\mu}dxdy\right|+(\twostar-1)\left|\int_\Omega\int_\Omega\frac{\bar U^{\twostar}(x)\bar U^{\twostar-2}(y)f_{z,\lambda}(y)\eta(y)}{|x-y|^\mu}dxdy\right|\\
\lesssim&\|f_{z,\lambda}\|_{L^\infty(\Omega)}
\left(\|\bar U\|_{L^{\frac{2N(\twostar-1)}{2N-\mu}}(\Omega)}^{\twostar-1}
+\|\bar U^{2_s^*-2}\|_{L^{\frac{2N}{N+2s}}(\Omega)}\right)\|\eta\|\lesssim\lambda^{-(N-2s)-2}\|\eta\|
=o\left(\lambda^{-2(N-2s)}\|\eta\|\right),
\end{split}
\end{equation*}
where we used $\mu<N<4s$, $N<6s$, $\|f_{z,\lambda}\|_{L^\infty(\Omega)}\lesssim\lambda^{-\frac{N-2s}{2}-2}$ and $N-2s<2$.

We next estimate the terms at least quadratic in $b+f_{z,\lambda}$. Since $\bar U\gtrsim\lambda^{-\frac{N-2s}{2}}$ on $\Omega$, we have $|b+f_{z,\lambda}|\leq C_0\bar U$ for some $C_0>0$. Taylor's formula gives, for $A=\bar U$ and $B=b+f_{z,\lambda}$,
\begin{equation*}
\left||A-B|^{\twostar}-A^{\twostar}+\twostar A^{\twostar-1}B\right|\lesssim A^{\twostar-2}B^2,
\end{equation*}
\begin{equation*}
\left||A-B|^{\twostar-2}(A-B)-A^{\twostar-1}+(\twostar-1)A^{\twostar-2}B\right|\lesssim A^{\twostar-3}B^2.
\end{equation*}
Both estimates hold whenever $A>0$ and $|B|\leq C_0A$. For $|B|\leq A/2$ they follow from Taylor's formula, also when $2<\twostar<3$. For $A/2<|B|\leq C_0A$, the left-hand sides are bounded by a constant times $A^{\twostar}$ and $A^{\twostar-1}$, respectively, and these are comparable to $A^{\twostar-2}B^2$ and $A^{\twostar-3}B^2$, since $A$ and $|B|$ are comparable. Using $|B|\lesssim\lambda^{-\frac{N-2s}{2}}$ and $|B|\leq C_0\bar U$, the terms at least quadratic in $B$ are bounded by
\begin{equation*}
\lambda^{-(N-2s)}\int_\Omega\int_\Omega\frac{\bar U^{\twostar}(x)\bar U^{\twostar-3}(y)|\eta(y)|+\bar U^{\twostar-2}(x)\bar U^{\twostar-1}(y)|\eta(y)|+\bar U^{\twostar-1}(x)\bar U^{\twostar-2}(y)|\eta(y)|}{|x-y|^\mu}dxdy .
\end{equation*}
By \eqref{easybarU} and H\"older's inequality, the contribution of the first integral is bounded by $\lambda^{-(N-2s)}\|\bar U^{2_s^*-3}\|_{L^{\frac{2N}{N+2s}}(\Omega)}\|\eta\|$. Since $(N-2s)(2_s^*-3)=6s-N$, Lemma~\ref{lem:barU-Lq-estimates} gives
\begin{equation*}
\lambda^{-(N-2s)}\left\|\bar U^{2_s^*-3}\right\|_{L^{\frac{2N}{N+2s}}(\Omega)}\lesssim
\begin{cases}
\lambda^{-2(N-2s)},&3N<10s,\\
\lambda^{-2(N-2s)}(\log\lambda)^{\frac{N+2s}{2N}},&3N=10s,\\
\lambda^{-\frac{N+2s}{2}},&3N>10s.
\end{cases}
\end{equation*}
For the second integral, take
\begin{equation*}
r=\frac{N^2}{N^2-2\mu(N-2s)},\qquad
t=\frac{N^2}{N^2-\mu(4s-N)}.
\end{equation*}
For every $0<\mu<N$, these exponents satisfy $r,t>1$, $1/r+1/t+\mu/N=2$ and $t<2_s^*$. By the HLS and H\"older inequalities and the Sobolev embedding, we obtain
\begin{equation*}
\begin{split}
&\lambda^{-(N-2s)}\int_\Omega\int_\Omega\frac{\bar U^{\twostar-2}(x)\bar U^{\twostar-1}(y)|\eta(y)|}{|x-y|^\mu}dxdy\\
\lesssim &\lambda^{-(N-2s)}\|\bar U^{\twostar-2}\|_{L^r(\Omega)}\|\bar U^{\twostar-1}\eta\|_{L^t(\Omega)}\\
\lesssim & \lambda^{-(N-2s)}\|\bar U^{\twostar-2}\|_{L^r(\Omega)}\|\bar U^{\twostar-1}\|_{L^{\frac{2Nt}{2N-(N-2s)t}}(\Omega)}\|\eta\|.
\end{split}
\end{equation*}
Since $\frac{4s-\mu}{N}-\frac1r=\frac{(4s-N)(N-\mu)}{N^2}>0$, we have $(\twostar-2)r>N/(N-2s)$. Applying Lemma~\ref{lem:barU-Lq-estimates} to the two remaining norms, we obtain
\begin{equation*}
\begin{split}
&\lambda^{-(N-2s)}\|\bar U^{\twostar-2}\|_{L^r(\Omega)}
\|\bar U^{\twostar-1}\|_{L^{\frac{2Nt}{2N-(N-2s)t}}(\Omega)}\|\eta\|\\
\lesssim&\left(\lambda^{-2(N-2s)}+\lambda^{-\frac{3(N-2s)}{2}-\frac Nr}\log\lambda\right)\|\eta\|
\lesssim\left(\lambda^{-2(N-2s)}+\lambda^{-\frac{N+2s}{2}}\log\lambda\right)\|\eta\|.
\end{split}
\end{equation*}
For the third integral, take instead
\begin{equation*}
r=\frac{2N^2}{2N^2-\mu(6s-N)},\qquad
t=\frac{2N^2}{2N^2-3\mu(N-2s)}.
\end{equation*}
Again $r,t>1$, $1/r+1/t+\mu/N=2$ and $t<2_s^*$ for every $0<\mu<N$. Applying the HLS inequality first, and then H\"older's inequality and the Sobolev embedding, gives
\begin{equation*}
\begin{split}
&\lambda^{-(N-2s)}\int_\Omega\int_\Omega\frac{\bar U^{\twostar-1}(x)\bar U^{\twostar-2}(y)|\eta(y)|}{|x-y|^\mu}dxdy\\
\lesssim&\lambda^{-(N-2s)}\|\bar U^{\twostar-1}\|_{L^r(\Omega)}\|\bar U^{\twostar-2}\eta\|_{L^t(\Omega)}\\
\lesssim&\lambda^{-(N-2s)}\|\bar U^{\twostar-1}\|_{L^r(\Omega)}
\|\bar U^{\twostar-2}\|_{L^{\frac{2Nt}{2N-(N-2s)t}}(\Omega)}\|\eta\|.
\end{split}
\end{equation*}
Since $\frac{4s-\mu}{N}-\left(\frac1t-\frac{N-2s}{2N}\right)=\frac{(6s-N)(N-\mu)}{2N^2}>0$, we have $(\twostar-2)\frac{2Nt}{2N-(N-2s)t}>\frac{N}{N-2s}$. Lemma~\ref{lem:barU-Lq-estimates} therefore gives
\begin{equation*}
\begin{split}
&\lambda^{-(N-2s)}\|\bar U^{\twostar-1}\|_{L^r(\Omega)}
\|\bar U^{\twostar-2}\|_{L^{\frac{2Nt}{2N-(N-2s)t}}(\Omega)}\|\eta\|\\
\lesssim&\left(\lambda^{-2(N-2s)}+\lambda^{-(N-2s)-\frac Nt}\log\lambda\right)\|\eta\|
\lesssim\left(\lambda^{-2(N-2s)}+\lambda^{-\frac{N+2s}{2}}\log\lambda\right)\|\eta\|.
\end{split}
\end{equation*}
This establishes the claimed expansion of $\mathcal J(\eta)$.

For the first term, \eqref{easybarU} gives
\begin{equation*}
\int_\Omega\int_\Omega\frac{\bar U^{\twostar}(x)\bar U^{\twostar-1}(y)\eta(y)}{|x-y|^\mu}dxdy
=c_0\int_\Omega\bar U^{2_s^*-1}\eta\,dy-\int_\Omega W(y)\bar U^{\twostar-1}(y)\eta(y)\,dy,
\qquad
W(y):=\int_{\R^N\setminus\Omega}\frac{\bar U^{\twostar}(x)}{|x-y|^\mu}dx .
\end{equation*}
Since $\bar U^{\twostar}(x)\lesssim\lambda^{-\frac{2N-\mu}{2}}|x-z|^{-(2N-\mu)}$ on $\R^N\setminus\Omega$ and $\mu<N$, we have $\|W\|_{L^\infty(\Omega)}\lesssim\lambda^{-\frac{2N-\mu}{2}}$, and therefore, by H\"older's inequality and Lemma~\ref{lem:barU-Lq-estimates},
\begin{equation*}
\left|\int_\Omega W\bar U^{\twostar-1}\eta\right|
\lesssim\lambda^{-\frac{2N-\mu}{2}}\left\|\bar U^{\twostar-1}\right\|_{L^{\frac{2N}{N+2s}}(\Omega)}\|\eta\|
\lesssim\left(\lambda^{-N}+\lambda^{-\frac{3N+2s-2\mu}{2}}\log\lambda\right)\|\eta\|
=O\left(\lambda^{-\gamma_*}\log\lambda\,\|\eta\|\right).
\end{equation*}

\emph{Step 4: the first order terms in $b$.} Write $H_a(z,x)=\phi_a(z)+R_z(x)$. By Lemma~\ref{lemma Ha expansion}, $|R_z(x)|\lesssim|x-z|^\sigma$ for $x\in\Omega$. For the contribution of $\phi_a(z)$, the HLS inequality with $t=r=\frac{2N}{2N-\mu}$ and Lemma~\ref{lem:barU-Lq-estimates} give
\begin{equation*}
|\phi_a(z)|\lambda^{-\frac{N-2s}{2}}\int_\Omega\int_\Omega\frac{\bar U^{\twostar-1}(x)\bar U^{\twostar-1}(y)|\eta(y)|}{|x-y|^\mu}dxdy
\lesssim|\phi_a(z)|\lambda^{-\frac{N-2s}{2}}\left\|\bar U\right\|_{L^{\frac{2N(\twostar-1)}{2N-\mu}}(\Omega)}^{\twostar-1}\|\eta\|
\lesssim|\phi_a(z)|\lambda^{-(N-2s)}\|\eta\|,
\end{equation*}
because $\frac{2N(\twostar-1)}{2N-\mu}>\frac{N}{N-2s}$ (equivalently, $\mu<4s$). Similarly, \eqref{easybarU} and H\"older's inequality give
\begin{equation*}
|\phi_a(z)|\lambda^{-\frac{N-2s}{2}}\int_\Omega\int_\Omega\frac{\bar U^{\twostar}(x)\bar U^{\twostar-2}(y)|\eta(y)|}{|x-y|^\mu}dxdy
\lesssim|\phi_a(z)|\lambda^{-\frac{N-2s}{2}}\left\|\bar U^{2_s^*-2}\right\|_{L^{\frac{2N}{N+2s}}(\Omega)}\|\eta\|\lesssim|\phi_a(z)|\lambda^{-(N-2s)}\|\eta\|,
\end{equation*}
since $\frac{4s}{N-2s}\cdot\frac{2N}{N+2s}>\frac{N}{N-2s}$ (equivalently, $N<6s$). Hence the contribution of $\phi_a(z)$ to the two first-order terms in Step~3 is bounded by $|\phi_a(z)|\lambda^{-(N-2s)}\|\eta\|$.
For the contribution of $R_z(x)$, take
\begin{equation*}
r=\frac{N^2}{N^2-2s\mu},\qquad
t=\frac{N^2}{N^2-(N-2s)\mu}.
\end{equation*}
These exponents satisfy $r,t>1$, $1/r+1/t+\mu/N=2$ and $t<2_s^*$. By the HLS and H\"older inequalities and the Sobolev embedding, we obtain
\begin{equation*}
\begin{split}
&\lambda^{-\frac{N-2s}{2}}\int_\Omega\int_\Omega
\frac{\bar U^{\twostar-1}(x)|R_z(x)|\bar U^{\twostar-1}(y)|\eta(y)|}{|x-y|^\mu}dxdy\\
\lesssim&\lambda^{-\frac{N-2s}{2}}
\left\||\cdot-z|^\sigma\bar U^{\twostar-1}\right\|_{L^r(\Omega)}
\|\bar U^{\twostar-1}\eta\|_{L^t(\Omega)}\\
\lesssim&\lambda^{-\frac{N-2s}{2}}
\left\||\cdot-z|^\sigma\bar U^{\twostar-1}\right\|_{L^r(\Omega)}
\|\bar U^{\twostar-1}\|_{L^{\frac{2Nt}{2N-(N-2s)t}}(\Omega)}\|\eta\|.
\end{split}
\end{equation*}
To estimate these two norms, we first check the integrability conditions after rescaling. Since $\sigma\leq4s-N$, we have
\begin{equation*}
\frac{N+2s-\mu-\sigma}{N}-\frac1r
=\frac{N(2s-\sigma)-\mu(N-2s)}{N^2}
\geq\frac{(N-2s)(N-\mu)}{N^2}>0,
\end{equation*}
\begin{equation*}
\frac{N+2s-\mu}{N}-\left(\frac1t-\frac{N-2s}{2N}\right)
=\frac{N+2s}{2N}-\frac{2s\mu}{N^2}
>\frac{N-2s}{2N}>0.
\end{equation*}
The first inequality gives $(N+2s-\mu-\sigma)r>N$, which ensures that the weighted integral after rescaling converges on $\R^N$. Thus, making the change of variables $X=\lambda(x-z)$, we obtain
\begin{equation*}
\left\||\cdot-z|^\sigma\bar U^{\twostar-1}\right\|_{L^r(\Omega)}
\lesssim\lambda^{\frac{N+2s-\mu}{2}-\sigma-\frac Nr}.
\end{equation*}
The second inequality gives $(\twostar-1)\frac{2Nt}{2N-(N-2s)t}>\frac{N}{N-2s}$, so Lemma~\ref{lem:barU-Lq-estimates} yields
\begin{equation*}
\|\bar U^{\twostar-1}\|_{L^{\frac{2Nt}{2N-(N-2s)t}}(\Omega)}
\lesssim\lambda^{\frac{N+2s-\mu}{2}-\frac Nt+\frac{N-2s}{2}}.
\end{equation*}
Since $N/r+N/t=2N-\mu$, these estimates bound the contribution of $R_z(x)$ by $C\lambda^{-(N-2s)-\sigma}\|\eta\|$. For the contribution of $R_z(y)$, \eqref{easybarU} and H\"older's inequality give
\begin{equation*}
\begin{split}
&\lambda^{-\frac{N-2s}{2}}\int_\Omega\int_\Omega
\frac{\bar U^{\twostar}(x)\bar U^{\twostar-2}(y)|R_z(y)||\eta(y)|}{|x-y|^\mu}dxdy\\
\lesssim&\lambda^{-\frac{N-2s}{2}}
\left\||\cdot-z|^\sigma\bar U^{2_s^*-2}\right\|_{L^{\frac{2N}{N+2s}}(\Omega)}\|\eta\|\\
\lesssim&\lambda^{-(N-2s)-\sigma}\|\eta\|.
\end{split}
\end{equation*}
The last inequality follows by the same change of variables, since $(4s-\sigma)\frac{2N}{N+2s}>N$. Hence the contribution of $R_z$ is $O(\lambda^{-(N-2s)-\sigma}\|\eta\|)$.

If $a$ is only continuous, $R_z(x)=o(1)$ as $x\to z$ uniformly (Remark~\ref{rem:continuous-a}). Given $\delta>0$, choose $r_\delta>0$ such that $|R_z(x)|\leq\delta$ whenever $|x-z|<r_\delta$, uniformly for $z$ in a fixed compact subset of $\Omega$. The uniform boundedness of $R_z$ then gives $|R_z(x)|\leq\delta+Cr_\delta^{-\sigma}|x-z|^\sigma$ on $\Omega$. The unweighted and weighted bubble estimates above bound the two contributions by
\begin{equation*}
C\left(\delta\lambda^{-(N-2s)}+r_\delta^{-\sigma}\lambda^{-(N-2s)-\sigma}\right)\|\eta\|.
\end{equation*}
Letting $\lambda\to\infty$ first and then $\delta\to0$ gives $o(\lambda^{-(N-2s)}\|\eta\|)$.

\emph{Step 5: conclusion.} Assume first that $a\in C(\overline\Omega)\cap C^1(\Omega)$. By \eqref{eq:numerator-expansion} and \eqref{eq:denominator-expansion}, we have
\begin{equation}
\label{eq:ratio-N-D}
\frac{\mathcal N_\varepsilon(\psi)}{\mathcal D(\psi)}=1+O\left(|\phi_a(z)|\lambda^{-(N-2s)}+\lambda^{-2s}+\varepsilon\lambda^{-(N-2s)}\right).
\end{equation}
Combining Steps 1--4, and using that $|\int_\Omega\bar U^{2_s^*-1}\eta|\lesssim\|\eta\|$, the terms $c_0\int_\Omega\bar U^{2_s^*-1}\eta$ cancel up to the factor in \eqref{eq:ratio-N-D}, and we obtain \eqref{eq:first-derivative}.

If $a\in C(\overline\Omega)$ only, we use the identity $\mathcal J(\psi)=\mathcal D(\psi)$. Taking $\eta=\psi$ in \eqref{eq:equation-psi} and in the estimates of Steps~3--4, with the continuity argument in Step~4, yields $\mathcal N_0(\psi)-\mathcal D(\psi)=O(|\phi_a(z)|\lambda^{-(N-2s)})+o(\lambda^{-(N-2s)})$, since $N<4s$. Adding $\varepsilon\int_\Omega V\psi^2\,dx=O(\varepsilon\lambda^{-(N-2s)})$ and dividing by $\mathcal D(\psi)$, we obtain
\begin{equation}\label{eq:ratio-N-D-only-C}
\frac{\mathcal N_\varepsilon(\psi)}{\mathcal D(\psi)}
=1+O\left(|\phi_a(z)|\lambda^{-(N-2s)}\right)
+o\left(\lambda^{-(N-2s)}\right)+O\left(\varepsilon\lambda^{-(N-2s)}\right),
\end{equation}
uniformly for $z$ in compact subsets of $\Omega$ and $\varepsilon\in[0,1]$. Using this estimate in the formula of Step~2 and combining the estimates above, we obtain the conclusion with $\omega_0(\lambda)=o(\lambda^{-(N-2s)})$.
\end{proof}

\begin{lem}[Second order expansion]
\label{lem:second-order}
Let $a\in C(\overline\Omega)$ be such that $(-\Delta)^s+a$ is coercive, and let $V\in L^\infty(\Omega)$. As $\lambda\to\infty$, $\varepsilon\to0^+$ and $\|\ds\eta\|_2\to0$, uniformly for $z$ in compact subsets of $\Omega$,
\begin{equation*}
\mathcal Q_\varepsilon[\psi_{z,\lambda}+\eta]
=\mathcal Q_\varepsilon[\psi_{z,\lambda}]+D\mathcal Q_\varepsilon[\psi_{z,\lambda}]\eta
+\SH^{-\theta}\mathcal E_{z,\lambda}(\eta)
+O\left((\bar U_{z,\lambda},\eta)^2\right)+o\left(\|\ds\eta\|_2^2\right),
\end{equation*}
where $(\bar U_{z,\lambda},\eta)=\int_{\R^N}\ds\bar U_{z,\lambda}\ds\eta\,dx=(P\bar U_{z,\lambda},\eta)$.
\end{lem}

\begin{proof}
Since $\twostar>2$, the map $v\mapsto\mathcal D(v)$ is of class $C^{2,\beta}$ on $L^{2_s^*}(\Omega)$ with $\beta:=\min\{1,\twostar-2\}$, and $\mathcal D(\psi)=D_0+o(1)$ is bounded away from zero. Hence Taylor's formula gives $\mathcal Q_\varepsilon[\psi+\eta]=\mathcal Q_\varepsilon[\psi]+D\mathcal Q_\varepsilon[\psi]\eta+\frac12D^2\mathcal Q_\varepsilon[\psi](\eta,\eta)+O(\|\eta\|^{2+\beta})$. Writing $\mathcal Q_\varepsilon=\mathcal N_\varepsilon\mathcal D^{-1/\twostar}$, we have
\begin{equation*}
\begin{split}
\frac12D^2\mathcal Q_\varepsilon[\psi](\eta,\eta)
=&\mathcal D(\psi)^{-\frac{1}{\twostar}}\left[\frac12D^2\mathcal N_\varepsilon(\psi)(\eta,\eta)-\frac{\mathcal N_\varepsilon(\psi)}{2\twostar\mathcal D(\psi)}D^2\mathcal D(\psi)(\eta,\eta)\right]\\
&-\frac{1}{\twostar}\mathcal D(\psi)^{-\frac{1}{\twostar}-1}D\mathcal N_\varepsilon(\psi)\eta\,D\mathcal D(\psi)\eta
+\frac{\twostar+1}{2(\twostar)^2}\mathcal N_\varepsilon(\psi)\mathcal D(\psi)^{-\frac{1}{\twostar}-2}\left(D\mathcal D(\psi)\eta\right)^2 .
\end{split}
\end{equation*}
By Steps 1--4 of the proof of Lemma~\ref{lem:first-derivative} (which use only the continuity of $H_a$ when the error terms are not quantified), $D\mathcal N_\varepsilon(\psi)\eta=2(\bar U,\eta)+o(\|\eta\|)$ and $D\mathcal D(\psi)\eta=2\twostar(\bar U,\eta)+o(\|\eta\|)$, so the last two terms are $O((\bar U,\eta)^2)+o(\|\eta\|^2)$. Next, $\frac12D^2\mathcal N_\varepsilon(\psi)(\eta,\eta)=\|\eta\|^2+\int_\Omega a\eta^2+\varepsilon\int_\Omega V\eta^2$, and $\varepsilon\int_\Omega V\eta^2=o(\|\eta\|^2)$ since $\varepsilon\to0$. Moreover,
\begin{small}\begin{equation*}
\frac{1}{2\twostar}D^2\mathcal D(\psi)(\eta,\eta)=(\twostar-1)\int_\Omega\int_\Omega\frac{|\psi(x)|^{\twostar}|\psi(y)|^{\twostar-2}\eta(y)^2}{|x-y|^\mu}dxdy+\twostar\int_\Omega\int_\Omega\frac{(|\psi|^{\twostar-2}\psi\eta)(x)(|\psi|^{\twostar-2}\psi\eta)(y)}{|x-y|^\mu}dxdy,
\end{equation*}\end{small}
and, since $\|\psi-\bar U\|_{L^{2_s^*}(\Omega)}\lesssim\lambda^{-\frac{N-2s}{2}}$, the HLS and H\"older inequalities allow us to replace $\psi$ by $\bar U$ at the cost of $o(\|\eta\|^2)$. Finally, $\mathcal N_\varepsilon(\psi)/\mathcal D(\psi)=1+o(1)$ by \eqref{eq:ratio-N-D-only-C}, and $\mathcal D(\psi)^{-1/\twostar}=D_0^{-1/\twostar}+o(1)=\SH^{-\theta}+o(1)$. This proves the lemma.
\end{proof}

\subsection{The tangential part of $h_z$}

Let $e_0:=P\bar U_{z,\lambda}$, $e_{N+1}:=\lambda\partial_\lambda P\bar U_{z,\lambda}$ and $e_i:=\lambda^{-1}\partial_{z_i}P\bar U_{z,\lambda}$, $i=1,\ldots,N$, so that $T_{z,\lambda}=\operatorname{span}\{e_0,\ldots,e_{N+1}\}$.

\begin{lem}
\label{lem:tangential-h}
Let $a\in C(\overline\Omega)$ be such that $(-\Delta)^s+a$ is coercive. As $\lambda\to\infty$ with $z$ in a compact subset of $\Omega$,
\begin{equation*}
-\kappa_{N,\mu,s}\lambda^{-\frac{N-2s}{2}}\Pi_{z,\lambda}h_z=\sum_{j=0}^{N+1}b_je_j,
\qquad
|b_0|+|b_{N+1}|=O\left(\lambda^{-(N-2s)}\right),
\qquad
\sum_{i=1}^N|b_i|=o\left(\lambda^{-(N-2s)}\right).
\end{equation*}
If in addition $a\in C^1(\Omega)$, then $\sum_{i=1}^N|b_i|=O(\lambda^{-(N-2s)-\sigma})$.
\end{lem}

\begin{proof}
The Gram matrix $M:=((e_j,e_k))_{j,k}$ converges, as $\lambda\to\infty$, to the corresponding Gram matrix of $\bar U_{0,1}$, $\lambda\partial_\lambda\bar U_{0,\lambda}|_{\lambda=1}$, $\partial_{z_i}\bar U_{0,1}$ in $\dot H^s(\R^N)$, which is diagonal and positive definite by translation and dilation invariance and parity (see \cite[Lemma~B.2]{DeNittiKonig2023}). Hence $M^{-1}$ is uniformly bounded and $(b_j)=-\kappa_{N,\mu,s}\lambda^{-\frac{N-2s}{2}}M^{-1}((h_z,e_k))_k$. Moreover, \cite[Lemma~B.2]{DeNittiKonig2023}  gives
\begin{equation*}
|M_{i0}|+|M_{i,N+1}|\lesssim\lambda^{-(N-2s)-1},
\qquad i=1,\ldots,N.
\end{equation*}
Grouping the indices $\{0,N+1\}$ and $\{1,\ldots,N\}$, both diagonal blocks converge to positive definite matrices. The block inversion formula therefore yields
\begin{equation*}
|(M^{-1})_{i0}|+|(M^{-1})_{i,N+1}|
\lesssim\lambda^{-(N-2s)-1},
\qquad i=1,\ldots,N.
\end{equation*}

Since $h_z\in\ths$, the definition of $P$ gives
\begin{equation*}
(h_z,e_0)=c_0\int_\Omega\bar U^{2_s^*-1}h_z,\qquad
(h_z,e_{N+1})=c_0\int_\Omega\lambda\partial_\lambda\left(\bar U^{2_s^*-1}\right)h_z,\qquad
(h_z,e_i)=c_0\lambda^{-1}\int_\Omega\partial_{z_i}\left(\bar U^{2_s^*-1}\right)h_z .
\end{equation*}
As $h_z$ is bounded, $|\lambda\partial_\lambda\bar U^{2_s^*-1}|\lesssim\bar U^{2_s^*-1}$ and $\int_\Omega\bar U^{2_s^*-1}\lesssim\lambda^{-\frac{N-2s}{2}}$, the first two quantities are $O(\lambda^{-\frac{N-2s}{2}})$. For the last one, since $\int_{\R^N}\partial_{z_i}\bar U^{2_s^*-1}=0$,
\begin{equation*}
\lambda^{-1}\int_\Omega\partial_{z_i}\left(\bar U^{2_s^*-1}\right)h_z
=\lambda^{-1}\int_\Omega\partial_{z_i}\left(\bar U^{2_s^*-1}\right)\left(h_z-h_z(z)\right)
-\lambda^{-1}h_z(z)\int_{\R^N\setminus\Omega}\partial_{z_i}\left(\bar U^{2_s^*-1}\right).
\end{equation*}
The last term is $O(\lambda^{-\frac{N+2s}{2}})$. Since $|\partial_{z_i}\bar U^{2_s^*-1}|\lesssim\lambda^{\frac{N+2s}{2}+1}(1+\lambda|x-z|)^{-(N+2s+1)}$, the change of variables $y=\lambda(x-z)$ and dominated convergence give $o(\lambda^{-\frac{N-2s}{2}})$ for the first term, uniformly for $z$ in compact subsets of $\Omega$, by Lemma~\ref{lem:hz} and Remark~\ref{rem:continuous-a}. If $a\in C^1(\Omega)$, Lemma~\ref{lem:hz} gives $|h_z(x)-h_z(z)|\lesssim|x-z|^\sigma$, and hence
\begin{equation*}
\lambda^{-1}\int_\Omega
\left|\partial_{z_i}\left(\bar U^{2_s^*-1}\right)\right|
|h_z(x)-h_z(z)|\,dx\lesssim
\lambda^{-\frac{N-2s}{2}-\sigma}
\int_{\R^N}\frac{|y|^\sigma}{(1+|y|)^{N+2s+1}}\,dy
\lesssim\lambda^{-\frac{N-2s}{2}-\sigma}.
\end{equation*}
where the integral is finite since $\sigma\leq1<2s+1$. Since $\sigma\leq4s-N<2s$, this bound also absorbs the exterior term.

Finally, the uniform bound for $M^{-1}$ gives $|b_0|+|b_{N+1}|=O(\lambda^{-(N-2s)})$. For the translation coefficients, the bounds on the mixed entries of $M^{-1}$ and on $(h_z,e_0)$ and $(h_z,e_{N+1})$ imply
\begin{equation*}
\sum_{i=1}^N|b_i|
\lesssim\lambda^{-\frac{N-2s}{2}}\sum_{i=1}^N|(h_z,e_i)|
+\lambda^{-2(N-2s)-1}.
\end{equation*}
The right-hand side is $o(\lambda^{-(N-2s)})$ when $a$ is only continuous, and $O(\lambda^{-(N-2s)-\sigma})$ if $a\in C^1(\Omega)$. The claim follows.

\end{proof}

\subsection{The refined lower bound}

\begin{Prop}[Refined lower bound]
\label{prop:refined-lower-bound}
Let $a\in C(\overline\Omega)$ satisfy $\SH(a)=\SH$, let $V\in L^\infty(\Omega)$, and let $\varepsilon\to0^+$ and $(u_\varepsilon)$ be as in Proposition~\ref{prop:profile-decomposition}. Write $u_\varepsilon=\alpha_\varepsilon(P\bar U_{z_\varepsilon,\lambda_\varepsilon}+w_\varepsilon)$ as in \eqref{eq:profile-decomposition}. Let $b_{j,\varepsilon}$ be the coefficients of Lemma~\ref{lem:tangential-h} at $(z_\varepsilon,\lambda_\varepsilon)$. For $\varepsilon$ small, set
\begin{equation*}
\lambda_\varepsilon':=\lambda_\varepsilon\left(1+\frac{b_{N+1,\varepsilon}}{1+b_{0,\varepsilon}}\right),
\qquad
\beta_\varepsilon:=(1+b_{0,\varepsilon})\left(\frac{\lambda_\varepsilon'}{\lambda_\varepsilon}\right)^{-\frac{N-2s}{2}}.
\end{equation*}

Define
\begin{equation*}
q_\varepsilon:=w_\varepsilon-\kappa_{N,\mu,s}\beta_\varepsilon\lambda_\varepsilon^{-\frac{N-2s}{2}}\Pi_{z_\varepsilon,\lambda_\varepsilon}^\perp h_{z_\varepsilon}\in T_{z_\varepsilon,\lambda_\varepsilon}^\perp .
\end{equation*}
We have $\beta_\varepsilon=1+O(\lambda_\varepsilon^{-(N-2s)})$. Then $\|\ds q_\varepsilon\|_2\to0$ and, with $\rho$ the constant of Proposition~\ref{prop:coercivity-with-potential},
\begin{enumerate}
\item[(a)] if $a\in C(\overline\Omega)$ only, then
\begin{equation}
\label{eq:refined-continuous}
\mathcal Q_\varepsilon[u_\varepsilon]\geq\SH+c_1\phi_a(z_\varepsilon)\lambda_\varepsilon^{-(N-2s)}+\frac{\rho}{4}\SH^{-\theta}\|\ds q_\varepsilon\|_2^2+o\left(\lambda_\varepsilon^{-(N-2s)}\right);
\end{equation}
\item[(b)] if $a\in C(\overline\Omega)\cap C^1(\Omega)$ and $\frac52s<N<4s$, then
\begin{equation}
\label{eq:refined-C1}
\begin{split}
\mathcal Q_\varepsilon[u_\varepsilon]\geq{}&\SH+c_1\phi_a(z_\varepsilon)\lambda_\varepsilon^{-(N-2s)}-c_2a(z_\varepsilon)\lambda_\varepsilon^{-2s}+c_3\varepsilon Q_V(z_\varepsilon)\lambda_\varepsilon^{-(N-2s)}+\frac{\rho}{4}\SH^{-\theta}\|\ds q_\varepsilon\|_2^2\\
&+o\left(\lambda_\varepsilon^{-2s}\right)+o\left(\varepsilon\lambda_\varepsilon^{-(N-2s)}\right)+o\left(|\phi_a(z_\varepsilon)|\lambda_\varepsilon^{-(N-2s)}\right).
\end{split}
\end{equation}
\end{enumerate}
\end{Prop}

\begin{proof}
We drop the index $\varepsilon$ and write $\|\cdot\|=\|\ds\cdot\|_2$. By Propositions~\ref{prop:profile-decomposition} and \ref{prop:remainder-estimates}, $z$ stays in a compact subset of $\Omega$, $\lambda\to\infty$ and $\|w\|\to0$; by Lemma~\ref{lem:hz}, $\|q\|\leq\|w\|+O(\lambda^{-\frac{N-2s}{2}})\to0$. Moreover, $q\in T^\perp_{z,\lambda}$ since both $w$ and $\Pi^\perp_{z,\lambda}h_z$ belong to $T_{z,\lambda}^\perp$.

\emph{Step 1: a reparametrization.} By the definition of $q$,
\begin{equation*}
\begin{split}
\alpha^{-1}u&=P\bar U_{z,\lambda}+\kappa_{N,\mu,s}\beta\lambda^{-\frac{N-2s}{2}}h_z-\kappa_{N,\mu,s}\beta\lambda^{-\frac{N-2s}{2}}\Pi_{z,\lambda}h_z+q\\
&=P\bar U_{z,\lambda}+\kappa_{N,\mu,s}\beta\lambda^{-\frac{N-2s}{2}}h_z+\beta\sum_{j=0}^{N+1}b_je_j+q,
\end{split}
\end{equation*}
with $b_j$ as in Lemma~\ref{lem:tangential-h}. By the choice of $\lambda'$, we have $\lambda'/\lambda=1+O(\lambda^{-(N-2s)})$ and $(1+b_0)(\lambda'-\lambda)\partial_\lambda P\bar U_{z,\lambda}=b_{N+1}e_{N+1}$. Since $\|\partial_\lambda^2P\bar U_{z,\mu}\|\lesssim\mu^{-2}$, Taylor's formula gives
\begin{equation*}
(1+b_0)P\bar U_{z,\lambda'}=(1+b_0)P\bar U_{z,\lambda}+b_{N+1}e_{N+1}+O\left(\lambda^{-2(N-2s)}\right)\quad\text{in }\ths,
\end{equation*}
Set
\begin{equation*}
\begin{split}
R_\lambda&:=(1+b_0)P\bar U_{z,\lambda}+b_{N+1}e_{N+1}-(1+b_0)P\bar U_{z,\lambda'},\\
\varrho&:=R_\lambda+(\beta-1)\sum_{j=0}^{N+1}b_je_j+\sum_{i=1}^Nb_ie_i.
\end{split}
\end{equation*}
The identity $(1+b_0)\lambda'^{-\frac{N-2s}{2}}=\beta\lambda^{-\frac{N-2s}{2}}$, together with $\|R_\lambda\|=O(\lambda^{-2(N-2s)})$, $\beta-1=O(\lambda^{-(N-2s)})$ and $\|\sum_{j=0}^{N+1}b_je_j\|=O(\lambda^{-(N-2s)})$, gives
\begin{equation}
\label{eq:reparametrization}
\alpha^{-1}u=(1+b_0)\left(\psi_{z,\lambda'}+\eta\right),\qquad
\eta:=\frac{q+\varrho}{1+b_0},\qquad
\|\varrho\|\lesssim\lambda^{-2(N-2s)}+\sum_{i=1}^N|b_i| .
\end{equation}
By Lemma~\ref{lem:tangential-h}, $\|\varrho\|=o(\lambda^{-(N-2s)})$ in case (a), while in case (b)
\begin{equation*}
\|\varrho\|\lesssim\lambda^{-(N-2s)-\sigma}+\lambda^{-2(N-2s)},
\qquad\text{hence}\qquad\|\varrho\|^2=o\left(\lambda^{-2s}\right),
\end{equation*}
because $N+\sigma>3s$ and $N>\frac52s$ imply $2(N-2s+\sigma)>2s$ and $4(N-2s)>2s$.
Since $\mathcal Q_\varepsilon$ is homogeneous of degree zero, $\mathcal Q_\varepsilon[u]=\mathcal Q_\varepsilon[\psi_{z,\lambda'}+\eta]$.

\emph{Step 2: expansion around $\psi_{z,\lambda'}$.} By Lemma~\ref{lem:second-order} at $(z,\lambda')$,
\begin{equation*}
\mathcal Q_\varepsilon[u]=\mathcal Q_\varepsilon[\psi_{z,\lambda'}]+D\mathcal Q_\varepsilon[\psi_{z,\lambda'}]\eta+\SH^{-\theta}\mathcal E_{z,\lambda'}(\eta)+O\left((\bar U_{z,\lambda'},\eta)^2\right)+o(\|\eta\|^2).
\end{equation*}
We estimate the four terms.

(i) Since $q\perp e_0$, we have $(\bar U_{z,\lambda'},q)=(\bar U_{z,\lambda'}-\bar U_{z,\lambda},q)=O(\lambda^{-(N-2s)}\|q\|)$, so that $(\bar U_{z,\lambda'},\eta)^2=o(\|q\|^2)+O(\|\varrho\|^2)$.

(ii) Since $|\bar U_{z,\lambda'}^{\gamma}-\bar U_{z,\lambda}^{\gamma}|$ is small in the relevant Lebesgue norms for $\gamma\in\{\twostar-2,\twostar-1,\twostar\}$ (as $\lambda'/\lambda\to1$), we have $\mathcal E_{z,\lambda'}(q)=\mathcal E_{z,\lambda}(q)+o(\|q\|^2)$. Hence, by Proposition~\ref{prop:coercivity-with-potential} and the Cauchy--Schwarz inequality,
\begin{equation*}
\mathcal E_{z,\lambda'}(\eta)\geq(1+b_0)^{-2}\left(\mathcal E_{z,\lambda}(q)-C\|q\|\|\varrho\|-C\|\varrho\|^2\right)+o(\|q\|^2)
\geq\frac{3\rho}{4}\|q\|^2-C\|\varrho\|^2+o(\|q\|^2).
\end{equation*}

(iii) By Lemma~\ref{lem:first-derivative} at $(z,\lambda')$ and Young's inequality, for every $\delta>0$,
\begin{equation*}
\begin{split}
\left|D\mathcal Q_\varepsilon[\psi_{z,\lambda'}]\eta\right|
&\lesssim\left(|\phi_a(z)|\lambda^{-(N-2s)}+\omega(\lambda)+\varepsilon\lambda^{-\frac{N-2s}{2}}\right)\left(\|q\|+\|\varrho\|\right)\\
&\leq\delta\|q\|^2+C_\delta\left(\phi_a(z)^2\lambda^{-2(N-2s)}+\omega(\lambda)^2+\varepsilon^2\lambda^{-(N-2s)}\right)\\
&\quad+C\left(|\phi_a(z)|\lambda^{-(N-2s)}+\omega(\lambda)+\varepsilon\lambda^{-\frac{N-2s}{2}}\right)\|\varrho\| ,
\end{split}
\end{equation*}
with $\omega$ replaced by $\omega_0$ in case (a). In case (b), $\omega(\lambda)^2=o(\lambda^{-2s})$ and Step~1 give $\omega(\lambda)\|\varrho\|\leq\frac12(\omega(\lambda)^2+\|\varrho\|^2)=o(\lambda^{-2s})$ and $\varepsilon\lambda^{-\frac{N-2s}{2}}\|\varrho\|=o(\varepsilon\lambda^{-(N-2s)})$. Since $\phi_a$ is bounded on compact sets, $\|\varrho\|\to0$ and $\varepsilon\to0$, the right-hand side is $\delta\|q\|^2+o(\lambda^{-2s})+o(\varepsilon\lambda^{-(N-2s)})+o(|\phi_a(z)|\lambda^{-(N-2s)})$. In case (a), using $\omega_0(\lambda)=o(\lambda^{-(N-2s)})$ and $\|\varrho\|=o(\lambda^{-(N-2s)})$ gives $\delta\|q\|^2+o(\lambda^{-(N-2s)})$.

(iv) In case (b), \eqref{eq:quotient-expansion-short} at $(z,\lambda')$, together with $\lambda'/\lambda=1+O(\lambda^{-(N-2s)})$ and $\mathcal T_1(\phi,\lambda)=O(\phi^2\lambda^{-2(N-2s)})$ for bounded $\phi$, gives
\begin{equation*}
\begin{split}
\mathcal Q_\varepsilon[\psi_{z,\lambda'}]&=\SH+c_1\phi_a(z)\lambda^{-(N-2s)}-c_2a(z)\lambda^{-2s}+c_3\varepsilon Q_V(z)\lambda^{-(N-2s)}\\&\quad+o\left(\lambda^{-2s}\right)+o\left(\varepsilon\lambda^{-(N-2s)}\right)+o\left(|\phi_a(z)|\lambda^{-(N-2s)}\right).
\end{split}
\end{equation*}
In case (a), \eqref{eq:quotient-expansion-continuous} gives $\mathcal Q_\varepsilon[\psi_{z,\lambda'}]=\SH+c_1\phi_a(z)\lambda^{-(N-2s)}+o(\lambda^{-(N-2s)})$.

Collecting (i)--(iv), choosing $\delta=\frac{\rho}{4}\SH^{-\theta}$, and using $\|\varrho\|^2=o(\lambda^{-2s})$ in case (b), respectively $\|\varrho\|^2=o(\lambda^{-(N-2s)})$ in case (a), we obtain \eqref{eq:refined-C1} and \eqref{eq:refined-continuous}.
\end{proof}

\section{Proof of Theorem~\ref{main} and Corollary~\ref{lem:de-1.3}}
\label{sec:proof-characterization}

\begin{proof}[Proof of (i) $\Rightarrow$ (ii)]
Let $x\in\Omega$ be such that $\phi_a(x)<0$. Take $\varepsilon=0$ and $z=x$ in Theorem~\ref{lem:de-2.1}. Since $N-2s<2s$, the term $\kappa_{N,\mu,s}^2a_{N,s}\mathfrak a^{(1)}_{N,s}\SH^{-\frac{N-2s}{N+2s-\mu}}\phi_a(x)\lambda^{-(N-2s)}$ in \eqref{eq:quotient-expansion} is negative and dominates all the other correction terms, so $\SH(a)\leq\SH(a)[\psi_{x,\lambda}]<\SH$ for $\lambda$ large. (Only the expansion up to $o(\lambda^{-(N-2s)})$ is needed here, which requires merely the continuity of $H_a$; hence the $C^1$ assumption of Theorem~\ref{lem:de-2.1} is not needed for this implication.)
\end{proof}

\begin{proof}[Proof of (iii) $\Rightarrow$ (ii)]
Taking $\varepsilon=0$ and letting $\lambda\to\infty$ in \eqref{eq:quotient-expansion} (again, only the leading term is needed) shows that $\SH(a)\leq\SH$ for every $a$ as in Theorem~\ref{main}. If $\SH(a)$ is achieved, Proposition~\ref{lem:de-3.1} excludes $\SH(a)=\SH$, and therefore $\SH(a)<\SH$.
\end{proof}

\begin{proof}[Proof of (ii) $\Rightarrow$ (iii)]
 Suppose that $\SH(a)<\SH$, and choose a minimizing sequence $(u_\varepsilon)\subset\ths$ normalized by
\begin{equation*}
\int_\Omega\int_\Omega
\frac{|u_\varepsilon(x)|^{\twostar}|u_\varepsilon(y)|^{\twostar}}{|x-y|^\mu}\,dx\,dy=1.
\end{equation*}
Then
\begin{equation*}
\|\ds u_\varepsilon\|_2^2+\int_\Omega au_\varepsilon^2\,dx=\SH(a)+o(1).
\end{equation*}
By coercivity, after passing to a subsequence, $u_\varepsilon\rightharpoonup u_0$ in $\ths$, with convergence in $L^2(\Omega)$ and almost everywhere in $\Omega$. The normalization and inequality \eqref{hls_ineq} imply
\begin{equation*}
-\int_\Omega au_0^2\,dx
=\lim_{\varepsilon\to0}\left(-\int_\Omega au_\varepsilon^2\,dx\right)
\geq\SH-\SH(a)>0,
\end{equation*}
so $u_0\not\equiv0$.

Set $v_\varepsilon=u_\varepsilon-u_0$. By weak convergence and strong convergence in $L^2(\Omega)$,
\begin{equation*}
\|\ds v_\varepsilon\|_2^2+\|\ds u_0\|_2^2+\int_\Omega au_0^2\,dx=\SH(a)+o(1).
\end{equation*}
Moreover, the nonlocal Brezis--Lieb lemma of Mukherjee and Sreenadh \cite[Lemma~3.3]{MukherjeeSreenadh2016}, together with $\twostar>1$, gives
\begin{equation*}
\begin{aligned}
1\leq{}&\left(\int_\Omega\int_\Omega
\frac{|u_0(x)|^{\twostar}|u_0(y)|^{\twostar}}{|x-y|^\mu}\,dx\,dy\right)^{\frac{1}{\twostar}}+\left(\int_\Omega\int_\Omega
\frac{|v_\varepsilon(x)|^{\twostar}|v_\varepsilon(y)|^{\twostar}}{|x-y|^\mu}\,dx\,dy\right)^{\frac{1}{\twostar}}+o(1).
\end{aligned}
\end{equation*}
Since $0<\SH(a)<\SH$ (positivity follows from coercivity), inequality \eqref{hls_ineq} applied to $v_\varepsilon$ yields
\begin{equation*}
\begin{aligned}
\SH(a)\leq{}&\SH(a)\left(\int_\Omega\int_\Omega
\frac{|u_0(x)|^{\twostar}|u_0(y)|^{\twostar}}{|x-y|^\mu}\,dx\,dy\right)^{\frac{1}{\twostar}}\\
&+\|\ds v_\varepsilon\|_2^2+o(1).
\end{aligned}
\end{equation*}
Combining this with the energy splitting and letting $\varepsilon\to0$, we obtain
\begin{equation*}
\|\ds u_0\|_2^2+\int_\Omega au_0^2\,dx
\leq\SH(a)\left(\int_\Omega\int_\Omega
\frac{|u_0(x)|^{\twostar}|u_0(y)|^{\twostar}}{|x-y|^\mu}\,dx\,dy\right)^{\frac{1}{\twostar}}.
\end{equation*}
The reverse inequality follows from the definition of $\SH(a)$. Thus $u_0$ is a minimizer. Note that this argument only uses that $a\in L^\infty(\Omega)$ and that $(-\Delta)^s+a$ is coercive; the continuity of $a$ is not needed.
\end{proof}

\begin{proof}[Proof of (ii) $\Rightarrow$ (i)]
Let $a\in C(\overline\Omega)$ be such that $(-\Delta)^s+a$ is coercive and $\SH(a)<\SH$, and assume by contradiction that $\phi_a\geq0$ in $\Omega$.

\emph{Step 1: the threshold potential.}
For $\gamma\geq0$, the operator $(-\Delta)^s+a+\gamma$ is coercive with the same coercivity constant as $(-\Delta)^s+a$. Hence, for $0\leq\gamma_1\leq\gamma_2$ and $u\in\ths\setminus\{0\}$, the Poincar\'e inequality gives
\begin{equation*}
0\leq\SH(a+\gamma_2)[u]-\SH(a+\gamma_1)[u]
=(\gamma_2-\gamma_1)\frac{\|u\|_{L^2(\Omega)}^2}{\mathcal D(u)^{1/\twostar}}
\leq C(\gamma_2-\gamma_1)\SH(a+\gamma_1)[u].
\end{equation*}
Taking the infimum over $u\in\ths\setminus\{0\}$, we obtain $0\leq\SH(a+\gamma_2)-\SH(a+\gamma_1)\leq C(\gamma_2-\gamma_1)\SH(a+\gamma_1)$. Hence $\gamma\mapsto\SH(a+\gamma)$ is nondecreasing and continuous.
Moreover, $\SH(a+\gamma)=\SH$ as soon as $a+\gamma\geq0$, since then $\SH(a+\gamma)[u]\geq\SH(0)[u]\geq\SH$ by \eqref{hls_ineq}, while $\SH(a+\gamma)\leq\SH$ by the test-function argument in the proof of (iii) $\Rightarrow$ (ii) above. Hence
\begin{equation*}
\gamma^*:=\sup\{\gamma\geq0:\SH(a+\gamma)<\SH\}\in(0,\|a^-\|_\infty],
\end{equation*}
and $a^*:=a+\gamma^*$ satisfies $\SH(a^*)=\SH$, while $\SH(a+\gamma)<\SH$ for every $\gamma\in[0,\gamma^*)$.

\emph{Step 2: $\inf_K\phi_{a^*}>0$ on compact sets.}
The resolvent identity gives
\begin{equation*}
G_a(x,y)-G_{a^*}(x,y)
=\frac{\gamma^*}{\gamma_{N,s}}\int_\Omega G_a(x,t)G_{a^*}(t,y)\,dt.
\end{equation*}
Since $2(N-2s)<N$, splitting the integral near and away from $x$ and using the Green function bounds and Remark~\ref{rem:continuous-a} allow us to let $y\to x$. Thus
\begin{equation*}
\phi_{a^*}(x)-\phi_a(x)
=\frac{\gamma^*}{\gamma_{N,s}}\int_\Omega G_a(x,t)G_{a^*}(t,x)\,dt>0,
\qquad x\in\Omega.
\end{equation*}
The strict inequality follows from $\gamma^*>0$ and the positivity of the Green functions. Since $\phi_a\geq0$, we have $\phi_{a^*}>0$ in $\Omega$.
Since $\phi_{a^*}$ is continuous, $m_K:=\min_K\phi_{a^*}>0$ for every compact set $K\subset\Omega$.

\emph{Step 3: blow-up and contradiction.} Let $\varepsilon_n\downarrow0$, $V\equiv-1$, so that $a^*+\varepsilon_nV=a+\gamma^*-\varepsilon_n$ and $\SH(a^*+\varepsilon_nV)<\SH$. Choose $u_n\in\ths$ satisfying \eqref{eq:profile-almost-minimizing} and \eqref{eq:profile-normalization} (with $a^*$ in place of $a$ and $\varepsilon_n$ in place of $\varepsilon$). Since $\SH(a^*)=\SH$, Propositions~\ref{prop:profile-decomposition} and \ref{prop:remainder-estimates} apply, and the concentration points $z_n$ stay in a compact set $K\subset\Omega$. By Proposition~\ref{prop:refined-lower-bound}\,(a),
\begin{equation*}
\SH(a^*+\varepsilon_nV)[u_n]\geq\SH+c_1m_K\lambda_n^{-(N-2s)}+o\left(\lambda_n^{-(N-2s)}\right)>\SH
\end{equation*}
for $n$ large. On the other hand, \eqref{eq:profile-almost-minimizing} and $\SH(a^*+\varepsilon_nV)<\SH$ give $\SH(a^*+\varepsilon_nV)[u_n]<\SH$ for $n$ large, a contradiction. Hence $\phi_a(x)<0$ for some $x\in\Omega$.
\end{proof}

\begin{proof}[Proof of Corollary~\ref{lem:de-1.3}]
Let $a$ be a critical potential. Then $\SH(a)=\SH>0$, so $(-\Delta)^s+a$ is coercive (see the proof of Proposition~\ref{prop:profile-decomposition}), and Theorem~\ref{main} gives $\phi_a\geq0$ in $\Omega$.
For $x\in\Omega$, set $w_x:=H_a(x,\cdot)-H_0(x,\cdot)$. Then $w_x\in\ths$ satisfies $(-\Delta)^sw_x+aw_x=aG_0(x,\cdot)$. Since $0\leq G_0(x,y)\leq|x-y|^{-(N-2s)}$ and $N<4s$,
\begin{equation*}
\sup_{x\in\Omega}\|G_0(x,\cdot)\|_{L^2(\Omega)}^2
\leq\sup_{x\in\Omega}\int_\Omega|x-y|^{-2(N-2s)}\,dy\leq C.
\end{equation*}
The energy estimate gives $\|w_x\|_{L^2(\Omega)}\leq C$. As $2>N/(2s)$, \cite[Lemma~2.5(a)]{BiccariWarmaZuazua2017}, applied to $(-\Delta)^sw_x=a(G_0(x,\cdot)-w_x)$, yields
\begin{equation*}
\|w_x\|_{L^\infty(\Omega)}\leq C\|a(G_0(x,\cdot)-w_x)\|_{L^2(\Omega)}\leq C,
\end{equation*}
uniformly for $x\in\Omega$. By continuity, $|\phi_a(x)-\phi_0(x)|=|w_x(x)|\leq C$. Together with $\phi_0(x)\gtrsim d(x)^{-(N-2s)}$ from \cite[Lemma~A.1, (A.1)]{DeNittiKonig2023}, this gives $\phi_a(x)\to+\infty$ as $x\to\partial\Omega$. Since $\phi_a$ is continuous,
the infimum of $\phi_a$ is attained.

Assume by contradiction that $m:=\min_\Omega\phi_a>0$. Let $\chi\in C(\overline\Omega)$, $0\leq\chi\leq1$, $\chi\not\equiv0$, and $\widetilde a:=a-\delta\chi$ with $\delta>0$ small.
The operators $(-\Delta)^s+\widetilde a$ are uniformly coercive for $0<\delta\leq\delta_0$, with $\delta_0>0$ sufficiently small. Set $v_x:=H_a(x,\cdot)-H_{\widetilde a}(x,\cdot)$. Then $v_x\in\ths$ satisfies $(-\Delta)^sv_x+\widetilde a v_x=\delta\chi G_a(x,\cdot)$, and the maximum principle gives $v_x\geq0$. Since $G_a(x,\cdot)=G_0(x,\cdot)-w_x$, the bounds above imply $\|G_a(x,\cdot)\|_{L^2(\Omega)}\leq C$. The energy estimate now gives $\|v_x\|_{L^2(\Omega)}\leq C\delta$, and \cite[Lemma~2.5(a)]{BiccariWarmaZuazua2017} yields
\begin{equation*}
\|v_x\|_{L^\infty(\Omega)}
\leq C\|\delta\chi G_a(x,\cdot)-\widetilde a v_x\|_{L^2(\Omega)}\leq C\delta,
\end{equation*}
with $C$ independent of $x$ and $\delta$. Evaluating at $x$ gives $0\leq\phi_a(x)-\phi_{\widetilde a}(x)=v_x(x)\leq C\delta$.
Hence $\phi_{\widetilde a}\geq m/2>0$ for $\delta$ small, and Theorem~\ref{main} gives $\SH(\widetilde a)=\SH$. Since $\widetilde a\leq a$ and $\widetilde a\not\equiv a$, this contradicts the criticality of $a$.
\end{proof}

\section{Proofs of Theorems \ref{lem:de-1.4}, \ref{lem:de-1.5} and \ref{lem:de-1.6}}
\label{sec:proof-asymptotics}

Throughout this section, $\frac52s<N<4s$, $a$ is a critical potential satisfying \eqref{eq:assumption-a-on-Na}, $V\in L^\infty(\Omega)$, and $c_1,c_2,c_3$, $\theta$ are as in Section~\ref{sec:lower-bound-refined}. By Corollaries~\ref{lem:de-1.3} and \ref{lem:de-2.2}, $\phi_a\geq0$ in $\Omega$, $\mathcal N_a\neq\emptyset$, and $\mathcal N_a$ is compact (since $\phi_a(x)\to+\infty$ as $x\to\partial\Omega$). For $x\in\mathcal N_a$ and $\Lambda>0$ we set
\begin{equation*}
F_x(\Lambda):=c_2|a(x)|\Lambda^{-2s}+c_3Q_V(x)\Lambda^{-(N-2s)},
\qquad
\Phi(x):=\frac{\left(Q_V(x)^-\right)^{\frac{2s}{4s-N}}}{|a(x)|^{\frac{N-2s}{4s-N}}},
\end{equation*}
where $Q_V^-:=\max\{-Q_V,0\}$, and
\begin{equation*}
C_*:=\frac{4s-N}{2s}\left(\frac{N-2s}{2s}\right)^{\frac{N-2s}{4s-N}}c_3^{\frac{2s}{4s-N}}c_2^{-\frac{N-2s}{4s-N}}
=\frac{4s-N}{2s}\left(\frac{N-2s}{2s}\right)^{\frac{N-2s}{4s-N}}
\frac{\kappa_{N,\mu,s}^{2}\,\SH^{-\theta}}{\left(a_{N,s}d_{N,s}b_{N,s}+\alpha_{N,s}\right)^{\frac{N-2s}{4s-N}}}.
\end{equation*}
Since $a<0$ on $\mathcal N_a$ and $Q_V$ is continuous (because $x\mapsto G_a(x,\cdot)^2$ is continuous in $L^1(\Omega)$ for $N<4s$), $\Phi$ is continuous on $\mathcal N_a$, and $\sup_{\mathcal N_a(V)}\Phi=\max_{\mathcal N_a}\Phi$ (with the convention $\sup\emptyset=0$).

\begin{lem}
\label{lem:F-minimization}
Let $x\in\mathcal N_a$. If $Q_V(x)\geq0$, then $F_x>0$ on $(0,\infty)$. If $Q_V(x)<0$, then
\begin{equation*}
\min_{\Lambda>0}F_x(\Lambda)=-C_*\Phi(x),
\end{equation*}
and the minimum is attained exactly at
\begin{equation*}
\Lambda_x:=\left(\frac{2s\,c_2|a(x)|}{(N-2s)\,c_3|Q_V(x)|}\right)^{\frac{1}{4s-N}}
=\left(\frac{2s\left(a_{N,s}d_{N,s}b_{N,s}+\alpha_{N,s}\right)|a(x)|}{(N-2s)|Q_V(x)|}\right)^{\frac{1}{4s-N}}.
\end{equation*}
\end{lem}

\begin{proof}
The first claim is obvious. If $Q_V(x)<0$, write $F_x(\Lambda)=A\Lambda^{-2s}-B\Lambda^{-(N-2s)}$ with $A=c_2|a(x)|$, $B=c_3|Q_V(x)|$. Since $N-2s<2s$, $F_x(\Lambda)\to+\infty$ as $\Lambda\to0^+$ and $F_x(\Lambda)\to0^-$ as $\Lambda\to\infty$, and $F_x'(\Lambda)=0$ if and only if $\Lambda^{4s-N}=\frac{2sA}{(N-2s)B}$. At this point $A\Lambda^{-2s}=\frac{N-2s}{2s}B\Lambda^{-(N-2s)}$, so that
\begin{equation*}
F_x(\Lambda_x)=-\frac{4s-N}{2s}B\Lambda_x^{-(N-2s)}
=-\frac{4s-N}{2s}\left(\frac{N-2s}{2s}\right)^{\frac{N-2s}{4s-N}}B^{\frac{2s}{4s-N}}A^{-\frac{N-2s}{4s-N}}
=-C_*\Phi(x). \qedhere
\end{equation*}
\end{proof}

\begin{proof}[Proof of Theorem~\ref{lem:de-1.4}: upper bound for $\SH(a+\varepsilon V)$]
Let $x\in\mathcal N_a(V)$ and $\Lambda>0$. Taking $z=x$ and $\lambda=\Lambda\varepsilon^{-\frac{1}{4s-N}}$ in \eqref{eq:quotient-expansion-short}, and using $\phi_a(x)=0$ and $\mathcal T_1(0,\lambda)=0$, we get
\begin{equation*}
\SH(a+\varepsilon V)\leq\SH(a+\varepsilon V)[\psi_{x,\lambda}]=\SH+\varepsilon^{\frac{2s}{4s-N}}\left(F_x(\Lambda)+o(1)\right).
\end{equation*}
Choosing $\Lambda=\Lambda_x$ and then optimizing over $x\in\mathcal N_a(V)$, Lemma~\ref{lem:F-minimization} yields
\begin{equation}
\label{eq:liminf-energy}
\liminf_{\varepsilon\to0^+}\frac{\SH-\SH(a+\varepsilon V)}{\varepsilon^{\frac{2s}{4s-N}}}\geq C_*\sup_{\mathcal N_a(V)}\Phi>0 .
\end{equation}
In particular, $\SH(a+\varepsilon V)<\SH$ for all small $\varepsilon>0$. Since $\varepsilon\mapsto\SH(a+\varepsilon V)$ is concave (as an infimum of affine functions of $\varepsilon$) and equals $\SH$ at $\varepsilon=0$, it follows that $\SH(a+\varepsilon V)<\SH$ for every $\varepsilon>0$.
\end{proof}

The matching lower bound and Theorem~\ref{lem:de-1.6} follow from the next proposition.

\begin{Prop}
\label{prop:asymptotics-sequences}
Let $\varepsilon\to0^+$ along a sequence such that $\SH(a+\varepsilon V)<\SH$, and let $(u_\varepsilon)$ satisfy \eqref{eq:profile-almost-minimizing} and \eqref{eq:profile-normalization}. Let $z_\varepsilon,\lambda_\varepsilon,\alpha_\varepsilon$, $q_\varepsilon$ be as in Propositions~\ref{prop:profile-decomposition} and \ref{prop:refined-lower-bound}. Then, up to a subsequence, $z_\varepsilon\to z_0\in\mathcal N_a$, $a(z_0)<0$, $\phi_a(z_\varepsilon)\to0$, and
\begin{equation}
\label{eq:basic-inequality}
\begin{split}
&c_2|a(z_\varepsilon)|\lambda_\varepsilon^{-2s}(1+o(1))+c_3\varepsilon Q_V(z_\varepsilon)\lambda_\varepsilon^{-(N-2s)}+c_1\phi_a(z_\varepsilon)\lambda_\varepsilon^{-(N-2s)}(1+o(1))+\frac{\rho}{4}\SH^{-\theta}\|\ds q_\varepsilon\|_2^2\\
&\qquad\leq\left(\SH(a+\varepsilon V)-\SH\right)(1+o(1))+o\left(\varepsilon\lambda_\varepsilon^{-(N-2s)}\right).
\end{split}
\end{equation}
\end{Prop}

\begin{proof}
By \eqref{eq:profile-almost-minimizing}, $\mathcal Q_\varepsilon[u_\varepsilon]-\SH=(\SH(a+\varepsilon V)-\SH)(1+o(1))<0$. By Proposition~\ref{prop:refined-lower-bound}\,(b),
\begin{equation*}
c_1\phi_a(z_\varepsilon)\lambda_\varepsilon^{-(N-2s)}(1+o(1))-c_2a(z_\varepsilon)\lambda_\varepsilon^{-2s}+c_3\varepsilon Q_V(z_\varepsilon)\lambda_\varepsilon^{-(N-2s)}+\frac{\rho}{4}\SH^{-\theta}\|\ds q_\varepsilon\|_2^2+o\left(\lambda_\varepsilon^{-2s}\right)+o\left(\varepsilon\lambda_\varepsilon^{-(N-2s)}\right)<0 .
\end{equation*}
Since $\phi_a\geq0$ and $a$, $Q_V$ are bounded on compact sets, this gives $\phi_a(z_\varepsilon)\lesssim\lambda_\varepsilon^{-(4s-N)}+\varepsilon+o(1)\to0$. By Proposition~\ref{prop:remainder-estimates}, $z_\varepsilon\to z_0\in\Omega$ up to a subsequence, and by continuity $\phi_a(z_0)=0$, that is, $z_0\in\mathcal N_a$. By \eqref{eq:assumption-a-on-Na}, $a(z_0)<0$, so that $-a(z_\varepsilon)=|a(z_\varepsilon)|$ for $\varepsilon$ small and the term $o(\lambda_\varepsilon^{-2s})$ can be absorbed into $c_2|a(z_\varepsilon)|\lambda_\varepsilon^{-2s}$. This gives \eqref{eq:basic-inequality}.
\end{proof}

\begin{proof}[Proof of Theorem~\ref{lem:de-1.4}: lower bound for $\SH(a+\varepsilon V)$]
Choose a sequence $\varepsilon\to0^+$ along which
\begin{equation*}
\varepsilon^{-\frac{2s}{4s-N}}\bigl(\SH-\SH(a+\varepsilon V)\bigr)\to\limsup_{\varepsilon'\to0^+}\varepsilon'^{-\frac{2s}{4s-N}}\bigl(\SH-\SH(a+\varepsilon' V)\bigr),
\end{equation*}
and let $(u_\varepsilon)$ be as in Proposition~\ref{prop:asymptotics-sequences}. By \eqref{eq:liminf-energy}, $\SH(a+\varepsilon V)-\SH\leq-\frac12C_*\sup_{\mathcal N_a(V)}\Phi\,\varepsilon^{\frac{2s}{4s-N}}<0$ for $\varepsilon$ small.

\emph{Step 1: $\lambda_\varepsilon\sim\varepsilon^{-\frac{1}{4s-N}}$.} Dropping the nonnegative terms involving $\phi_a$ and $q_\varepsilon$ in \eqref{eq:basic-inequality}, we obtain
\begin{equation*}
\frac{c_2}{2}|a(z_0)|\lambda_\varepsilon^{-2s}+\frac{1}{4}C_*\sup_{\mathcal N_a(V)}\Phi\,\varepsilon^{\frac{2s}{4s-N}}\leq C\varepsilon\lambda_\varepsilon^{-(N-2s)}
\end{equation*}
for $\varepsilon$ small. Comparing the right-hand side with the first term on the left gives $\lambda_\varepsilon^{-(4s-N)}\lesssim\varepsilon$, and comparing it with the second term gives $\lambda_\varepsilon^{-(N-2s)}\gtrsim\varepsilon^{\frac{N-2s}{4s-N}}$. Hence $\Lambda_\varepsilon:=\varepsilon^{\frac{1}{4s-N}}\lambda_\varepsilon$ satisfies $0<c\leq\Lambda_\varepsilon\leq C$, and in particular $o(\lambda_\varepsilon^{-2s})+o(\varepsilon\lambda_\varepsilon^{-(N-2s)})=o(\varepsilon^{\frac{2s}{4s-N}})$.

\emph{Step 2: conclusion.} Dividing \eqref{eq:basic-inequality} by $\varepsilon^{\frac{2s}{4s-N}}$ and using $a(z_\varepsilon)\to a(z_0)$, $Q_V(z_\varepsilon)\to Q_V(z_0)$, we get
\begin{equation}
\label{eq:rescaled-inequality}
F_{z_0}(\Lambda_\varepsilon)+\frac{c_1\phi_a(z_\varepsilon)\lambda_\varepsilon^{-(N-2s)}+\frac{\rho}{4}\SH^{-\theta}\|\ds q_\varepsilon\|_2^2}{\varepsilon^{\frac{2s}{4s-N}}}(1+o(1))+o(1)
\leq\frac{\SH(a+\varepsilon V)-\SH}{\varepsilon^{\frac{2s}{4s-N}}}(1+o(1)).
\end{equation}
Since the right-hand side is negative and bounded away from zero, Lemma~\ref{lem:F-minimization} forces $Q_V(z_0)<0$, i.e.\ $z_0\in\mathcal N_a(V)$, and then $F_{z_0}(\Lambda_\varepsilon)\geq-C_*\Phi(z_0)\geq-C_*\sup_{\mathcal N_a(V)}\Phi$. Therefore
\begin{equation*}
\limsup_{\varepsilon\to0^+}\frac{\SH-\SH(a+\varepsilon V)}{\varepsilon^{\frac{2s}{4s-N}}}\leq C_*\sup_{\mathcal N_a(V)}\Phi .
\end{equation*}
Together with \eqref{eq:liminf-energy}, this proves Theorem~\ref{lem:de-1.4}.
\end{proof}

\begin{proof}[Proof of Theorem~\ref{lem:de-1.6}]
By Theorem~\ref{lem:de-1.4}, $\SH(a+\varepsilon V)<\SH$ for every $\varepsilon>0$, so Propositions~\ref{prop:profile-decomposition}, \ref{prop:refined-lower-bound} and \ref{prop:asymptotics-sequences} apply to $(u_\varepsilon)$. By definition of $q_\varepsilon$,
\begin{equation*}
u_\varepsilon=\alpha_\varepsilon\Big(P\bar U_{z_\varepsilon,\lambda_\varepsilon}+\kappa_{N,\mu,s}\beta_\varepsilon\lambda_\varepsilon^{-\frac{N-2s}{2}}\Pi^\perp_{z_\varepsilon,\lambda_\varepsilon}\left(H_0(z_\varepsilon,\cdot)-H_a(z_\varepsilon,\cdot)\right)+r_\varepsilon\Big),
\qquad r_\varepsilon:=q_\varepsilon\in T_{z_\varepsilon,\lambda_\varepsilon}^\perp .
\end{equation*}
Consider any subsequence. By Theorem~\ref{lem:de-1.4}, the right-hand side of \eqref{eq:rescaled-inequality} converges to $-C_*\sup_{\mathcal N_a(V)}\Phi$. Passing to a further subsequence such that $z_\varepsilon\to z_0$ and $\Lambda_\varepsilon\to\Lambda_0\in(0,\infty)$, \eqref{eq:rescaled-inequality} and Lemma~\ref{lem:F-minimization} give
\begin{equation*}
-C_*\Phi(z_0)\leq F_{z_0}(\Lambda_0)\leq F_{z_0}(\Lambda_0)+\limsup_{\varepsilon\to0^+}\frac{c_1\phi_a(z_\varepsilon)\lambda_\varepsilon^{-(N-2s)}+\frac{\rho}{4}\SH^{-\theta}\|\ds q_\varepsilon\|_2^2}{\varepsilon^{\frac{2s}{4s-N}}}\leq-C_*\sup_{\mathcal N_a(V)}\Phi .
\end{equation*}
Hence all inequalities are equalities. Consequently: $z_0\in\mathcal N_a(V)$ and $\Phi(z_0)=\sup_{\mathcal N_a(V)}\Phi$, which is the asserted maximality property of $z_0$; $F_{z_0}(\Lambda_0)=\min F_{z_0}$, so that $\Lambda_0=\Lambda_{z_0}$ by Lemma~\ref{lem:F-minimization}; and
\begin{equation*}
\phi_a(z_\varepsilon)\lambda_\varepsilon^{-(N-2s)}=o\left(\varepsilon^{\frac{2s}{4s-N}}\right),
\qquad
\|\ds r_\varepsilon\|_2^2=o\left(\varepsilon^{\frac{2s}{4s-N}}\right).
\end{equation*}
Since $\lambda_\varepsilon^{-(N-2s)}\sim\varepsilon^{\frac{N-2s}{4s-N}}$, the first relation gives $\phi_a(z_\varepsilon)=o(\varepsilon)$. Thus, along every subsequence such that $z_\varepsilon\to z_0$, we have $\varepsilon^{\frac{1}{4s-N}}\lambda_\varepsilon\to\Lambda_{z_0}$, which is the asserted limit (if $\Phi$ has a unique maximizer on $\mathcal N_a(V)$, these convergences hold along the whole sequence).

It remains to estimate $\alpha_\varepsilon$. By \eqref{eq:profile-normalization}, $|\alpha_\varepsilon|^{2\twostar}\mathcal D(P\bar U_{z_\varepsilon,\lambda_\varepsilon}+w_\varepsilon)=D_0$. Since $\|\ds w_\varepsilon\|_2^2=O(\lambda_\varepsilon^{-(N-2s)})$ (Proposition~\ref{prop:remainder-estimates}), $w_\varepsilon\perp P\bar U_{z_\varepsilon,\lambda_\varepsilon}$ and $z_\varepsilon$ stays in a compact set, the expansion \eqref{eq:first-expansion-denominator} gives $\mathcal D(P\bar U_{z_\varepsilon,\lambda_\varepsilon}+w_\varepsilon)=D_0+O(\lambda_\varepsilon^{-(N-2s)})$. Hence $|\alpha_\varepsilon|=1+O(\lambda_\varepsilon^{-(N-2s)})=1+O(\varepsilon^{\frac{N-2s}{4s-N}})$, and $\alpha_\varepsilon=\varsigma+O(\varepsilon^{\frac{N-2s}{4s-N}})$ with $\varsigma\in\{-1,1\}$ given by Proposition~\ref{prop:profile-decomposition}.
Finally, Proposition~\ref{prop:refined-lower-bound} and the scale estimate give $\beta_\varepsilon=1+O(\lambda_\varepsilon^{-(N-2s)})=1+O(\varepsilon^{\frac{N-2s}{4s-N}})$.
\end{proof}

\begin{proof}[Proof of Theorem~\ref{lem:de-1.5}]
Assume $\mathcal N_a(V)=\emptyset$, that is, $Q_V\geq0$ on $\mathcal N_a$. For fixed $z\in\Omega$ and $\varepsilon\geq0$, letting $\lambda\to\infty$ in \eqref{eq:quotient-expansion} gives $\SH(a+\varepsilon V)\leq\SH$.
If $\SH(a+\varepsilon V)=\SH$ for all small $\varepsilon>0$ there is nothing to prove; otherwise, let $\varepsilon\to0^+$ along a sequence with $\SH(a+\varepsilon V)<\SH$ and let $(u_\varepsilon)$ be as in Proposition~\ref{prop:asymptotics-sequences}. After passing to a subsequence, we have $z_\varepsilon\to z_0\in\mathcal N_a$ and $Q_V(z_0)\geq0$.

Fix $\delta>0$. For $\varepsilon$ small, $c_3Q_V(z_\varepsilon)\geq-\delta$ and $|a(z_\varepsilon)|(1+o(1))\geq\frac12|a(z_0)|$, so that \eqref{eq:basic-inequality} and $o(\varepsilon\lambda_\varepsilon^{-(N-2s)})\geq-\delta\varepsilon\lambda_\varepsilon^{-(N-2s)}$ give
\begin{equation*}
\left(\SH(a+\varepsilon V)-\SH\right)(1+o(1))
\geq\frac{c_2}{2}|a(z_0)|\lambda_\varepsilon^{-2s}-2\delta\varepsilon\lambda_\varepsilon^{-(N-2s)}
\geq-C(\delta)\,\varepsilon^{\frac{2s}{4s-N}},
\end{equation*}
where, by the computation in the proof of Lemma~\ref{lem:F-minimization} (with $A=\frac{c_2}{2}|a(z_0)|$, $B=2\delta$ and $\lambda_\varepsilon=\Lambda\varepsilon^{-\frac{1}{4s-N}}$), $C(\delta)=C\,\delta^{\frac{2s}{4s-N}}$ with $C$ independent of $\delta$. Since $\delta>0$ is arbitrary, $\SH(a+\varepsilon V)-\SH\geq-o(\varepsilon^{\frac{2s}{4s-N}})$ along every such sequence, which proves the first assertion.

Assume now that $Q_V>0$ on $\mathcal N_a$ and, by contradiction, that $\SH(a+\varepsilon V)<\SH$ along a sequence $\varepsilon\to0^+$. With $(u_\varepsilon)$ as above, $Q_V(z_\varepsilon)\to Q_V(z_0)>0$, and \eqref{eq:basic-inequality} gives
\begin{equation*}
\frac{c_2}{2}|a(z_0)|\lambda_\varepsilon^{-2s}+\frac{c_3}{2}Q_V(z_0)\varepsilon\lambda_\varepsilon^{-(N-2s)}\leq\left(\SH(a+\varepsilon V)-\SH\right)(1+o(1))<0
\end{equation*}
for $\varepsilon$ small, which is impossible. Hence $\SH(a+\varepsilon V)=\SH$ for all sufficiently small $\varepsilon>0$.
\end{proof}

\appendix
\section{Basic computations}

The constants used above are given by the following identities; see \cite[Lemma~B.5]{DeNittiKonig2023}.

\begin{lem}[Explicit constants]
\label{lem:explicit-constants}
Let $2s<N<4s$ and set
\begin{equation*}
U_{0,1}(x):=\left(\frac{1}{1+|x|^2}\right)^{\frac{N-2s}{2}},\qquad
h(x):=\frac{1}{|x|^{N-2s}}
-\frac{1}{(1+|x|^2)^{\frac{N-2s}{2}}}.
\end{equation*}
Then, for every $0\leq k<\frac{N}{N-2s}$, we have
\begin{equation*}
\mathfrak a_{N,s}^{(k)}
:=\int_{\R^N}U_{0,1}(x)^{2_s^*-k}dx
=\pi^{N/2}
\frac{\Gamma\left(\frac{N}{2}(1-k)+ks\right)}
{\Gamma\left(\frac{N}{2}(2-k)+ks\right)}.
\end{equation*}
Furthermore,
\begin{equation*}
b_{N,s}
:=\int_{\R^N}
U_{0,1}(x)^{\frac{N+2s}{N-2s}}|x|^{4s-N}dx
=\pi^{N/2}
\frac{\Gamma(2s)\Gamma\left(\frac{N}{2}-s\right)}
{\Gamma\left(\frac{N}{2}\right)
 \Gamma\left(\frac{N}{2}+s\right)},
\end{equation*}
and
\begin{equation*}
\alpha_{N,s}
:=\int_{\R^N}U_{0,1}(x)h(x)dx
=\frac{\pi^{N/2}}{\Gamma\left(\frac{N}{2}\right)}
\Gamma\left(\frac{N}{2}-2s\right)
\left(
\frac{\Gamma(s)}{\Gamma\left(\frac{N}{2}-s\right)}
-\frac{\Gamma\left(\frac{N}{2}\right)}{\Gamma(N-2s)}
\right).
\end{equation*}
Moreover, the normalization constant in
\begin{equation*}
(-\Delta)^s u(z)
=C_{N,s}\operatorname{P.V.}\int_{\R^N}
\frac{u(z)-u(x)}{|z-x|^{N+2s}}dx
\end{equation*}
is given by
\begin{equation*}
C_{N,s}
:=\frac{2^{2s}s\Gamma\left(\frac{N+2s}{2}\right)}
{\pi^{N/2}\Gamma(1-s)}.
\end{equation*}
Finally,
\begin{equation*}
(-\Delta)^sU_{0,1}
=a_{N,s}U_{0,1}^{\frac{N+2s}{N-2s}},
\end{equation*}
where
\begin{equation*}
a_{N,s}
:=2^{2s}
\frac{\Gamma\left(\frac{N+2s}{2}\right)}
{\Gamma\left(\frac{N-2s}{2}\right)}.
\end{equation*}
\end{lem}

\begin{lem}
\label{lem:projection-barU}
Let $z\in\Omega$ and $N>2s$. Then the following assertions hold.

\textnormal{(i)} We have
\begin{equation*}
0\leq P\bar U_{z,\lambda}\leq\bar U_{z,\lambda}.
\end{equation*}
Moreover, the function
\begin{equation*}
\varphi_{z,\lambda}
:=\bar U_{z,\lambda}-P\bar U_{z,\lambda}
\end{equation*}
satisfies
\begin{equation*}
\left\|\varphi_{z,\lambda}\right\|_{L^\infty(\R^N)}
\lesssim
d(z)^{-N+2s}\lambda^{-\frac{N-2s}{2}},
\end{equation*}
and
\begin{equation*}
\left\|\varphi_{z,\lambda}\right\|_{L^{\frac{2N}{N-2s}}(\R^N)}
\lesssim
\bigl(d(z)\lambda\bigr)^{-\frac{N-2s}{2}}.
\end{equation*}

\textnormal{(ii)} We have the expansion
\begin{equation*}
\begin{split}
P\bar U_{z,\lambda}(x)
&=\bar U_{z,\lambda}(x)
-\kappa_{N,\mu,s}
\lambda^{-\frac{N-2s}{2}}H_0(z,x)
-f_{z,\lambda}(x),
\end{split}
\end{equation*}
where
\begin{equation*}
\begin{split}
f_{z,\lambda}
&:=\varphi_{z,\lambda}
-\kappa_{N,\mu,s}
\lambda^{-\frac{N-2s}{2}}H_0(z,\cdot)
\end{split}
\end{equation*}
satisfies
\begin{equation*}
\left\|f_{z,\lambda}\right\|_{L^\infty(\Omega)}
\lesssim
d(z)^{-N-2+2s}\lambda^{-\frac{N+4-2s}{2}}.
\end{equation*}
\end{lem}

\begin{lem}
\label{lem:barU-Lq-estimates}
Let $z\in\Omega$ and $q\in[1,\infty]$. As $\lambda\to\infty$, uniformly for $z$ in compact subsets of $\Omega$, we have
\begin{equation*}
\left\|\bar U_{z,\lambda}\right\|_{L^q(\Omega)}
\lesssim
\begin{cases}
\lambda^{\frac{N-2s}{2}-\frac{N}{q}},
& q>\dfrac{N}{N-2s},
\\[2mm]
\lambda^{-\frac{N-2s}{2}}
(\ln\lambda)^{\frac{N-2s}{N}},
& q=\dfrac{N}{N-2s},
\\[2mm]
\lambda^{-\frac{N-2s}{2}},
& 1\leq q<\dfrac{N}{N-2s}.
\end{cases}
\end{equation*}
Moreover,
\begin{equation*}
\left|\partial_\lambda\bar U_{z,\lambda}\right|
\lesssim
\lambda^{-1}\bar U_{z,\lambda},
\end{equation*}
and hence
\begin{equation*}
\left\|\partial_\lambda\bar U_{z,\lambda}\right\|_{L^q(\Omega)}
\lesssim
\lambda^{-1}\left\|\bar U_{z,\lambda}\right\|_{L^q(\Omega)}.
\end{equation*}
Finally, for every $i=1,\ldots,N$,
\begin{equation*}
\left\|\partial_{z_i}\bar U_{z,\lambda}\right\|_{L^q(\Omega)}
\lesssim
\begin{cases}
\lambda^{\frac{N-2s+2}{2}-\frac{N}{q}},
& q>\dfrac{N}{N-2s+1},
\\[2mm]
\lambda^{-\frac{N-2s}{2}}
(\ln\lambda)^{\frac{N-2s+1}{N}},
& q=\dfrac{N}{N-2s+1},
\\[2mm]
\lambda^{-\frac{N-2s}{2}},
& 1\leq q<\dfrac{N}{N-2s+1}.
\end{cases}
\end{equation*}
\end{lem}

\begin{lem}
\label{lem:barU-global-exterior-Lq}
Let $z\in\Omega$. As $\lambda\to\infty$, we have
\begin{equation*}
\left\|\bar U_{z,\lambda}\right\|_{L^q(\R^N)}
\lesssim
\lambda^{\frac{N-2s}{2}-\frac{N}{q}},
\qquad
\frac{N}{N-2s}<q<\infty,
\end{equation*}
\begin{equation*}
\left\|\partial_\lambda\bar U_{z,\lambda}\right\|_{L^q(\R^N)}
\lesssim
\lambda^{\frac{N-2s-2}{2}-\frac{N}{q}},
\qquad
\frac{N}{N-2s}<q<\infty,
\end{equation*}
and
\begin{equation*}
\left\|\partial_{z_i}\bar U_{z,\lambda}\right\|_{L^q(\R^N)}
\lesssim
\lambda^{\frac{N-2s+2}{2}-\frac{N}{q}},
\qquad
\frac{N}{N-2s+1}<q<\infty.
\end{equation*}
Moreover, for any $\tau=\tau_\lambda>0$ satisfying $\tau\lambda\to\infty$, we have
\begin{equation*}
\left\|\bar U_{z,\lambda}\right\|_{L^q(\R^N\setminus B_\tau(z))}
\lesssim
\lambda^{-\frac{N-2s}{2}}
\tau^{\frac{N-q(N-2s)}{q}},
\qquad
\frac{N}{N-2s}<q<\infty,
\end{equation*}
\begin{equation*}
\left\|\partial_\lambda\bar U_{z,\lambda}\right\|_{L^q(\R^N\setminus B_\tau(z))}
\lesssim
\lambda^{-\frac{N-2s+2}{2}}
\tau^{\frac{N-q(N-2s)}{q}},
\qquad
\frac{N}{N-2s}<q<\infty,
\end{equation*}
and
\begin{equation*}
\left\|\partial_{z_i}\bar U_{z,\lambda}\right\|_{L^q(\R^N\setminus B_\tau(z))}
\lesssim
\lambda^{-\frac{N-2s}{2}}
\tau^{\frac{N-q(N-2s+1)}{q}},
\qquad
\frac{N}{N-2s+1}<q<\infty.
\end{equation*}
\end{lem}

For the regular part of the Green function, we use \cite[Lemma~A.3]{DeNittiKonig2023}.

\begin{lem}[Diagonal expansion of $H_a$]
\label{lemma Ha expansion}
Let $2s<N<4s$, and let $a\in C^1(\Omega)\cap C(\overline{\Omega})$ be such that $(-\Delta)^s+a$ is coercive.

\textnormal{(i)} If $4s-N<1$, then, as $x\to z$,
\begin{equation*}
H_a(z,x)
=\phi_a(z)
-d_{N,s}a(z)|x-z|^{4s-N}
+o\left(|x-z|^{4s-N}\right).
\end{equation*}

\textnormal{(ii)} If $4s-N\geq1$, then there exists $\nu_z\in\R^N$ such that, as $x\to z$,
\begin{equation*}
H_a(z,x)
=\phi_a(z)
+\nu_z\cdot(x-z)
-d_{N,s}a(z)|x-z|^{4s-N}
+o\left(|x-z|^{4s-N}\right).
\end{equation*}
Here,
\begin{equation*}
d_{N,s}
:=-2^{-2s}
\frac{\Gamma\left(\frac{N-4s}{2}\right)\Gamma(s)}
{\Gamma\left(\frac{N-2s}{2}\right)\Gamma(2s)}
>0.
\end{equation*}
The expansions are uniform for $z$ in compact subsets of $\Omega$.
\end{lem}

Combining \cite[Lemma~A.4]{DeNittiKonig2023} with the identity $\bar U_{z,\lambda}=\kappa_{N,\mu,s}U_{z,\lambda}$ gives the following estimates.

\begin{lem}
\label{lem:barU-Ha-power-asymptotics}
Under the assumptions of Lemma \ref{lemma Ha expansion}, let $k \in \N$ with $k\leq2_s^*$. If $k=1$, then
\begin{small}
\begin{equation*}
\begin{split}
\lambda^{-\frac{N-2s}{2}}
\int_\Omega
\bar U_{z,\lambda}^{2_s^*-1}(x)H_a(z,x)dx
&=\kappa_{N,\mu,s}^{2_s^*-1}
\mathfrak a_{N,s}^{(1)}\phi_a(z)\lambda^{-(N-2s)}
-\kappa_{N,\mu,s}^{2_s^*-1}
d_{N,s}b_{N,s}a(z)\lambda^{-2s}
+o\left(\lambda^{-2s}\right).
\end{split}
\end{equation*}
\end{small}
If $2\leq k\leq K$, then
\begin{equation*}
\begin{split}
\lambda^{-\frac{k(N-2s)}{2}}
\int_\Omega
\bar U_{z,\lambda}^{2_s^*-k}(x)H_a^k(z,x)dx
&=\kappa_{N,\mu,s}^{2_s^*-k}
\mathfrak a_{N,s}^{(k)}\phi_a^k(z)
\lambda^{-k(N-2s)}+o\left(\lambda^{-2s}\right).
\end{split}
\end{equation*}
Finally, if $k\geq K+1$, then
\begin{equation*}
\begin{split}
\lambda^{-\frac{k(N-2s)}{2}}
\int_\Omega
\bar U_{z,\lambda}^{2_s^*-k}(x)
\left|H_a(z,x)\right|^kdx
&=o\left(\lambda^{-2s}\right).
\end{split}
\end{equation*}
The asymptotics are uniform for $z$ in compact subsets of $\Omega$.
\end{lem}

\begin{lem}
\label{lem:barU-Ha-double-endpoint-asymptotics}
Under the assumptions of Lemma \ref{lemma Ha expansion}, let $i,j$ be nonnegative integers. If $2\leq i+j\leq K$, and  $i=0$ or $j=0$, then
\begin{equation*}
\begin{split}
&\lambda^{-\frac{(i+j)(N-2s)}{2}}
\int_\Omega\int_\Omega
\frac{
\bar U_{z,\lambda}^{\twostar-i}(x)H_a^i(z,x)
\bar U_{z,\lambda}^{\twostar-j}(y)H_a^j(z,y)}
{|x-y|^\mu}dxdy\\=&a_{N,s}A_{N,\mu,s}^{-\frac{2s}{N+2s-\mu}}
\kappa_{N,\mu,s}^{2_s^*-i-j}
\mathfrak a_{N,s}^{(i+j)}\phi_a^{i+j}(z)
\lambda^{-(i+j)(N-2s)}
+o\left(\lambda^{-2s}\right).
\end{split}
\end{equation*}
If $2\leq i+j\leq K$, and  $i\geq1$, $j\geq1$, then
\begin{equation*}
\begin{split}
&\lambda^{-\frac{(i+j)(N-2s)}{2}}
\int_\Omega\int_\Omega
\frac{
\bar U_{z,\lambda}^{\twostar-i}(x)H_a^i(z,x)
\bar U_{z,\lambda}^{\twostar-j}(y)H_a^j(z,y)}
{|x-y|^\mu}dxdy
\\
=&C_{i,j}\phi_a^{i+j}(z)
\lambda^{-(i+j)(N-2s)}
+o\left(\lambda^{-2s}\right),
\end{split}
\end{equation*}
where
\begin{equation*}
C_{i,j}
:=
\int_{\R^N}\int_{\R^N}
\frac{
\bar U_{0,1}^{\twostar-i}(x)
\bar U_{0,1}^{\twostar-j}(y)}
{|x-y|^\mu}dxdy.
\end{equation*}
The constants $C_{i,j}$ are finite. Indeed, $\bar U_{0,1}^{\twostar-i}(x)\lesssim(1+|x|)^{-\alpha_i}$ with $\alpha_i:=2N-\mu-i(N-2s)$, and $\alpha_i+\mu=2N-i(N-2s)>N$, since $i(N-2s)<K(N-2s)\leq2s<N$. Hence
\begin{equation*}
\int_{\R^N}\frac{(1+|y|)^{-\alpha_j}}{|x-y|^\mu}dy\lesssim
\begin{cases}
(1+|x|)^{-\mu}, & \alpha_j>N,\\
(1+|x|)^{-\mu}\log(2+|x|), & \alpha_j=N,\\
(1+|x|)^{N-\mu-\alpha_j}, & \alpha_j<N,
\end{cases}
\end{equation*}
and the remaining integral in $x$ converges because $\alpha_i+\mu>N$ and $\alpha_i+\alpha_j>2N-\mu$; the latter is equivalent to $(i+j)(N-2s)<2N-\mu$, which holds since $(i+j)(N-2s)\leq K(N-2s)\leq2s<N<2N-\mu$.
Finally, if $0\leq i,j\leq K$ and $i+j\geq K+1$, then
\begin{equation*}
\begin{split}
&\lambda^{-\frac{(i+j)(N-2s)}{2}}
\int_\Omega\int_\Omega
\frac{
\bar U_{z,\lambda}^{\twostar-i}(x)|H_a(z,x)|^i
\bar U_{z,\lambda}^{\twostar-j}(y)|H_a(z,y)|^j}
{|x-y|^\mu}dxdy
=o\left(\lambda^{-2s}\right).
\end{split}
\end{equation*}

The asymptotics are uniform for $z$ in compact subsets of $\Omega$.
\end{lem}

\begin{proof}
First suppose that $2\leq i+j\leq K$, and  $i=0$ or $j=0$. By symmetry, it suffices to assume that $i=0$ and $2\leq j\leq K$. The argument for \eqref{eq:linear-Ha-exterior-error} gives
\begin{equation*}
\lambda^{-\frac{j(N-2s)}{2}}
\left|
\int_\Omega\int_{\R^N\setminus\Omega}
\frac{
\bar U_{z,\lambda}^{\twostar}(x)
\bar U_{z,\lambda}^{\twostar-j}(y)H_a^j(z,y)}
{|x-y|^\mu}dxdy
\right|
=o\left(\lambda^{-2s}\right).
\end{equation*}
Combining this estimate with \eqref{easybarU} and Lemma \ref{lem:barU-Ha-power-asymptotics}, we obtain
\begin{equation*}
\begin{split}
&\lambda^{-\frac{j(N-2s)}{2}}
\int_\Omega\int_\Omega
\frac{
\bar U_{z,\lambda}^{\twostar}(x)
\bar U_{z,\lambda}^{\twostar-j}(y)H_a^j(z,y)}
{|x-y|^\mu}dxdy
\\
=&\lambda^{-\frac{j(N-2s)}{2}}
\int_\Omega\int_{\R^N}
\frac{
\bar U_{z,\lambda}^{\twostar}(x)
\bar U_{z,\lambda}^{\twostar-j}(y)H_a^j(z,y)}
{|x-y|^\mu}dxdy
+o\left(\lambda^{-2s}\right)
\\
=&a_{N,s}A_{N,\mu,s}^{-\frac{2s}{N+2s-\mu}}
\lambda^{-\frac{j(N-2s)}{2}}
\int_\Omega
\bar U_{z,\lambda}^{2_s^*-j}(y)H_a^j(z,y)dy
+o\left(\lambda^{-2s}\right)
\\
=&a_{N,s}A_{N,\mu,s}^{-\frac{2s}{N+2s-\mu}}
\kappa_{N,\mu,s}^{2_s^*-j}
\mathfrak a_{N,s}^{(j)}\phi_a^j(z)
\lambda^{-j(N-2s)}
+o\left(\lambda^{-2s}\right).
\end{split}
\end{equation*}
The case $j=0$ follows by interchanging $x$ and $y$.

Next suppose that $i\geq1$, $j\geq1$, and $2\leq i+j\leq K$. Fix $0<r_0<\operatorname{dist}(z,\partial\Omega)$. Arguing as in \eqref{eq:linear-Ha-exterior-error} and using Lemma \ref{lem:barU-global-exterior-Lq}, we obtain
\begin{equation*}
\begin{split}
&\lambda^{-\frac{(i+j)(N-2s)}{2}}
\left(
\int_{\Omega\setminus B_{r_0}(z)}\int_{B_{r_0}(z)}
\frac{
\bar U_{z,\lambda}^{\twostar-i}(x)|H_a(z,x)|^i
\bar U_{z,\lambda}^{\twostar-j}(y)|H_a(z,y)|^j}
{|x-y|^\mu}dxdy
\right.
\\
&\left.\qquad
+\int_\Omega\int_{\Omega\setminus B_{r_0}(z)}
\frac{
\bar U_{z,\lambda}^{\twostar-i}(x)|H_a(z,x)|^i
\bar U_{z,\lambda}^{\twostar-j}(y)|H_a(z,y)|^j}
{|x-y|^\mu}dxdy
\right)
=o\left(\lambda^{-2s}\right).
\end{split}
\end{equation*}
It remains to bound the integral over $B_{r_0}(z)\times B_{r_0}(z)$. If $0<4s-N<1$, Lemma \ref{lemma Ha expansion}\,(i) gives
\begin{equation*}
H_a^i(z,x)H_a^j(z,y)
=\phi_a^{i+j}(z)
+O\left(|x-z|^{4s-N}+|y-z|^{4s-N}\right).
\end{equation*}
A direct calculation using the HLS inequality gives
\begin{equation*}
\lambda^{-\frac{(i+j)(N-2s)}{2}}
\int_{B_{r_0}(z)}\int_{B_{r_0}(z)}
\frac{
\bar U_{z,\lambda}^{\twostar-i}(x)
\bar U_{z,\lambda}^{\twostar-j}(y)}
{|x-y|^\mu}
\left(
|x-z|^{4s-N}+|y-z|^{4s-N}
\right)dxdy
=o\left(\lambda^{-2s}\right).
\end{equation*}
On the other hand, if $4s-N\geq1$, then Lemma \ref{lemma Ha expansion}\,(ii) gives
\begin{small}
  \begin{equation*}
\begin{split}
H_a^i(z,x)H_a^j(z,y)
=&\phi_a^{i+j}(z)
+\phi_a^{i+j-1}(z)
\left(i\nu_z\cdot(x-z)+j\nu_z\cdot(y-z)\right)
+O\left(
|x-z|^{4s-N}+|y-z|^{4s-N}
+|x-z|^2+|y-z|^2
\right).
\end{split}
\end{equation*}
\end{small}
By symmetry,
\begin{equation*}
\begin{split}
&\int_{B_{r_0}(z)}\int_{B_{r_0}(z)}
\frac{
\bar U_{z,\lambda}^{\twostar-i}(x)
\bar U_{z,\lambda}^{\twostar-j}(y)}
{|x-y|^\mu}
\left(i\nu_z\cdot(x-z)+j\nu_z\cdot(y-z)\right)dxdy
=0.
\end{split}
\end{equation*}
The remaining terms satisfy, by the HLS inequality,
\begin{equation*}
\lambda^{-\frac{(i+j)(N-2s)}{2}}
\int_{B_{r_0}(z)}\int_{B_{r_0}(z)}
\frac{
\bar U_{z,\lambda}^{\twostar-i}(x)
\bar U_{z,\lambda}^{\twostar-j}(y)}
{|x-y|^\mu}
\left(
|x-z|^{4s-N}+|y-z|^{4s-N}
+|x-z|^2+|y-z|^2
\right)dxdy
=o\left(\lambda^{-2s}\right).
\end{equation*}
In both cases,
\begin{equation*}
\begin{split}
&\lambda^{-\frac{(i+j)(N-2s)}{2}}
\int_\Omega\int_\Omega
\frac{
\bar U_{z,\lambda}^{\twostar-i}(x)H_a^i(z,x)
\bar U_{z,\lambda}^{\twostar-j}(y)H_a^j(z,y)}
{|x-y|^\mu}dxdy
\\
=&\phi_a^{i+j}(z)\lambda^{-\frac{(i+j)(N-2s)}{2}}
\int_{B_{r_0}(z)}\int_{B_{r_0}(z)}
\frac{
\bar U_{z,\lambda}^{\twostar-i}(x)
\bar U_{z,\lambda}^{\twostar-j}(y)}
{|x-y|^\mu}dxdy
+o\left(\lambda^{-2s}\right)
\\
=&C_{i,j}\phi_a^{i+j}(z)\lambda^{-(i+j)(N-2s)}
+o\left(\lambda^{-2s}\right).
\end{split}
\end{equation*}

Finally, suppose that $0\leq i,j\leq K$ and $i+j\geq K+1$. By the weighted Young inequality,
\begin{equation*}
A^iB^j
\leq
\frac{i}{i+j}A^{i+j}
+\frac{j}{i+j}B^{i+j},
\qquad A,B\geq0.
\end{equation*}
Using this inequality, symmetry, and \eqref{easybarU}, we obtain
\begin{equation*}
\begin{split}
&\lambda^{-\frac{(i+j)(N-2s)}{2}}
\int_\Omega\int_\Omega
\frac{
\bar U_{z,\lambda}^{\twostar-i}(x)|H_a(z,x)|^i
\bar U_{z,\lambda}^{\twostar-j}(y)|H_a(z,y)|^j}
{|x-y|^\mu}dxdy
\\
=&\lambda^{-\frac{(i+j)(N-2s)}{2}}
\int_\Omega\int_\Omega
\frac{
\bar U_{z,\lambda}^{\twostar}(x)
\bar U_{z,\lambda}^{\twostar}(y)}
{|x-y|^\mu}
\left(
\frac{|H_a(z,x)|}{\bar U_{z,\lambda}(x)}
\right)^i
\left(
\frac{|H_a(z,y)|}{\bar U_{z,\lambda}(y)}
\right)^jdxdy
\\
\lesssim&\lambda^{-\frac{(i+j)(N-2s)}{2}}
\int_\Omega\int_\Omega
\frac{
\bar U_{z,\lambda}^{\twostar-i-j}(x)
|H_a(z,x)|^{i+j}
\bar U_{z,\lambda}^{\twostar}(y)}
{|x-y|^\mu}dxdy
\\
\lesssim&\lambda^{-\frac{(i+j)(N-2s)}{2}}
\int_\Omega
\bar U_{z,\lambda}^{2_s^*-i-j}(x)
|H_a(z,x)|^{i+j}dx.
\end{split}
\end{equation*}
Since $K+1\leq i+j\leq2K<2_s^*$, Lemma \ref{lem:barU-Ha-power-asymptotics} shows that the last expression is $o\left(\lambda^{-2s}\right)$.

\end{proof}

\vspace{0.5cm}
\noindent \textbf{Data Availability.} Data sharing is not applicable to this article as no datasets were generated or analyzed during the current study.

\vspace{0.5cm}
\noindent \textbf{Conflict of interest.} The authors have no competing interests to declare that are relevant to the content of this article.

\end{document}